\documentclass[a4paper,11pt]{article}
\usepackage[utf8]{inputenc}
\usepackage[T1]{fontenc}
\usepackage[english]{babel}
\usepackage{multirow}
\usepackage{microtype}
\usepackage{graphicx}
\usepackage{subfigure}
\usepackage{booktabs} 
\usepackage{enumerate}
\usepackage{enumitem}
\setlist{itemsep=1pt}
\usepackage{lmodern}
\usepackage{amsthm}
\newtheorem{theorem}{Theorem}[section]
\newtheorem{lemma}[theorem]{{Lemma}}
\newtheorem{assumption}{Assumption}
\newtheorem{definition}{Definition}
 \newtheorem{proposition}[theorem]{Proposition}
 
 \newtheorem{corollary}[theorem]{{Corollary}}

\usepackage{caption}

\usepackage{lipsum}

\RequirePackage{hyperref}
\RequirePackage{nameref}
\hypersetup{colorlinks, linkcolor=blue, citecolor=blue, urlcolor=magenta, linktocpage, plainpages=false}

\RequirePackage{natbib}

\newcommand{\floor}[1]{\lfloor #1 \rfloor}
\newcommand{\bigfloor}[1]{\Bigl \lfloor #1 \Bigr \rfloor}

\usepackage{natbib}
\usepackage[top=1in, bottom=1in, left=1in, right=1in, dvips,letterpaper]{geometry}
\usepackage{hyperref}
\graphicspath{{./images/}}

\usepackage{algorithm}
\usepackage{algorithmic}
\usepackage{amsmath, amsfonts,amssymb}
\usepackage[final]{changes}

\newcommand{\N}{\mathbb{N}}
\newcommand{\R}{\mathbb{R}}
\newcommand{\E}{\mathbb{E}}

\title{On the Convergence of Step Decay Step-Size for Stochastic Optimization}

\author{Xiaoyu Wang \thanks{
	The Division of Decision and Control Systems, School of Electrical Engineering and Computer Science, KTH Royal Institute of Technology, SE-100 44 Stockholm, Sweden.}
	\and Sindri Magn{ú}sson \thanks{
		Department of Computer and System Science, Stockholm University, Stockholm, Sweden.
		{ Emails}: {\tt wang10@kth.se}, {\tt sindri.magnusson@dsv.su.se}, {\tt mikaelj@kth.se},} 
	\and Mikael Johansson\footnotemark[1]

}

\begin{document}


\maketitle

\begin{abstract}


 
 The convergence of stochastic gradient descent is highly dependent on the step-size, especially on non-convex problems such as neural network training.
Step decay step-size schedules (constant and then cut) are widely used in practice because of their excellent convergence and generalization qualities, but their theoretical properties are not yet well understood. We provide the 
convergence results for step decay in the non-convex regime, ensuring that the gradient norm vanishes at an $\mathcal{O}(\ln T/\sqrt{T})$ rate. We also provide the 
convergence guarantees for general (possibly non-smooth) convex problems, ensuring an $\mathcal{O}(\ln T/\sqrt{T})$ convergence rate. Finally, in the strongly convex case, we establish an $\mathcal{O}(\ln T/T)$ rate for smooth problems, which we also prove to be tight, and an $\mathcal{O}(\ln^2 T /T)$ rate without the smoothness assumption. We illustrate the practical efficiency of the step decay step-size in several large scale deep neural network training tasks.
 
 
\end{abstract}

\section{Introduction}
We focus on stochastic programming problems on the form
\begin{equation}\label{P1}
\min_{x\in \mathcal{X} }\,\, f(x):= \E_{\xi \sim \Xi }[ f(x; \xi)].
\end{equation} 
Here, $\xi$ is a random variable drawn from  some source distribution $\Xi$ over an arbitrary probability space and $\mathcal{X}$ is a closed, convex subset of $\R^d$. 
This problem is often encountered in machine learning applications, such as training of deep neural networks. Depending the specifics of the application, the function $f$ can either nonconvex, convex or strongly convex; it can be smooth or non-smooth; and it may also have additional structure that can be exploited.

Despite the many advances in the field of stochastic programming,  the stochastic gradient descent (SGD) method \citep{SGD-1951,SGD-complex} remains important and is arguably still the most popular method for solving \eqref{P1}. The SGD method updates the decision vector $x$ using the following recursion 
\begin{align}
x_{t+1} = \Pi_{\mathcal{X}}(x_t-\eta_t\hat{g}_t)
\end{align}
where $\hat{g}_t$ is an unbiased estimation of the gradient (or subgradient) and $\eta_t > 0$ is the step-size (learning rate). 

The step-size is a critical parameter which controls the rate (or speed) at which the model learns and guarantees that the SGD iterates converge to an optimizer of (\ref{P1}).  Setting the step-size too large will result in iterates which never converge; and setting it too small leads to slow convergence and may  even cause the iterates to get stuck at bad local minima. As long as the iterates do not diverge,  a large constant step-size promotes fast convergence but only to a large neighborhood of the optimal solution. To increase the accuracy, we have to decrease the step-size. 

The traditional approach is to decrease the step-size in every iteration, typically as 
%
$\eta_0/t$ or $\eta_0/\sqrt{t}$. Both these step-size schedules have been studied extensively and guarantee a non-asymptotic convergence of SGD~\citep{Moulines-Bach2011, Lacoste-Schmidt-Bach2012,Rakhlin-Shamir-Sridharan2011, hazan2014beyond, Shamir-Zhang2013, Gower-etal.2019}. However, from a practical perspective, these step-size policies often perform poorly, since they begin to decrease too early. \citet{Gower-etal.2019} have proposed to use a constant step-size for the first $4\kappa$ iterates (where $\kappa$ is the condition number) and then shift to a $1/t$ step-size to guarantee convergence for strongly convex functions. However, this strategy requires knowledge of the condition number and is not suitable for the nonconvex setting.

For non-convex problems, such as those which arise in training of deep neural networks, the most popular step-size policy in practice is the step decay step-size~\citep{imagenet, he2016deep, huang2017densely}. This step-size policy starts with a relatively large constant step-size and then cuts the step-size by a fixed number (called decay factor) at after a given number of epochs. Not only does this step-size result in a faster initial convergence, but it also guarantees that the SGD iterates eventually converge to an exact solution. In \citet{yuan2019stagewise}, the step decay step-size (which they call stagewise step decay) was shown to accelerate the convergence of SGD compared to polynomially decay step-sizes, such as $1/t$ and $1/\sqrt{t}$. \citet{ge2019step} prove significant improvements of the step decay step-size over any polynomial decay step-size for least squares problems.  In fact, the step decay step-size is the default choice in many deep learning libraries, such as TensorFlow~\citep{tensorflow} and PyTorch~\citep{pytorch}; both use a decay rate of $0.1$ and user-defined milestones when the step-size is decreased.  The milestones are often set in advance by experience. If we know some quantities or a certain conditions to characterize when the function is optimized, it will be ideal to decide when to drop the step-size. However, it is hard and time-consuming to get access to these quantities or conditions in practice.  

\citet{ge2019step} assume that the SGD algorithm will run for a fixed number $T$ of iterations. They then analyze a step decay step-size with decay rate $1/2$ applied every  $T/\log_2 T$ iterations and establish a near-optimal $\mathcal{O}(\log_2 T/T)$ convergence rate for least-squares problems. However, its non-asymptotic convergence for general strongly convex functions, convex functions or nonconvex functions is not analyzed. Motivated by this, in this paper, we focus on SGD with the step decay step-size which uses a general decay factor $\alpha$ ($\alpha > 1$)  rather than the fixed number $2$ \citep{ge2019step, yuan2019stagewise, davis2019stochastic,davis2019low}. This is more relevant in practice. 

\subsection{Main Contributions}
This work establishes \replaced{novel}{the} convergence \replaced{guarantees for}{behavior of} SGD with \added{the} step decay step-size \replaced{on}{under several cases including} smooth (nonconvex), general convex, and strongly convex \replaced{optimization problems}{cases}. More precisely, we make the following contributions:
\begin{itemize}
    \item We propose a non-uniformly probability rule $P_t \propto 1/\eta_t$ for \added{selecting} the output in the nonconvex and smooth setting. Based on this rule, 
    \begin{itemize}
        \item we establish a near-optimal $\mathcal{O}(\ln T/\sqrt{T})$ rate for SGD with the step decay step-size;
        \item we improve the results for exponential decay step-size \citep{li2020exponential};
        \item we remove the $\ln T$ factor \replaced{in}{of} the \added{best known} convergence rate for the classic $1/\sqrt{t}$ step-size.
    \end{itemize}
    \item For \added{the} general convex case,  we prove that the step decay step-size at the last iterate can achieve a \replaced{near}{nearly}-optimal convergence rate (up to \added{a} $\ln T$ factor).
    \item For  strongly convex problems, we establish the following error bounds \replaced{for}{of} the last iterate under step-decay: \deleted{for step decay step-size are}
    \begin{itemize}
        \item $\mathcal{O}(\ln T/T)$ for smooth problem, which \replaced{we also prove to be tight;}{is also proved to be tight; }
        \item $\mathcal{O}(\ln^2T/T)$ without the smoothness assumption.
    \end{itemize} 
\end{itemize}

\subsection{Related Work}
	For SGD, the best known bound for the expected error of the $T^{\rm th}$ iterate is of ${\mathcal O}(1/\sqrt{T})$ when the objective is convex and smooth with Lipschitz continuous gradient~\citep{SGD-complex,ghadimi2013stochastic}, and of ${\mathcal O}(1/T)$ when the objective is also  strongly convex~\citep{Moulines-Bach2011, Rakhlin-Shamir-Sridharan2011}. Without any further assumptions, these rates are known to be optimal.  If we restrict our attention to diminshing step-sizes, $\eta_t=\eta_0/t$, the best known error bound for strongly convex \added{and nonsmooth} problems is of ${\mathcal O}(\ln T/T)$~\citep{Shamir-Zhang2013}, which is also tight~\citep{harvey2018tight}. This rate can be improved to ${\mathcal O}(1/T)$ by averaging strategies~\citep{Rakhlin-Shamir-Sridharan2011,Lacoste-Schmidt-Bach2012,Shamir-Zhang2013} or a step decay step-size~\citep{hazan2014beyond}. For smooth  nonconvex functions,  \citet{ghadimi2013stochastic} established an  $\mathcal{O}(1/\sqrt{T})$ rate for SGD with constant step-size $\eta_t = \mathcal{O}(1/\sqrt{T})$. Recently, \citet{drori2020complexity} have proven that this error bound is tight up to a constant, unless additional assumptions are made.

	The step decay step-size was used for deterministic subgradient methods in \citet{goffin1977} and \citet{shor2012}.  Recently, it has been employed to improve the convergence rates under various conditions: local growth (convex) \citep{xu2016accelerated}, Polyak-L\'ojasiewicz (PL) condition~\citep{yuan2019stagewise}, sharp growth (nonconex)~\citep{davis2019stochastic}. \citet{davis2019low} also apply the step decay scheme to prove the high confidence bounds in stochastic convex optimization. Most of these references consider the proximal point algorithm \replaced{which introduce a quadratic term to the original problem.}{and weakly convex objective functions.} In contrast, we study the
	performance of step decay step-size for standard SGD. Compared to  \citet{hazan2014beyond} and \citet{yuan2019stagewise}, where the inner-loop size $S$ is growing exponentially, we use a constant value of $S$, which is known to work better in practice. In the extreme case when $S=1$, step-decay reduces to the exponentially decaying step-size which ~\citet{li2020exponential} have recently studied under the PL condition and a general smoothness assumption.

	A number of adaptive step-size selection strategies have been proposed for SGD (\emph{e.g.}, \citep{AdaGrad, RMSProp, Adam, AdamW, polyak2020}), some of which result in step-decay policies~\citep{auto-line-search, lang2019using, zhang2020statistical}. For example, \citet{lang2019using} develop a statistical procedure to automatically determine when the SGD iterates with a constant step-size no longer make progress, and then halve the step-size. Empirically, this automatic scheme is competitive with the best hand-tuned step-decay schedules, but no formal guarantees for this observed behaviour are given. 
	
The remaining part in this paper are organized as follows. Notation and basic definitions are introduced in Section~\ref{sec:pre}. In Section \ref{sec:nonconvex}, we analyze the convergence rates of step decay step-size on nonconvex case and propose a novel non-uniform sampling rule for the algorithm output. The convergence for general convex and strongly convex functions are investigated in Sections~\ref{sec:convex} and~ \ref{sec:sc}, respectively. Numerical results of our algorithms are presented and discussed in Section \ref{sec:numerical}. Finally, conclusions are made in Section~\ref{sec:conclusion}.

\section{Preliminaries}\label{sec:pre}

In this part, we will give some definitions and notations used throughout the paper.
\begin{definition}\label{assump:gradient}
\begin{enumerate}[label=\textbf{(\arabic*)}]
    \item[]
    \item \label{assump:bounded-variance} The stochastic gradient oracle $\tilde{\mathcal{O}}$ is variance-bounded if for any input vector $\hat{x}$, the oracle $\tilde{\mathcal{O}}$ returns a random vector $\hat{g}$ such that $
\E[\left\|\hat{g} - \E[\hat{g}]\right\|^2 ] \leq V^2$.    \item\label{assump:bounded-gradient} The stochastic gradient oracle $\tilde{\mathcal{O}}$ is bounded if for any input vector $\hat{x}$, the oracle $\tilde{\mathcal{O}}$ returns a random vector $\hat{g}$ such that $
\E[\left\|\hat{g} \right\|^2 ] \leq G^2$
for some fixed $G > 0$.
\end{enumerate}
\end{definition}

\begin{definition}[$L$-smooth]
The function $f$ is differentiable and $L$-smooth on $\mathcal{X}$ if there exists a constant $L > 0$ such that $\left\|\nabla f(x) - \nabla f(y) \right\| \leq L\left\|x-y\right\|$. 
It also implies that $f(y) \leq f(x) + \left\langle \nabla f(x), y-x \right\rangle + \frac{L}{2}\left\|x-y\right\|^2$ for any $x, y \in \mathcal{X}$.

Especially, when $f$ is non-differentiable on $\mathcal{X}$, we define $f$ is $L$-smooth with respect to $x^{\ast}$ if $
f(x) - f(x^{\ast}) \leq \frac{L}{2}\left\|x-x^{\ast}\right\|^2, \forall x \in \mathcal{X},$
with $L > 0$.
\end{definition}
The smooth property with respect to $x^{\ast}$ has been considered by \citet{Rakhlin-Shamir-Sridharan2011}.

\begin{definition}[$\mu$-strongly convex]
The function $f$ is $\mu$-strongly convex on $\mathcal{X} \subseteq \R^d$ if $
f(y) \geq f(x) + \left\langle g, y-x\right\rangle + \frac{\mu}{2}\left\|y-x\right\|^2, \, \forall  x, y\in \mathcal{X}, g \in \partial f(x),$
with $\mu > 0$.
\end{definition}

\begin{definition}[Convex]
The function $f$ is convex on $\mathcal{X}$ if $f(y) \geq f(x) + \left\langle g, y-x \right\rangle $  for any $g \in \partial f(x)$ and $x,y\in \mathcal{X}$.
\end{definition}


Throughout the paper, we assume the objective function $f$ is bounded below on $\mathcal{X}$ and let $f^{\ast}$ denote its infimum. If $f$ is strongly convex, let  $x^{\ast}$ be the unique minimum point of $f$ and $f^{\ast} = f(x^{\ast})$.


{\bf Notations:} Let $[n]$ denote the set of $\left\lbrace 1,2,\cdots, n\right\rbrace$ and $\left\|\cdot\right\| := \left\|\cdot\right\|_2$ without specific mention. We use $\lfloor r \rfloor$ and $\lceil r \rceil$ to denote the nearest integer to the real number $r$ from below and above. For simplicity, we assume that $S$, $N$, $\log_{\alpha} T/2$, $\log_{\alpha} T$ and $T/\log_{\alpha} T$ are all integers. 

\section{Convergence \deleted{Guarantees} for Nonconvex Problems}
\label{sec:nonconvex}

In this section, we provide \replaced{the first}{a} convergence bounds for SGD with step-decay step-sizes on non-convex problems. We also show that \replaced{our}{the} technical approach \deleted{main analytical idea} \deleted{we use} can be used to improve the best known convergence bounds for both \added{i)} standard $1/\sqrt{t}$ step-sizes and \added{ii)} exponential decay step-sizes.

Before \replaced{proceeding, it is useful to illustrate}{providing the results we overview} the main theoretical novelty that allows us to derive the result. Typically when we study the convergence of SGD for non-convex problems, we analyse \replaced{a random iterate drawn from $\{x_t\}_{t=1}^T$ with some probability $P_t$}{the \added{average of the} output\replaced{s at every iteration where we select iteration $t$ with some probability $P_t$}{which is selected from $\left\lbrace x_t\right\rbrace_{t=1}^{T}$ according to a probability distribution}} \citep{ghadimi2013stochastic,ghadimi2016mini,li2020exponential}. For example,  \citet{ghadimi2013stochastic} \added{provide the following result.} 
\begin{proposition}\label{pro:1sqrt}
Suppose that $f$ is $L$-smooth on $\R^d$ and the stochastic gradient oracle is variance-bounded \added{by $V^2$}. If the step-size $\eta_t = \eta_0/\sqrt{t} \leq 1/L$, then
\begin{align}\label{eq:CB_GL}
\E[\left\|\nabla f(\hat{x}_T)\right\|^2] \leq  \frac{f(x_1)-f^{\ast}}{\eta_0(\sqrt{T}-1)} + \frac{LV^2\eta_0(\ln T {+}1 )}{2(\sqrt{T}-1)},
\end{align}
where $\hat{x}_T$ is randomly chosen from $\left\lbrace x_t\right\rbrace_{t=1}^{T}$ with probability $P_t\propto\eta_t$. \footnote{\added{ In \citet{ghadimi2013stochastic}, $P_t \propto (2\eta_t - L\eta_t^2)$. If $\eta_t$ is far smaller than $1/L$, we have $(2\eta_t - L\eta_t^2) \approx 2\eta_t$. For simplicity, we rewrite the probability as $P_t \propto \eta_t$ to show the results.}}
\end{proposition}
\added{Since $\eta_t$ is decreasing,  $P_t\propto\eta_t$ means that initial iterates are given higher weights in the average in Equation~\eqref{eq:CB_GL} than the final iterates. This contradicts the intuition that the gradient norm decreases as the algorithm progresses. Ideally, we should do the opposite, \emph{i.e}. put high weights on the final iterates and low weights on the initial iterates.  
This is exactly what we do to obtain convergence bounds for step decay step-size: we use the probability $P_t \propto 1/\eta_t$ instead of $P_t \propto \eta_t$.
This is especially important when the step-sizes decrease exponentially fast, like in step decay and exponential decay step-sizes.
For example, suppose that  $\eta_t\propto 0.9^{-t}$ and $T=100$.  Then with $P_t \propto \eta_t$ we pick the output from the first 10 iteration with $65\%$ probability. On the other hand, with $P_t \propto 1/\eta_t$ we pick the output from the last 10 iterations with $65\%$ probability. 
We illustrate this better in Figure~\ref{fig:probability}.}


In the following three subsections we \replaced{perform such an analysis to}{use this idea to, respectively,} 1) provide convergence bounds for step decay step-size (where no bounds existed before) 2) provide improved convergence rate results for exponential decay step-size, and 3) improve the convergence bound for $1/\sqrt{t}$ step-size.

\subsection{Convergence Rates under the Step Decay Step-size}



The step decay step-size (termed Step-Decay) is widely used for training of deep neural networks. This policy decreases the step-size by a constant factor $\alpha$ at every $S$ iterations; see Algorithm~1. Despite its widespread use, we are unaware of any convergence results for this algorithm in the non-convex regime. Next, we will establish the first such guarantees for an iterate of step decay SGD drawn by the method described above.

 In practical deep neural network training, $S$ is typically a hyper-parameter selected by experience. To make theoretical statements we must consider a particular choice of $S$. In our theoretical and numerical results we focus on $S=2T/\log_{\alpha}(T)$.  Note that we can use $\alpha>1$ to obtain the  desired $S$. We have the following main result.

\begin{algorithm}[tb]
	\caption{SGD with step decay step-size on nonconvex case}\label{alg:nonconvex}
\begin{algorithmic}
   \STATE {\bfseries Input:} initial point $x_1^1$, initial step-size $\eta_0$, decay factor $\alpha > 1$, the number of iterations $T$, inner-loop size $S$, outer-loop size $N=T/S$
		\FOR{ $t =  1: N$ }
	    \STATE $\eta_t = \eta_0/\alpha^{t-1}$
		\FOR{ $i = 1: S$}
		\STATE Query a stochastic gradient oracle at $x_i^t$ to get a vector $\hat{g}_i^t$ such that $\E[\hat{g}_i^t] = \nabla f(x_i^t)$ \\\vspace{0.01in}
		\STATE $x_{i+1}^t =x_i^t - \eta_t  \hat{g}_i^t$
		\ENDFOR
		\STATE $x_1^{t+1} = x_{S+1}^{t}$
		\ENDFOR
		\STATE {\bfseries Return:} $\hat{x}_T$ is randomly chosen from all the previous iterations $\left\lbrace x_i^t\right\rbrace$ with probability $P_i^t = \frac{1/\eta_t}{S\sum_{t=1}^{N}1/\eta_t}$ where $i\in [S]$ and $t \in [N]$


\end{algorithmic}
\end{algorithm}

\begin{theorem}\label{thm:nonconvex:1}
Suppose that the \added{non-convex} objective function $f$ is $L$-smooth  on $\R^d$ and upper bounded by $f_{\max}$ \footnote{The function $f$ is upper bounded by $f_{\max}$ if $f(x) \leq f_{\max}$ for any $x \in \R^d$. } and the stochastic gradient oracle is variance-bounded \added{by $V^2$}.
If \added{we run Algorithm 1 with $T > 1$, $S=2T/\log_{\alpha}(T)$,} $\eta_0 \leq 1/L$\replaced{, and}{ and given any $T \geq 1$, running Algorithm \ref{alg:nonconvex} with any initial point} $x_1^1 \in \R^d$ \replaced{then}
 {, we have}
\begin{align*}
\E[\left\| \nabla f(\hat{x}_T)\right\|^2]  \leq A \frac{f_{\max}}{ \eta_0}\cdot \frac{\ln T}{\sqrt{T}-1} + B \frac{LV^2\eta_0}{\sqrt{T}-1},
\end{align*}
where $A=(\alpha-1)/(\alpha^2\ln \alpha )$ and $B=\alpha-1$.
\end{theorem}

The theorem establishes the first convergence guarantee \replaced{
for step decay step-size on non-convex problems. It ensures a ${\mathcal O}(\ln T/\sqrt{T})$ convergence rate towards a stationary point, which is}
{of step decay towards a stationary point for non-convex problems. Moreover, the theorem ensures a $\mathcal{O}(\ln T/\sqrt{T})$ convergence rate towards the solution. 
This upper bound} is comparable to the results for $\eta_t=\mathcal{O}(1/\sqrt{t})$ step-size in~\citet{ghadimi2013stochastic} or Proposition~\ref{pro:1sqrt}. However, as illustrated in our experiments in Section~\ref{sec:numerical}, step decay step-size converges faster in practice and tends to find stationary points that generalize better. 

In the extreme case when $\alpha=\sqrt{T}$, the step decay policy produces a constant step-size. In this case, Theorem~\ref{thm:nonconvex:1} matches the existing convergence bounds for SGD in~\citet{ghadimi2013stochastic}. Moreover, in the deterministic case when $V = 0$, Theorem~\ref{thm:nonconvex:1} yields the standard $\mathcal{O}(1/T)$ convergence rate result for deterministic
gradient descent \citep{nesterov2003introductory,cartis2010complexity,ghadimi2016mini}.

 A clear advantage of Theorem~\ref{thm:nonconvex:1} over, e.g., Proposition~\ref{pro:1sqrt}, is that it possible to reduce the effect of the noise $V^2$ in the convergence bound by proper parameter tuning.  In particular, the following result is easily derived from Theorem~\ref{thm:nonconvex:1} (proved in the supplementary material). 
\begin{corollary}\label{cor:nonconvex}
  If we set $\alpha=1+1/V^2$, then under the assumptions of Theorem~\ref{thm:nonconvex:1} we have
\begin{align*}
\E[\left\| \nabla f(\hat{x}_T)\right\|^2]  \leq  \frac{f_{\max}}{ \eta_0}\cdot \frac{\ln T}{\sqrt{T}-1} + \frac{L\eta_0}{\sqrt{T}-1},
\end{align*}  
\end{corollary}
The corollary shows us that if we choose $\alpha=1+1/V^2$ in step decay, then the convergence bound will be unaffected by the noise variance $V^2$. This is in contrast to Proposition~\ref{pro:1sqrt} and $1/\sqrt{t}$ step-sizes, where an increased $V^2$ results in a worse upper bound. However, as the proof of Corollary~\ref{cor:nonconvex} reveals, the bound is only tight in the the regime where both noise variance and  iteration counts are large. 

\subsection{A Special Case: Exponentially Decaying Step-Size}
\label{sec:exp}

An interesting special case occurs when we set $S=1$ in the step decay step-size. In this case, the step-size will decay exponentially, $\eta_t=\eta_0/\alpha^t$. The first convergence rate results for such step sizes, which we will call Exp-Decay, have only recently appeared in the preprint~\citet{li2020exponential}; we compare our work to those results below.

 Clearly, for $S=1$, the exponential decay factor $\alpha>1$ cannot be chosen arbitrarily: if $\alpha$ is too large, then $\eta_t$ will vanish in only a few iterations. 
 To avoid this, we choose a similar form of the decaying factor as~\citet{li2020exponential} and set
 $\alpha= (\beta/T)^{-1/T}$, where $\beta\in[1,T)$.
 The intuition is that at the final iteration $T$ the step-size is on the order of $1/\sqrt{T}$ \added{(setting $\beta=\mathcal{O}(\sqrt{T})$)}, and thus does not vanish over the $T$ iterations. 
 We are now ready to prove the algorithm's convergence. 

\begin{theorem}\label{thm:exp:nonconvex}
Suppose that the non-convex objective function \replaced{$f$ is $L$-smooth and upper bounded by $f_{\max}$, and that the}{satisfies Assumption \ref{assump:nonconvex}. The} stochastic gradient oracle is variance-bounded \added{by $V^2$}.
If we run Algorithm 1 with $T > 1$, $S=1$, $\eta_0 \leq 1/L$,   $x_1^1 \in \R^d$, and $\alpha= (\beta/T)^{-1/T}$, then 
 %
 %
 %
\begin{align*}
\E[\left\| \nabla f(\hat{x}_T)\right\|^2] \leq \frac{\eta_0\ln(\frac{T}{\beta})}{(\frac{T}{\beta}-1)T}\left[\frac{2f_{\max}}{\eta_0^2}\left(\frac{T}{\beta}\right)^2 + LV^2T\right],
\end{align*}
where $\hat{x}_T$ is randomly drawn from  $\left\lbrace x_t \right\rbrace_{t=1}^{T}$ with probability $P_t = \frac{1/\eta_t}{\sum_{t=1}^{T}1/\eta_t}$.
In particular, $\beta = \sqrt{T}$ yields
\begin{align}\label{equ:exp:1}
\E[\left\| \nabla f(\hat{x}_T)\right\|^2] \leq \left(\frac{f_{\max}}{\eta_0} + \frac{LV^2\eta_0}{2}\right)\cdot \frac{\ln T}{\sqrt{T}-1}.
\end{align}

\end{theorem}

 Theorem~\ref{thm:exp:nonconvex} establishes the convergence of Exp-Decay to a stationary point. With $\beta = \sqrt{T}$, Exp-Decay yields the rate of $\mathcal{O}(\ln T/\sqrt{T})$, which is comparable to the results in \citet{li2020exponential}. However,  in~\citet{li2020exponential}, the $\mathcal{O}(\ln T/\sqrt{T})$-rate is obtained under the (arguably impractical) assumption that the initial step-size $\eta_0$ is bounded by $\mathcal{O}(1/\sqrt{T})$. This means that the initial step-size $\eta_0$ will be small if we plan on running the algorithm for a large number of iterations $T$. This is in conflict to the original motivation for using exponentially decaying step-sizes, namely to allow for large initial step-sizes which decrease as the algorithm progresses. On the other hand, our results in Theorem~\ref{thm:exp:nonconvex} only require that the initial step-size is bounded by $\eta_0\leq 1/L$, which matches the largest fixed step-sizes that ensure convergence.  

 Another advantage of Theorem~\ref{thm:exp:nonconvex} is that it uses the output probability $P_t\propto 1/\eta_t$, instead of $P_t\propto \eta_t$ as in~\citet{li2020exponential}. 
 As discussed in the beginning of the section (and in Figure~\ref{fig:probability}), $P_t\propto 1/\eta_t$ means that the output is much more likely to be chosen during the final iterates, whereas with $P_t\propto \eta_t$ the output is more likely to come from the initial iterates. 
 Therefore, Theorem~\ref{thm:exp:nonconvex} better reflects the actual convergence of the algorithm, since in practice it is typically the final iterate that is used as the trained model. We illustrate this better in our experiments in Section~\ref{Sec:Experiments}.
 \begin{figure}[ht]
\vskip 0.1in
\begin{center}
\centerline{\includegraphics[width=0.5\columnwidth]{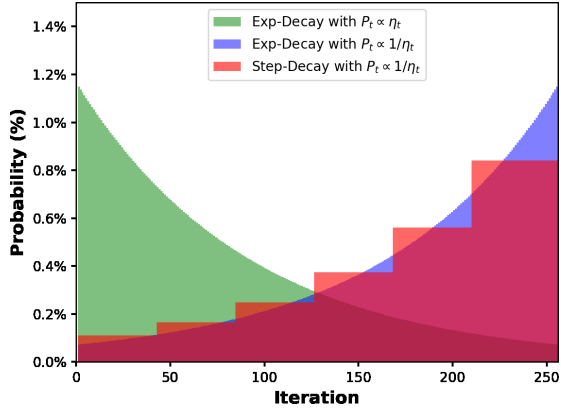}}
\caption{The probabilistic distribution of the output}
\label{fig:probability}
\end{center}
\vskip -0.1in
\end{figure}

\subsection{Improved Convergence for the $1/\sqrt{t}$ Step-Size}
\label{sec:sqrt}

\replaced{The next theorem demonstrates how the idea of using the output distribution $P_t\propto 1/\eta_t$ instead of $P_t\propto \eta_t$ can improve standard convergence bounds for $1/\sqrt{t}$ step-sizes. }{In this subsection, we employ the idea that the probability distribution $P_t \propto 1/\eta_t$ to the $1/\sqrt{t}$ diminishing step-size.}

\begin{theorem}\label{thm:sqrt}
Suppose that the objective function is $L$-smooth and upper bounded by $f_{\max}$,  and that the stochastic gradient oracle is variance-bounded by $V^2$. 
 If $\eta_t = \eta_0/\sqrt{t} \leq 1/L$, then 
\begin{align*}
\E[\left\|\nabla f(\hat{x}_T) \right\|^2] \leq \left(\frac{3f_{\max}}{\eta_0} + \frac{3LV^2\eta_0}{2}\right)\cdot \frac{1}{\sqrt{T}}
\end{align*}
where $\hat{x}_T$ is randomly drawn from the sequence $\left\lbrace x_t\right\rbrace_{t=1}^{T}$ with probabilities $P_t = \frac{1/\eta_t}{\sum_{t=1}^{T}1/\eta_t}$.
\end{theorem}
Theorem~\ref{thm:sqrt} establishes an $\mathcal{O}(1/\sqrt{T})$ convergence rate of SGD with $1/\sqrt{t}$ step-sizes. This improves the best known convergence bounds $\mathcal{O}(\ln T/\sqrt{T})$ for the algorithm, see Proposition~\ref{pro:1sqrt} or \citet{ghadimi2013stochastic}.

\section{Convergence for General Convex Problems}
\label{sec:convex}

We now establish the first convergence rate results for the step decay step-size in the general convex setting. More specifically, we consider a possibly non-differentiable convex objective function on a closed and convex constraint set $\mathcal{X} \subseteq \R^d$. For this problem class, we analyze the projected SGD with step decay step-size detailed in Algorithm ~\ref{alg:convex}. 
\begin{algorithm}[ht]
	\caption{Projected SGD with step decay step-size on convex case}\label{alg:convex}
	\begin{algorithmic}		
		\STATE {\bfseries Input:} initial point $x_1^1$, initial step-size $\eta_0$, decay factor $\alpha > 1$, the number of iteration $T$, inner-loop size $S$ and outer-loop size $N = T/S$
		\FOR{ $t =  1: N$ }
		\STATE $\eta_t = \eta_0/\alpha^{t-1}$
		\FOR{ $i = 1: S$}
		\STATE Query a stochastic gradient oracle at $x_i^t$ to get a random vector $\hat{g}_i^t$ such that $\E[\hat{g}_i^t] = g_i^t$ where $g_i^t \in \partial f(x_i^t)$
		\STATE $x_{i+1}^t =\Pi_{\mathcal{X}}(x_i^t - \eta_t  \hat{g}_i^t)$, where $\Pi_{\mathcal{X}}$ is the projection operator on $\mathcal{X}$
		\ENDFOR
		\STATE $x_1^{t+1} = x_{S+1}^{t}$
		\ENDFOR
 		\STATE {\bf Return:} $x_{S+1}^{N}$

\end{algorithmic}
\end{algorithm}

We have the following convergence guarantee:
\begin{theorem}\label{thm:convex:lastloop}
 Suppose that the objective function $f$ is convex on $\mathcal{X}$ and \added{$\sup_{x,y\in \mathcal{X}}\left\|x-y\right\|^2 \leq D^2$}. The stochastic gradient oracle is bounded \added{by $G^2$}. \added{If we run Algorithm \ref{alg:convex} with $T > 1$, $S = 2T/\log_{\alpha} T$, $x_1^1 \in \mathcal{X}$}, then we have 
\begin{align*}
\frac{1}{S}\sum_{i=1}^{S} \E[f(x_i^{N})] - f^{\ast} \leq A_2 \frac{\ln T}{\sqrt{T}} + \frac{B_2}{\sqrt{T}},
\end{align*}
where $A_2= D^2/(4\eta_0\alpha\ln \alpha)$ and $B_2= G^2\eta_0\alpha/2$. 
Moreover, we have the following bound on the final iterate:
\begin{align*}
 \E[f(x_{S}^{N})] - f^{\ast} & \leq  \left(A_2 + B_2 \right)\frac{\ln T }{\sqrt{T}} 
 + \frac{(2+\ln 2)B_2}{\sqrt{T}}.
\end{align*}
\end{theorem}

Theorem \ref{thm:convex:lastloop} establishes an $\mathcal{O}(\ln T/\sqrt{T})$ convergence rate for both the the average objective function value and the objective function value at the final iterate. 
This convergence rate is comparable to the results obtained for other diminishing step-sizes such as $\eta_t=\eta_0/\sqrt{t}$ under the same assumptions~\citep{Shamir-Zhang2013}.

\section{Convergence for Strongly Convex Problems}\label{sec:sc}

We will now investigate the convergence of the step decay step-sizes in the strongly convex case. The algorithm is the same as for general convex problems (Algorithm \ref{alg:convex}). 



\subsection{Strongly Convex and $L$-Smooth Functions}
We first consider the case when the objective function $f$ is strongly convex and $L$-smooth (with respect to $x^{\ast}$). 

\begin{theorem}\label{thm:sc}
Assume that the objective function $f$ \replaced{is $\mu$-strongly convex on $\mathcal{X}$ }{satisfies Assumptions \ref{assump:strongly-case}\ref{assump:strongly-convex}} and the stochastic gradient oracle is bounded by $G^2$. \replaced{If we run Algorithm~\ref{alg:convex} with $T > 1$, $S = T/\log_{\alpha} T$, $x_1^1 \in \mathcal{X}$,  and $\eta_0 < 1/(2\mu) $, then we have}{Algorithm \ref{alg:stronglyconvx} is initialized by any point of $\mathcal{X}$, then if $\eta_0 < 1/(2\mu) $, we have}
\begin{align*}
  \E \left\|x_{S+1}^{N}\! - x^{\ast}\right\|^2   \leq &
\frac{R}{\exp\left(A_3 \frac{T-1}{\ln T}\right)}  +
  \frac{\alpha G^2 \exp\left(\frac{A_3}{\ln T}\right)  }{2\mu A_3}\frac{\ln T}{T}
%
\end{align*}
where $A_3=2\mu \eta_0 \alpha \ln \alpha /(\alpha-1)$ and $R=\left\|x_1^1 - x^{\ast} \right\|^2 $. Further, if the objective function $f$ is $L$-smooth with respect to $x^{\ast}$ then we can bound the objective function as follows
\begin{align*}
\E[f(x_{S+1}^{N}) - f(x^{\ast})] 
& \leq  \frac{L}{2} \E \left\|x_{S+1}^{N} - x^{\ast}\right\|^2 .
\end{align*}
\end{theorem}

The theorem provides an $\mathcal{O}(\ln T/T)$ theoretical guarantee for the last iterate under step decay step-size.\footnote{Note that $\exp(A_3/\ln T)\leq \exp(A_3/\ln 2)$ for all $T\geq 2$ and $\exp(A_3/\ln T)$ converges to $1$ as $T$ goes to infinity. } This bound matches the convergence rate of \citet{ge2019step} for strongly convex least squares \replaced{problems}{regression}. \replaced{Therefore, Theorem~\ref{thm:sc} can be considered a generalization of the result in \citet{ge2019step} to general smooth and strongly convex problems.}{
We extend the result of \citet{ge2019step} for step decay step-size to the more general case under strongly convexity and smoothness assumptions.}
\deleted{Note that the condition for initial step-size is $\eta_0 \leq 1/(2\mu)$ which is similar to the polynomial step-size for obtaining optimal rates \citep{Moulines-Bach2011}. This condition is looser than that of \citep{li2020exponential} for exponential decay step-size.
Moreover, the convergence rate derived in Theorem \ref{thm:sc} is better than that of \citet{li2020exponential} which is  $\mathcal{O}(\ln ^2T/T)$(see theorem 1).}

\replaced{Under some step-sizes, e.g. $\eta_t=1/t$ , it is  possible to get an $\mathcal{O}(1/T)$ convergence rate for SGD for smooth and strongly convex problems~\cite{Moulines-Bach2011,Rakhlin-Shamir-Sridharan2011, hazan2014beyond}. However, our next result shows that the rate in Theorem~\ref{thm:sc} is tight for Algorithm \ref{alg:convex}.}{
Next, we give one dimension function, which is 1-strongly convex and 1-smooth on a bounded region $\mathcal{X}$, to show the high probability lower bound of Algorithm \ref{alg:stronglyconvx} in strongly convex case. 
}

\begin{theorem}\label{thm:lower-bound}
 Consider Algorithm~\ref{alg:convex} with $S = T/\log_{\alpha} T$ and $x_1^1=0$. For any $T \in \N^{+}$ and $\delta\in(0,1)$, there exists a  function $\tilde{f}_T :\mathcal{X}\rightarrow \R$, where $\mathcal{X}=[-4,4]$, that is both 1-strongly convex and 1-smooth such that
 \begin{align*}
\tilde{f}_T(x_{S+1}^{N}) - \tilde{f}_T(x^{\ast}) \geq  \frac{ \ln(1/\delta)}{9\exp(2)\ln\alpha}\cdot \frac{\ln T}{T}
\end{align*}
 with probability at least $\delta$, where $x^{\ast} = \min_{x \in \mathcal{X}} \tilde{f}_T(x)$.

\end{theorem}

The high probability lower bound $\mathcal{O}(\ln T/T)$ matches the upper bound demonstrated in Theorem \ref{thm:sc}. This \replaced{suggests}{also implies} that the error bound in Theorem \ref{thm:sc} is actually tight for \deleted{step decay schedule in} Algorithm \ref{alg:convex} \added{in the smooth and strongly-convex case}.

\subsection{Strongly Convex and Non-Smooth Functions}


For general, not necessarily smooth, strongly convex functions  we have the following result.
\begin{theorem}\label{sec:thm:last-iterate}
Suppose that the objective function $f$ \replaced{is $\mu$-strongly convex on $\mathcal{X}$}{satisfies Assumptions \ref{assump:strongly-case}\ref{assump:strongly-convex}} and the stochastic gradient oracle is bounded by $G^2$. \replaced{If we run Algorithm~\ref{alg:convex} with $T > 1$, $S = T/\log_{\alpha} T$, $x_1^1 \in \mathcal{X}$,  and $\eta_0 < 1/(2\mu) $, then we have}{If $\eta_0 < 1/(2\mu)$, we have}
\begin{align*}
 \E[f(x_{S}^{N}) - f(x^{\ast})] \leq&   \frac{ R \ln T}{A_4 \exp\left(B_4 \frac{T-\alpha}{\ln T}\right)} +C_4\frac{\ln T+2}{T} \\
&+ D_4 \exp\left( \frac{E_4}{\ln T}\right) \frac{\ln^2 T}{T}
\end{align*}
where $A_4= \eta_0 \alpha \ln \alpha$, $B_4=2\mu A_4 /(\alpha-1)$, $C_4=G^2\eta_0\alpha$, $D_4=G^2/(2\mu\ln \alpha B_4)$, $E_4=\alpha B_4$.

\end{theorem}
 The theorem shows that even without the smoothness assumption, we can still ensure convergence at the rate $\mathcal{O}(\ln^2 T/T)$. We can improve the rate to $\mathcal{O}(\ln T/ T)$ with an averaging technique, as illustrated next.
%
%
%
\begin{theorem}\label{thm:nonsmooth:average}
Under the assumptions of Theorem~\ref{sec:thm:last-iterate}, we have 
\begin{align*}
     \E[f(\hat{x}_{T}) -f(x^{\ast})] \leq \frac{A_5 R}{\exp\left( B_4\frac{T}{\ln T}-1 \right)}
     +C_5 \frac{\ln T}{T}
\end{align*}
where 
$\hat{x}_T = { \sum_{t=t^{\ast}}^{N}\eta_t \sum_{i=1}^{S} x_i^t}/({{S}\sum_{t=t^{\ast}}^{{N}}\eta_t}),$
$t^{\ast}:= \max \lbrace 0, \floor{ \log_{\alpha}(\eta_0 \alpha A_5 T/\log_{\alpha} T)}\rbrace$,
$A_5=2\mu\alpha/(\alpha-1)$, $C_5=\alpha(2+1/(\alpha^2-1))G^2/(2\mu \ln \alpha)$, and $B_4$ is as defined in Theorem~\ref{sec:thm:last-iterate}.
\end{theorem}

\section{Numerical Experiments} \label{Sec:Experiments}

\label{sec:numerical}
In this section, we evaluate the practical performance of step decay step-size \deleted{(termed Step-Decay)} and compare it against 
the following popular step-size policies: 1) constant step-size, $\eta_t=\eta_0$; 2) $1/t$ step-size, $\eta_t=\eta_0/(1+a_0 t)$; 3) $1/\sqrt{t}$ step-size $\eta_t=\eta_0/(1+a_0\sqrt{t})$; 
4) Exp-Decay~\cite{li2020exponential}, $\eta_t= \eta_0 /\alpha^t$ with $\alpha=(\beta/T)^{-1/T}$ for $\beta\geq 1$.
In each experiment, we perform a grid search to select the best values for the free parameters $\eta_0, a_0, \beta$ as well as for the step decay step-size parameter $\alpha$. More details about the relationship between the different step-size policies are given in the supplementary material.

\subsection{Experiments on MNIST with Neural Networks}
\label{experiment:mnist}
First, we consider the classification task on MNIST database of handwritten digits\footnote{\url{http://yann.lecun.com/exdb/mnist/}} \replaced{using}{which consists a training set of 60,000 examples and a testing set of 10,000 examples, on} a fully-connected 2-layer neural network with 100 hidden nodes (784-100-10). 

The last iterate $x_T$ is \replaced{usually}{often} chosen as the output of each algorithm in practice. The theoretical output $\hat{x}_T$ (in Theorem \ref{thm:nonconvex:1}) is drawn from all the previous iterates $\left\lbrace x_i^t \right\rbrace$ with probability $P_t \propto 1/\eta_t$. \replaced{To get an insight into}{To figure out} the relationship between the last iterate and theoretical output in Theorem \ref{thm:nonconvex:1}, we randomly choose 6000 iterates (10\% of the total iterates) with probability $P_t \propto 1/\eta_t$, record their exact training loss and testing loss, and calculate the probabilities (shown in Figure \ref{fig:stepdecay:prob}). We can see that the theoretical output can reach the results of the last iterate with high probability, no matter training loss or testing loss. 

\begin{figure}[ht!]
\vskip 0.1in
\begin{center}
\centerline{
\subfigure[Training loss]{\includegraphics[width=0.35\textwidth,height=1.8in]{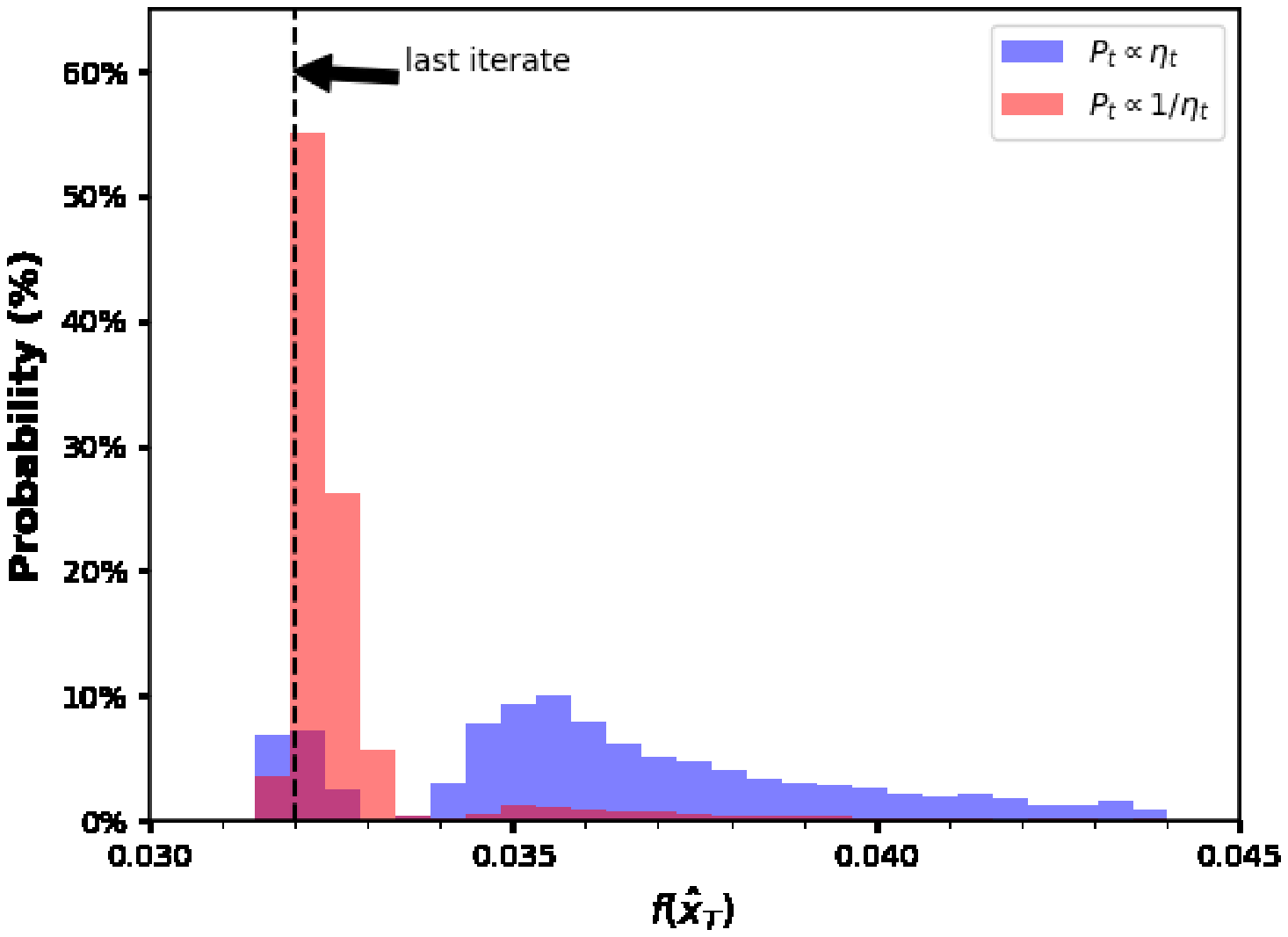}}
\subfigure[Testing loss]{\includegraphics[width=0.35\textwidth,height=1.8in]{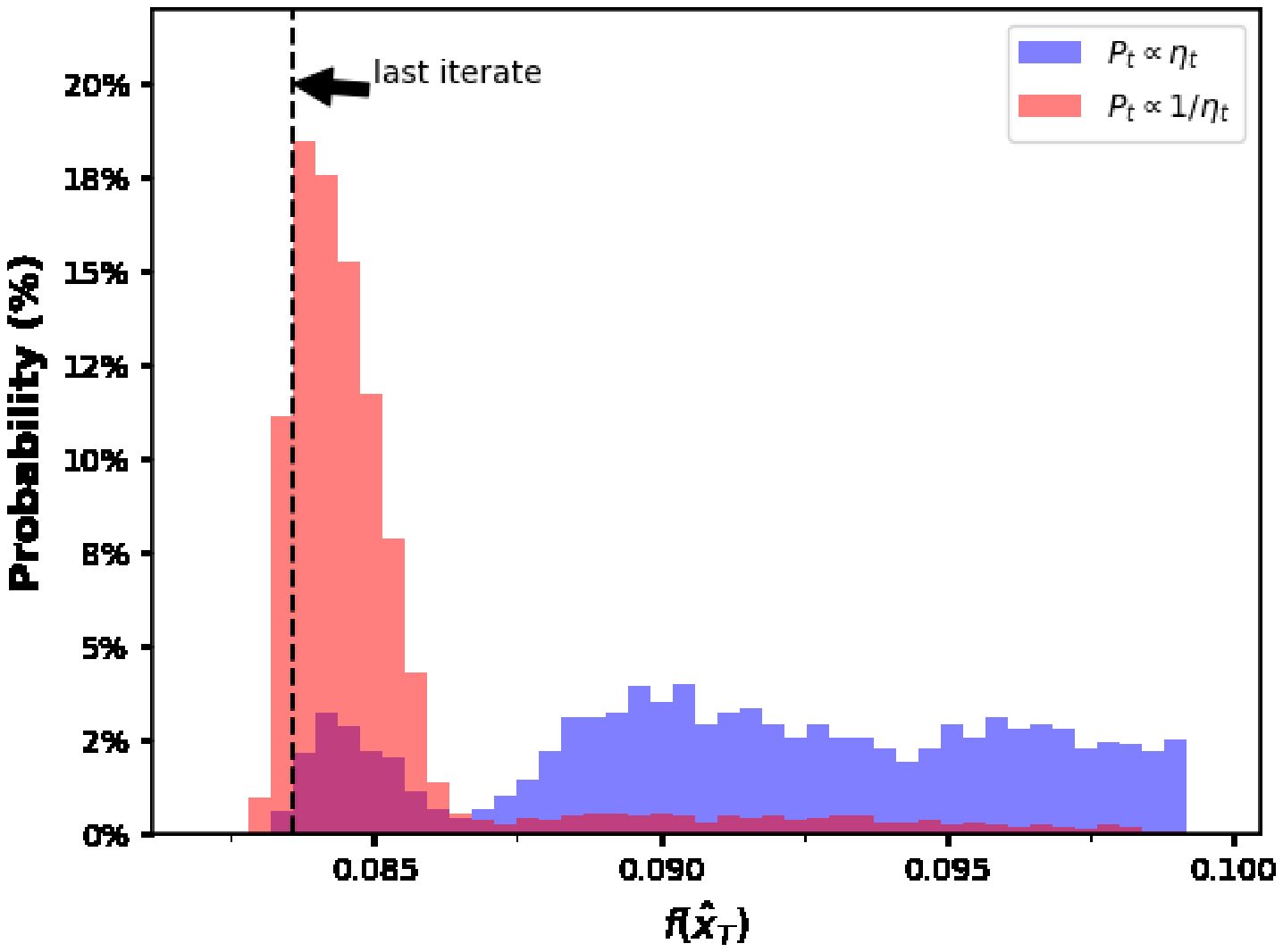} }
} 
\caption{The comparison of the two probabilistic outputs.}
\label{fig:stepdecay:prob}  
\end{center}
\vskip -0.1in
\end{figure}

In the same way, in order to show the advantages of probability $P_t \propto 1/\eta_t$ over $P_t \propto \eta_t$, we also implement Step-Decay with probability $P_t \propto \eta_t$ in Figure \ref{fig:stepdecay:prob}. We can see that the output with probability $P_t \propto 1/\eta_t$ is more concentrated at the last phase and has a higher probability, especially, in terms of loss, compared to the result of probability $P_t \propto \eta_t$.

The performance of the various step-size schedules is shown in Figure \ref{fig:stepdecay:mnist}. It is observed that Exp-Decay and Step-Decay performs better than other step-sizes both in loss (training and testing) and testing accuracy. \replaced{Step-Decay has an advantage over Exp-Decay in the later stages of training \replaced{where it}{and} attains a lower training and testing loss.}{In particular, Step-Decay exceeds Exp-Decay and achieves lower loss in terms of training loss and testing loss at the second stage.}
\begin{figure*}[ht]
\vskip 0.1in
\begin{center}
\centerline{
\subfigure[Training loss]{\includegraphics[width=0.3\textwidth,height=1.6in]{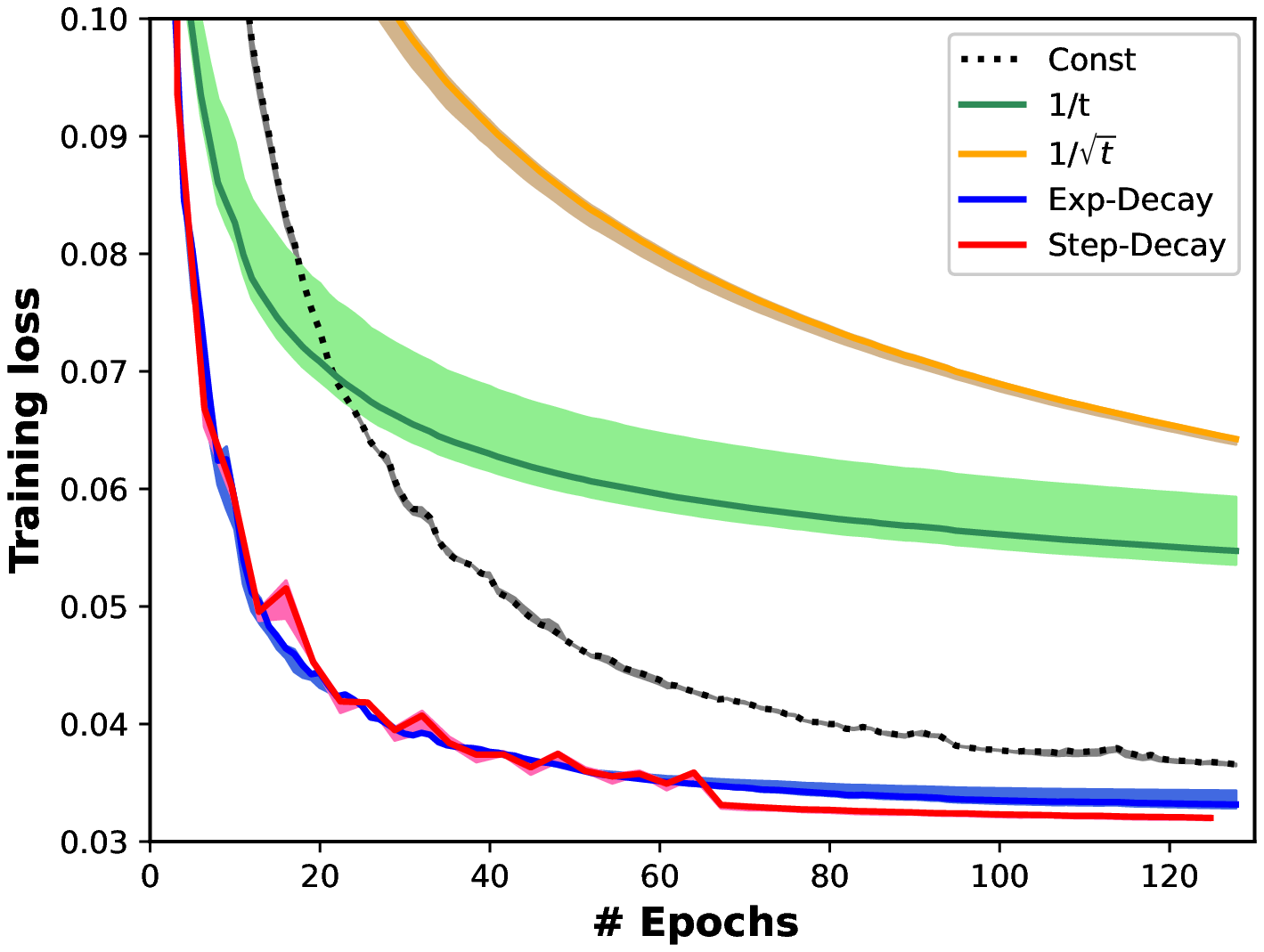} }
\subfigure[Testing loss]{\includegraphics[width=0.3\textwidth,height=1.6in]{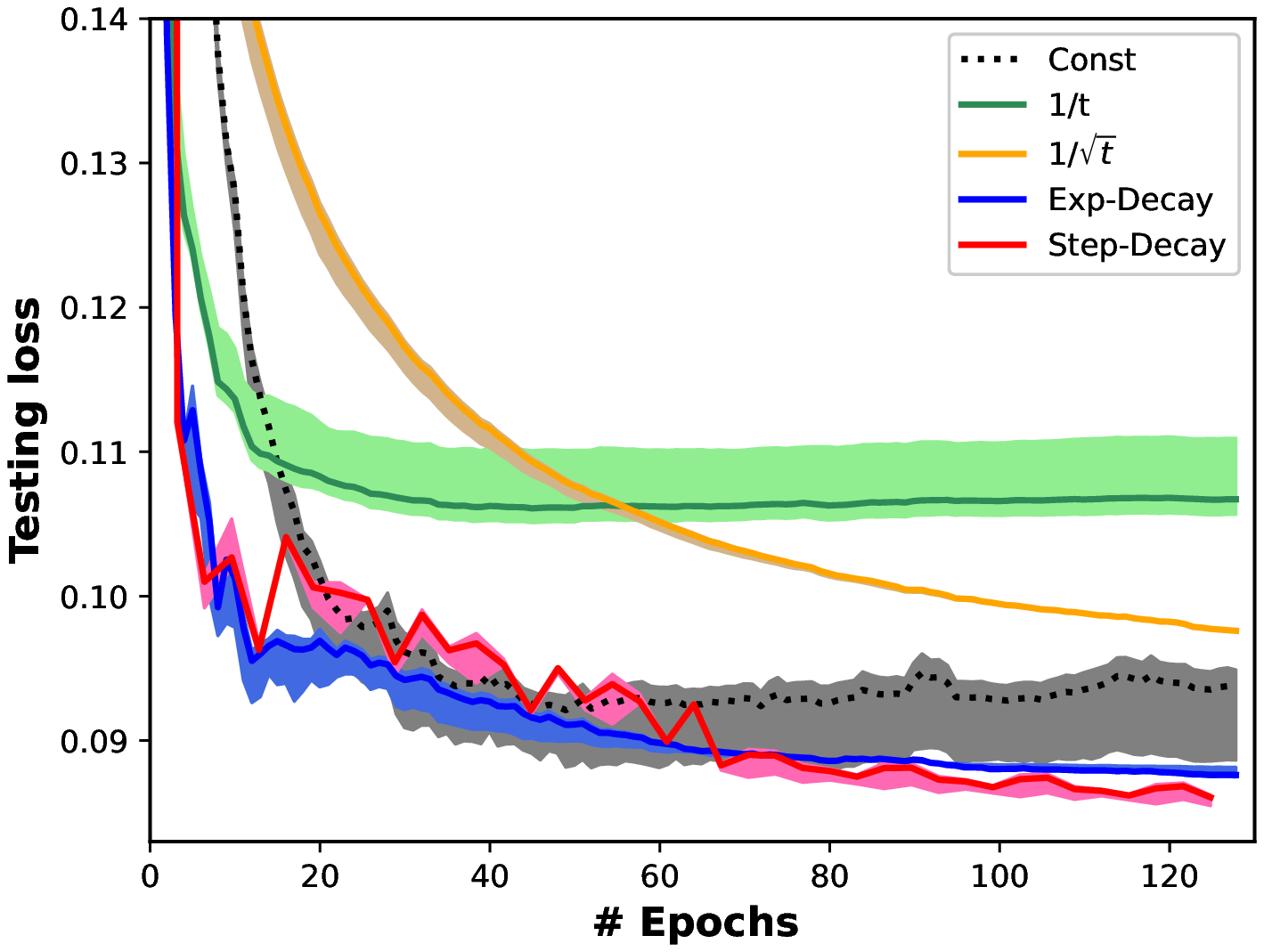} }
\subfigure[Testing accuracy]{\includegraphics[width=0.3\textwidth,height=1.6in]{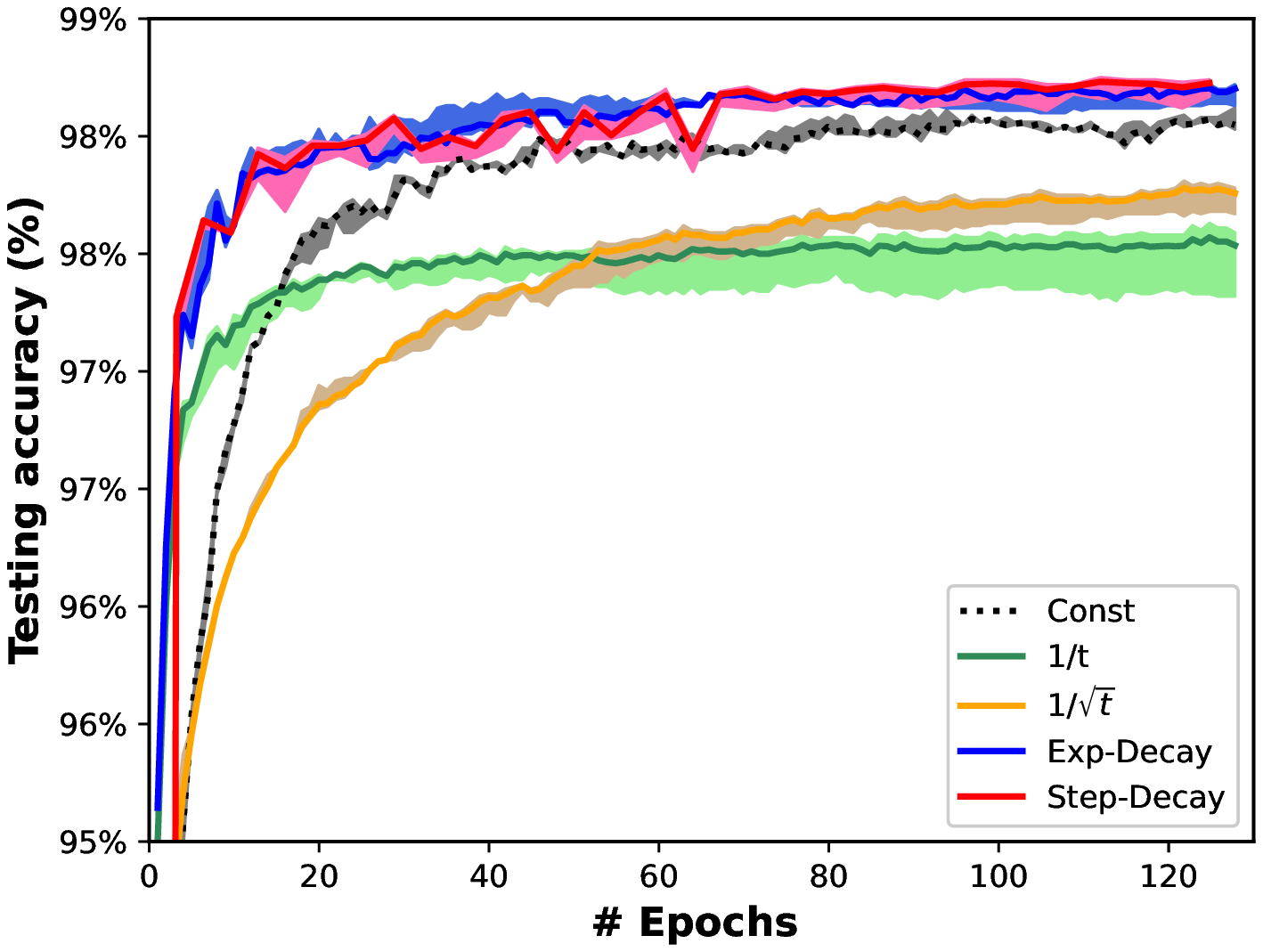} }} 
\caption{Results on MNIST}
\label{fig:stepdecay:mnist}  
\end{center}
\vskip -0.1in
\end{figure*}

\subsection{Experiments on CIFAR10 and CIFAR100}
\label{experiment:cifar}
To illustrate the practical implications of \added{the} step decay step-size, we perform experiments with deep learning tasks on the CIFAR\footnote{\url{https://www.cs.toronto.edu/~kriz/cifar.html}} dataset. We will focus on results for CIFAR100 here, and present complementary results for CIFAR10 in the supplementary material.
To eliminate the influence of stochasticity, all experiments are repeated 5 times.

We consider the benchmark experiments for CIFAR100 on a 100-layer DenseNet~\citep{huang2017densely}. We employ vanilla SGD without dampening and use \added{a} weight decay of 0.0005.  \replaced{The optimal step-size and algorithm parameters are selected using a grid search detailed in the supplementary material.}{, We use grid search to select the optimal parameters of the step-sizes and the model on CIFAR100 (the details are given in the supplementary material)} The results \deleted{on CIFAR100} are shown in Figure \ref{fig:stepdecay:cifar100}. \replaced{We observe that Step-Decay achieves the best results in terms of both testing loss and testing accuracy, and that it is also fast in reaching a competitive solution.}{can get better results compared to other step-sizes in terms of testing loss and testing accuracy at the final stage. Besides, Step-Decay step-size can be faster to reach a good solution.}
Another observation is that as the iterates proceeds \replaced{in}{at} each phase, its testing loss and accuracy is getting worse because its generalization ability is weakened. Therefore, deciding when to stop the iteration or reduce the step-size is important. 

\begin{figure*}[ht]
\vskip 0.1in
\begin{center}
\centerline{
\subfigure[Training loss]{\includegraphics[width=0.3\textwidth,height=1.6in]{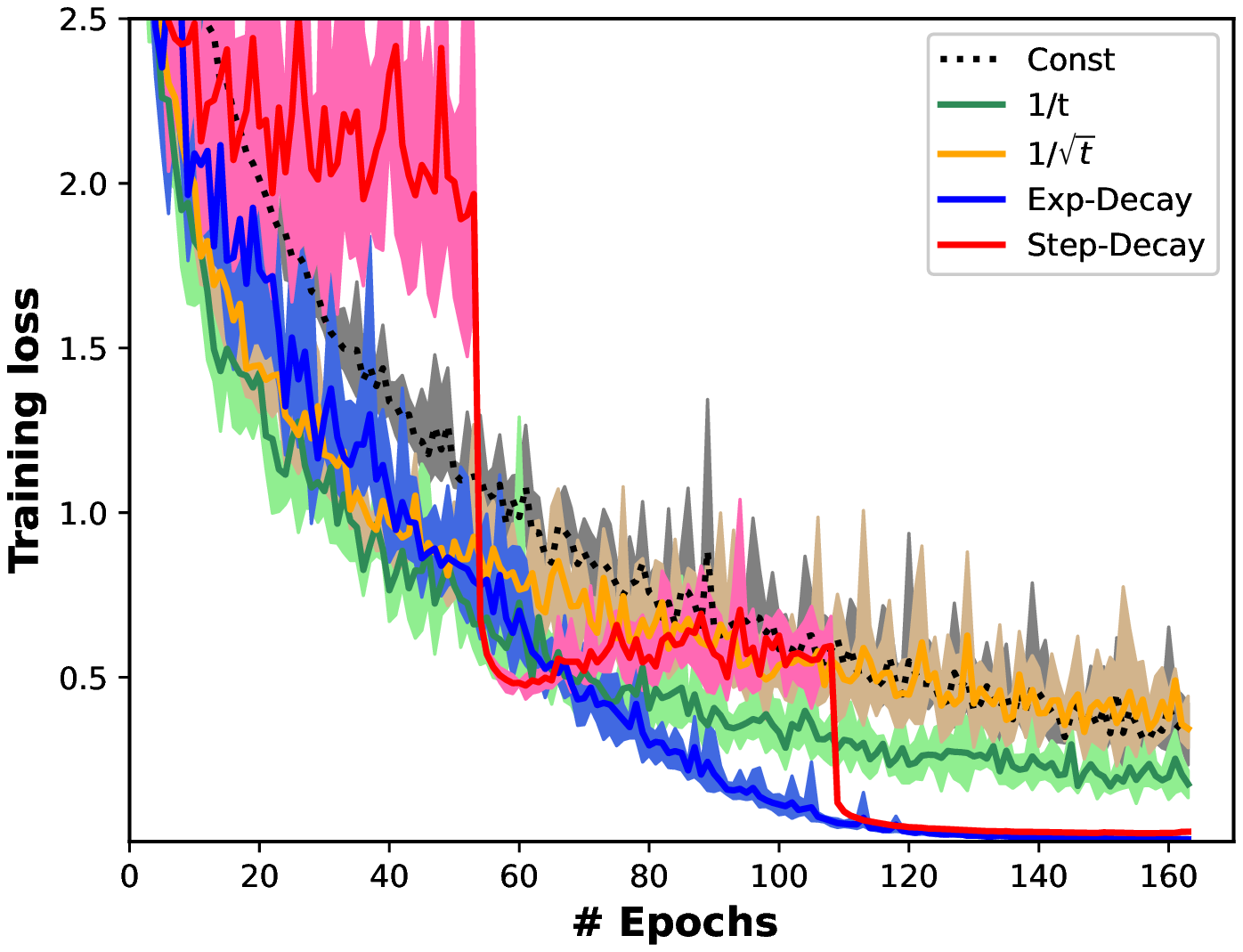} }
\subfigure[Testing loss]{\includegraphics[width=0.3\textwidth,height=1.6in]{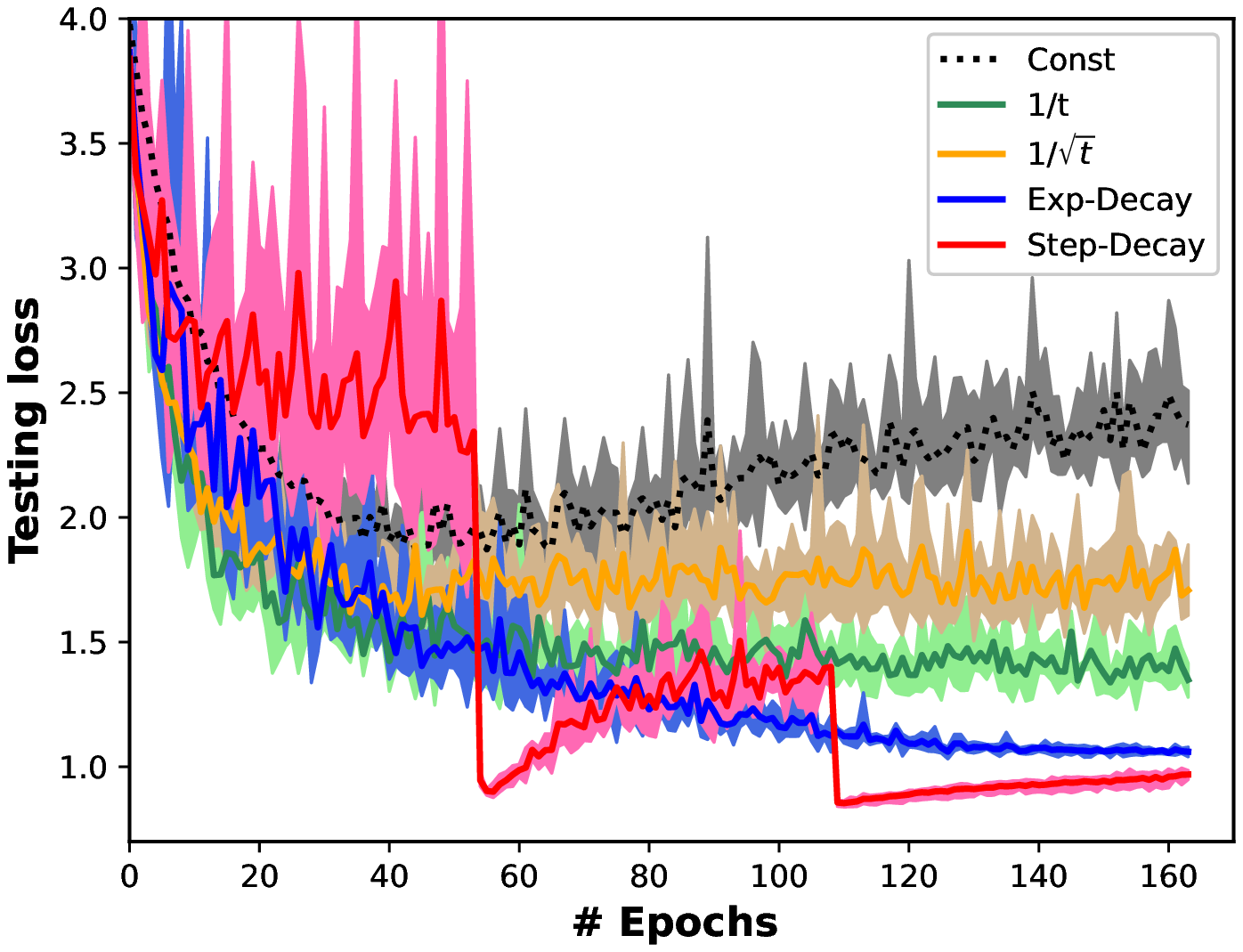} }
\subfigure[Testing accuracy]{\includegraphics[width=0.3\textwidth,height=1.6in]{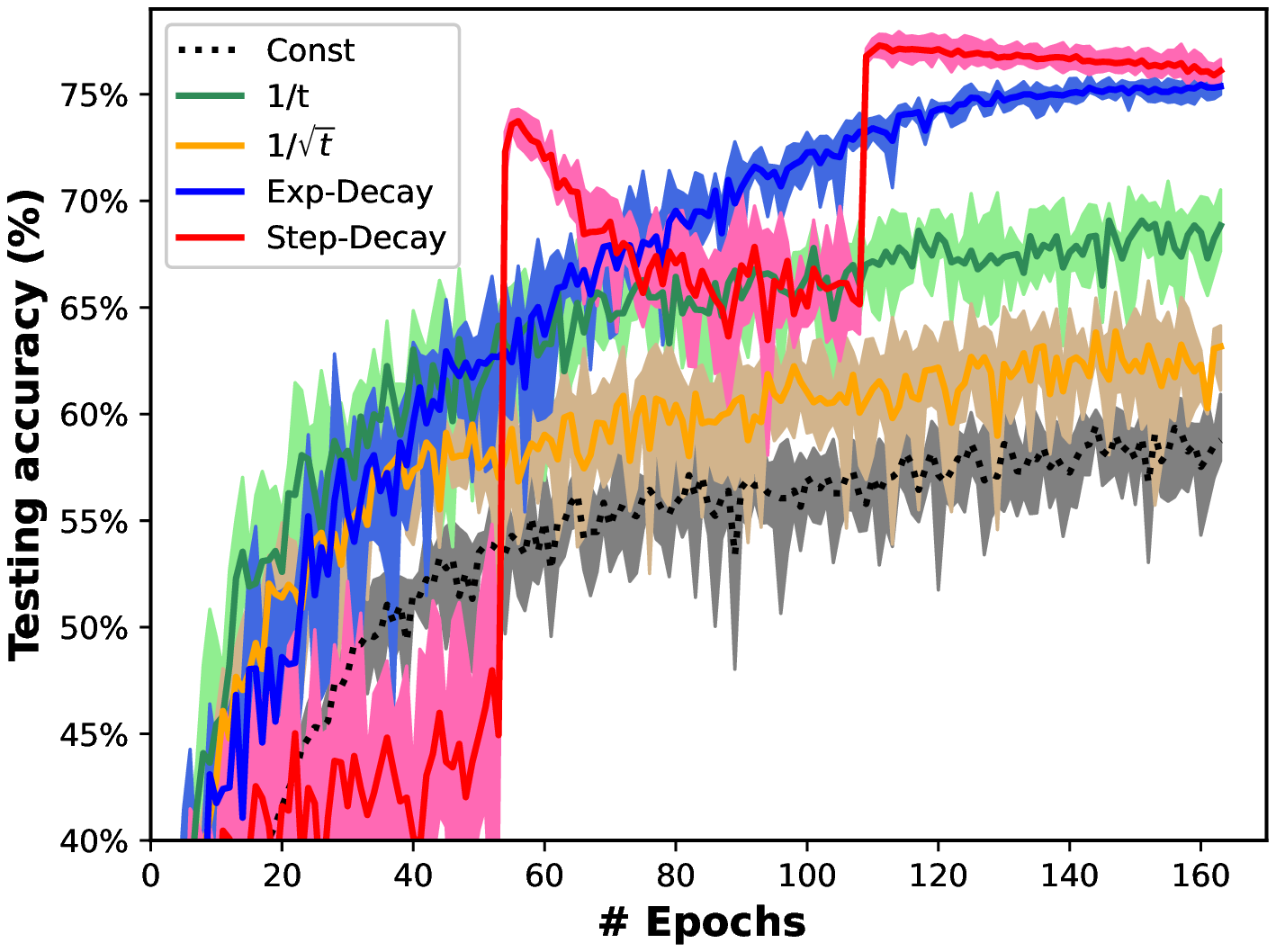}  }} 
\caption{Results on CIFAR100 - DenseNet}
\label{fig:stepdecay:cifar100}  
\end{center}
\vskip -0.1in
\end{figure*}

Finally, we compare the performance of Exp-Decay and Step-Decay on Nesterov’s accelerated gradient (NAG)~\citep{nesterov1983,sutskever2013importance} \deleted{of 0.9} and other adaptive gradient methods, \replaced{including}{for example,} AdaGrad~\citep{AdaGrad}, Adam~\citep{Adam} and AdamW~\citep{AdamW}. The results are shown in Table \ref{tab:cifar100:adap-grad}. 
 All the parameters involved in step-sizes, algorithms, and models are best-tuned (shown in supplementary material).  The $\pm$ shows 95\% confidence intervals of the mean accuracy value over 5 runs.  We can see that compared to \added{the} Exp-Decay step-size, Step-Decay \deleted{step-size} can reach  higher testing accuracy on Adam, AdamW and NAG. \replaced{We therefore beleieve that}{This means that} Step-Decay step-size is more likely to be extended and applied to other methods.
\begin{table}[t]
\caption{The testing accuracy on CIFAR100-DenseNet}
\label{tab:cifar100:adap-grad}
\begin{center}
\begin{small}
\begin{sc}
\begin{tabular}{lc}
\toprule
 \multirow{2}{*}{Method}  & CIFAR100-DenseNet \\
 & Testing accuracy \\ \midrule
AdaGrad  & 0.6197 $\pm$ 0.00518 \\ \hline 
   Adam + Exp-Decay & 0.6936 $\pm$ {\bf 0.00483} \\ \hline 
   Adam + Step-Decay & {\bf 0.7041} $\pm$ 0.00971\\ \hline 
   AdamW + Exp-Decay & 0.7165 $\pm$ 0.00353 \\ \hline 
   AdamW + Step-Decay &  {\bf 0.7335} $\pm$ {\bf 0.00261} \\ \hline NAG + Exp-Decay  & 0.7531 $\pm$ 0.00606 \\ \hline 
   NAG + Step-Decay &  {\bf 0.7568} $\pm$ {\bf 0.00156}\\
   \bottomrule
\end{tabular}
\end{sc}
\end{small}
\end{center}
\end{table}



\subsection{Experiments on Regularized Logistic Regression } 
\label{experiment:logistic}

 We now turn our attention to how the step-decay step-size and other related step-sizes behave in the strongly convex setting. 
We consider the regularized logistic regression problem on the binary classification dataset rcv1.binary ($n = 20242; d =
47236$) from LIBSVM~\footnote{\url{https://www.csie.ntu.edu.tw/~cjlin/libsvmtools/datasets/}}, where a 0.75 partition is used for training and the rest is for testing. 

Figure~\ref{fig:rcv:constant} shows how a constant step-size has to be well-tuned to give good performance: if we choose it too small, convergence will be painstakingly slow; if we set it too large, iterates will stagnate or even diverge. This effect is also visualized in blue in Figure~\ref{fig:rcv:initial}, where only a narrow range of values results in a low training loss for the constant step-size. In contrast, the initial step-size of both Exp-Decay and Step-Decay can be selected from a wide range ($\approx 10-10^4$) and still yield good results in the end. In other words, Exp-Decay and Step-Decay are more robust to the choice of initial step-size than the alternatives. A more thorough evaluation of all the considered step-sizes on logistic regression can be found in Supplementary~D.3. 
 \begin{figure}[ht]
\begin{center}
\centerline{\subfigure[Constant step-size]{ \label{fig:rcv:constant} \includegraphics[width=0.35\textwidth,height=1.8in]{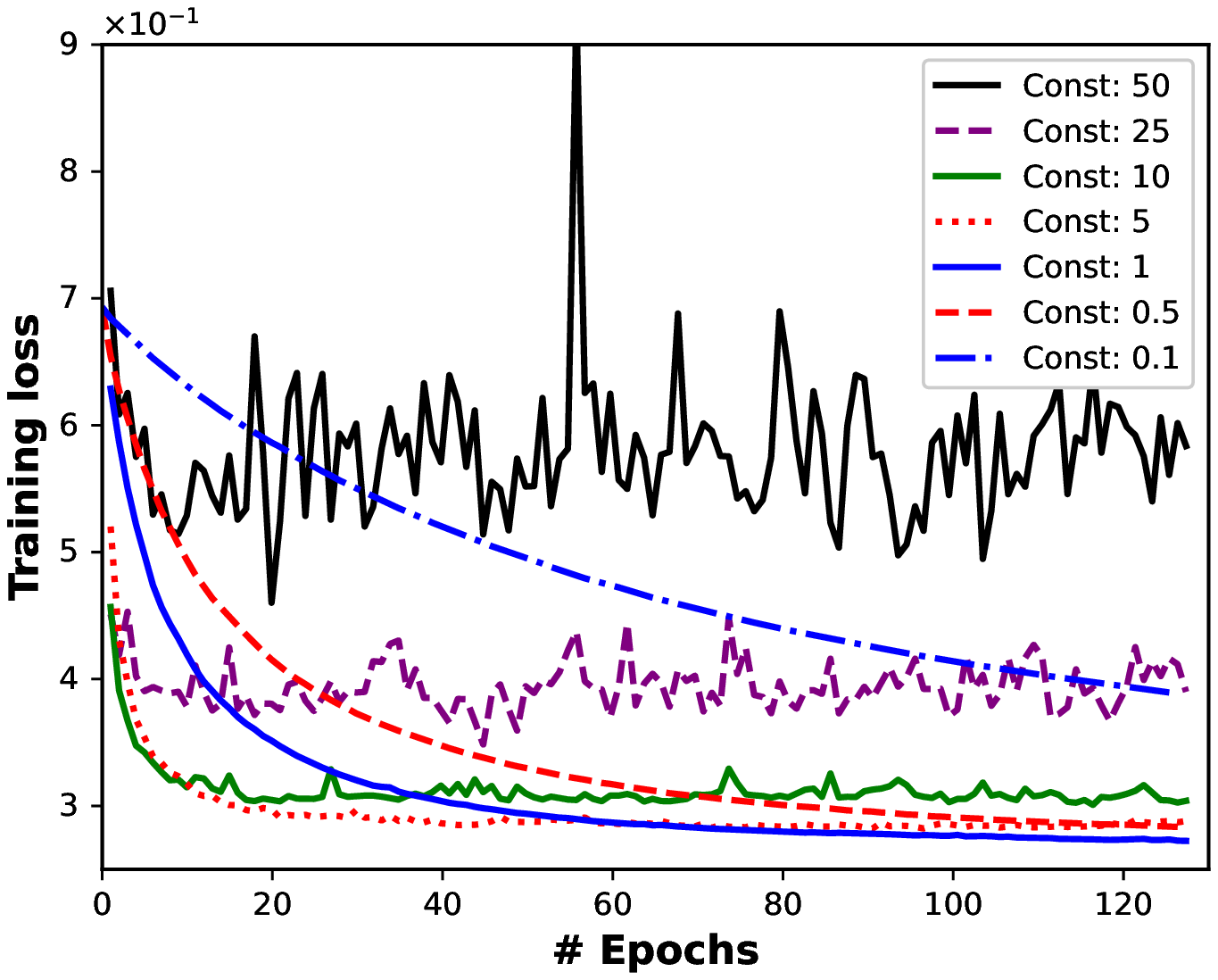}}
 \subfigure[Initial step-sizes]{ \label{fig:rcv:initial} \includegraphics[width=0.35\textwidth,height=1.8in]{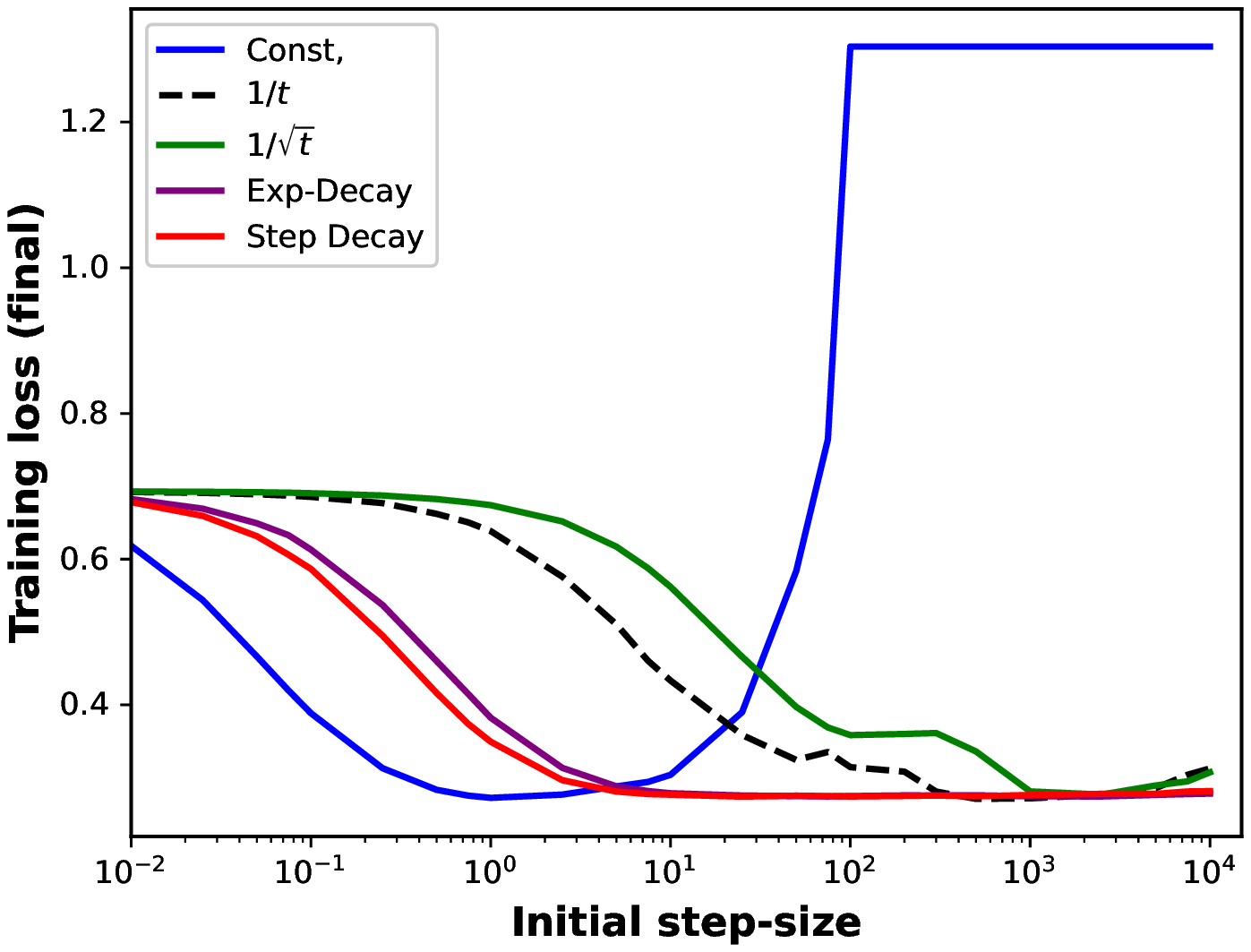}}
}
\caption{The results on rcv1.binary - logistic($L_2$)}
\label{fig:const:stepdecay} 
\end{center}
 \end{figure}

\section{Conclusion}
\label{sec:conclusion}

We have provided theoretical guarantees for SGD under the step decay family of step-sizes, widely used in deep learning.
%
\replaced{Our first results established a near-optimal $\mathcal{O}(\ln T/\sqrt{T})$ rate for step decay step-sizes in the non-convex setting. A key step in our analysis was to use a novel non-uniform probability distribution $P_t \propto 1/\eta_t$ for selecting the output of the algorithm. We showed that this approach allows to improve the convergence results for SGD under other step-sizes as well, e.g.,  by removing the $\ln T$ term in the best known convergence rate for  $\eta_t =1/\sqrt{t}$ step-sizes.}{
In the nonconvex and smooth case, we proposed a novel non-uniform probability policy $P_t \propto 1/\eta_t$ for selecting the output 
and the experiment on MNIST showed that the theoretical output can reach the result of the last iterate with high probability. Based on this policy, we proved a near-optimal $\mathcal{O}(\ln T/\sqrt{T})$ rate for step decay step-sizes, improved the results for exponential decay step-size~\citep{li2020exponential} and removed the $\ln T$ term for the standard $1/\sqrt{t}$ step-size.} Moreover, we established near-optimal (compared to the min-max rate) convergence rates \replaced{for}{in several cases including}  general convex, strongly convex and smooth, and strongly convex and nonsmooth problems. \replaced{We illustrated the superior performance of step-decay step-sizes for training of large-scale deep neural networks.}{The experiments on several real world datasets shows that step decay step-size outperforms than the constant and polynomial step-sizes, is competitive to exponential decay step-size and more likely to be extended to other methods (e.g. NAG and Adam), and is more robust to select the initial step-size than polynomial step-sizes.} In the experiments, we observed that as the iterates proceeding in each phase, their generalization abilities are getting worse. Therefore, it will be an interesting to study how to best select the size of inner-loop $S$ (instead of constant or exponentially growing) to avoid the loss of generalization.


\bibliography{sgd_stepdecay}

\begin{thebibliography}{39}
\providecommand{\natexlab}[1]{#1}
\providecommand{\url}[1]{\texttt{#1}}
\expandafter\ifx\csname urlstyle\endcsname\relax
  \providecommand{\doi}[1]{doi: #1}\else
  \providecommand{\doi}{doi: \begingroup \urlstyle{rm}\Url}\fi

\bibitem[Abadi et~al.(2015)Abadi, Agarwal, Barham, Brevdo, Chen, Citro,
  Corrado, Davis, Dean, Devin, et~al.]{tensorflow}
M.~Abadi, A.~Agarwal, P.~Barham, E.~Brevdo, Z.~Chen, C.~Citro, G.~Corrado,
  A.~Davis, J.~Dean, M.~Devin, et~al.
\newblock Tensorflow: Large-scale machine learning on heterogeneous distributed
  systems.
\newblock \emph{Software available from tensorflow.org}, 2015.

\bibitem[Cartis et~al.(2010)Cartis, Gould, and Toint]{cartis2010complexity}
C.~Cartis, N.~I. Gould, and P.~L. Toint.
\newblock On the complexity of steepest descent, newton's and regularized
  newton's methods for nonconvex unconstrained optimization problems.
\newblock \emph{SIAM Journal on Optimization}, 20\penalty0 (6):\penalty0
  2833--2852, 2010.

\bibitem[Davis et~al.(2019{\natexlab{a}})Davis, Drusvyatskiy, and
  Charisopoulos]{davis2019stochastic}
D.~Davis, D.~Drusvyatskiy, and V.~Charisopoulos.
\newblock Stochastic algorithms with geometric step decay converge linearly on
  sharp functions.
\newblock \emph{arXiv:1907.09547}, 2019{\natexlab{a}}.

\bibitem[Davis et~al.(2019{\natexlab{b}})Davis, Drusvyatskiy, Xiao, and
  Zhang]{davis2019low}
D.~Davis, D.~Drusvyatskiy, L.~Xiao, and J.~Zhang.
\newblock From low probability to high confidence in stochastic convex
  optimization.
\newblock \emph{arXiv:1907.13307}, 2019{\natexlab{b}}.

\bibitem[Drori and Shamir(2020)]{drori2020complexity}
Y.~Drori and O.~Shamir.
\newblock The complexity of finding stationary points with stochastic gradient
  descent.
\newblock In \emph{International Conference on Machine Learning}, pages
  2658--2667. PMLR, 2020.

\bibitem[Duchi et~al.(2011)Duchi, Hazan, and Singer]{AdaGrad}
J.~Duchi, E.~Hazan, and Y.~Singer.
\newblock Adaptive subgradient methods for online learning and stochastic
  optimization.
\newblock \emph{Journal of Machine Learning Research}, 12\penalty0
  (Jul):\penalty0 2121--2159, 2011.

\bibitem[Ge et~al.(2019)Ge, Kakade, Kidambi, and Netrapalli]{ge2019step}
R.~Ge, S.~M. Kakade, R.~Kidambi, and P.~Netrapalli.
\newblock The step decay schedule: A near optimal, geometrically decaying
  learning rate procedure for least squares.
\newblock In \emph{Advances in Neural Information Processing Systems}, pages
  14977--14988, 2019.

\bibitem[Ghadimi and Lan(2013)]{ghadimi2013stochastic}
S.~Ghadimi and G.~Lan.
\newblock Stochastic first-and zeroth-order methods for nonconvex stochastic
  programming.
\newblock \emph{SIAM Journal on Optimization}, 23\penalty0 (4):\penalty0
  2341--2368, 2013.

\bibitem[Ghadimi et~al.(2016)Ghadimi, Lan, and Zhang]{ghadimi2016mini}
S.~Ghadimi, G.~Lan, and H.~Zhang.
\newblock Mini-batch stochastic approximation methods for nonconvex stochastic
  composite optimization.
\newblock \emph{Mathematical Programming}, 155\penalty0 (1-2):\penalty0
  267--305, 2016.

\bibitem[Goffin(1977)]{goffin1977}
J.-L. Goffin.
\newblock On convergence rates of subgradient optimization methods.
\newblock \emph{Mathematical Programming}, 13\penalty0 (1):\penalty0 329--347,
  1977.

\bibitem[Gower et~al.(2019)Gower, Loizou, Qian, Sailanbayev, Shulgin, and
  Richt{\'a}rik]{Gower-etal.2019}
R.~M. Gower, N.~Loizou, X.~Qian, A.~Sailanbayev, E.~Shulgin, and
  P.~Richt{\'a}rik.
\newblock {SGD}: General analysis and improved rates.
\newblock In \emph{International Conference on Machine Learning}, pages
  5200--5209. PMLR, 2019.

\bibitem[Harvey et~al.(2019{\natexlab{a}})Harvey, Liaw, Plan, and
  Randhawa]{harvey2018tight}
N.~J. Harvey, C.~Liaw, Y.~Plan, and S.~Randhawa.
\newblock Tight analyses for non-smooth stochastic gradient descent.
\newblock In \emph{Conference on Learning Theory}, pages 1579--1613. PMLR,
  2019{\natexlab{a}}.

\bibitem[Harvey et~al.(2019{\natexlab{b}})Harvey, Liaw, and
  Randhawa]{harvey2019simple}
N.~J. Harvey, C.~Liaw, and S.~Randhawa.
\newblock Simple and optimal high-probability bounds for strongly-convex
  stochastic gradient descent.
\newblock \emph{arXiv:1909.00843}, 2019{\natexlab{b}}.

\bibitem[Hazan and Kale(2014)]{hazan2014beyond}
E.~Hazan and S.~Kale.
\newblock Beyond the regret minimization barrier: optimal algorithms for
  stochastic strongly-convex optimization.
\newblock \emph{Journal of Machine Learning Research}, 15\penalty0
  (1):\penalty0 2489--2512, 2014.

\bibitem[He et~al.(2016)He, Zhang, Ren, and Sun]{he2016deep}
K.~He, X.~Zhang, S.~Ren, and J.~Sun.
\newblock Deep residual learning for image recognition.
\newblock In \emph{Proceedings of the IEEE conference on computer vision and
  pattern recognition}, pages 770--778, 2016.

\bibitem[Huang et~al.(2017)Huang, Liu, Van Der~Maaten, and
  Weinberger]{huang2017densely}
G.~Huang, Z.~Liu, L.~Van Der~Maaten, and K.~Q. Weinberger.
\newblock Densely connected convolutional networks.
\newblock In \emph{Proceedings of the IEEE Conference on Computer Vision and
  Pattern Recognition}, pages 4700--4708, 2017.

\bibitem[Kingma and Ba(2015)]{Adam}
D.~P. Kingma and J.~L. Ba.
\newblock {ADAM}: A method for stochastic optimization.
\newblock In \emph{International Conference on Learning Representations}, 2015.

\bibitem[Klein and Young(2015)]{klein2015number}
P.~Klein and N.~E. Young.
\newblock On the number of iterations for dantzig--wolfe optimization and
  packing-covering approximation algorithms.
\newblock \emph{SIAM Journal on Computing}, 44\penalty0 (4):\penalty0
  1154--1172, 2015.

\bibitem[Krizhevsky et~al.(2012)Krizhevsky, Sutskever, and Hinton]{imagenet}
A.~Krizhevsky, I.~Sutskever, and G.~E. Hinton.
\newblock Imagenet classification with deep convolutional neural networks.
\newblock In \emph{Advances in Neural Information Processing Systems}, pages
  1106--1114, 2012.

\bibitem[Lacoste-Julien et~al.(2012)Lacoste-Julien, Schmidt, and
  Bach]{Lacoste-Schmidt-Bach2012}
S.~Lacoste-Julien, M.~Schmidt, and F.~Bach.
\newblock A simpler approach to obtaining an $\mathcal{O}(1/t)$ convergence
  rate for the projected stochastic subgradient method.
\newblock \emph{arXiv:1212.2002}, 2012.

\bibitem[Lang et~al.(2019)Lang, Xiao, and Zhang]{lang2019using}
H.~Lang, L.~Xiao, and P.~Zhang.
\newblock Using statistics to automate stochastic optimization.
\newblock In \emph{Advances in Neural Information Processing Systems}, pages
  9540--9550, 2019.

\bibitem[Li et~al.(2020)Li, Zhuang, and Orabona]{li2020exponential}
X.~Li, Z.~Zhuang, and F.~Orabona.
\newblock A second look at exponential and cosine step sizes: Simplicity,
  convergence, and performance.
\newblock \emph{arXiv:2002.05273}, 2020.

\bibitem[Loizou et~al.(2020)Loizou, Vaswani, Laradji, and
  Lacoste-Julien]{polyak2020}
N.~Loizou, S.~Vaswani, I.~Laradji, and S.~Lacoste-Julien.
\newblock Stochastic {Polyak} step-size for {SGD}: A adaptive learning rate for
  fast convergence.
\newblock \emph{arXiv:2002.10542}, 2020.

\bibitem[Loshchilov and Hutter(2019)]{AdamW}
I.~Loshchilov and F.~Hutter.
\newblock Decoupled weight decay regularization.
\newblock In \emph{International Conference on Learning Representations}, 2019.

\bibitem[Moulines and Bach(2011)]{Moulines-Bach2011}
E.~Moulines and F.~R. Bach.
\newblock Non-asymptotic analysis of stochastic approximation algorithms for
  machine learning.
\newblock In \emph{Advances in Neural Information Processing Systems}, pages
  451--459, 2011.

\bibitem[Nemirovski et~al.(2009)Nemirovski, Juditsky, Lan, and
  Shapiro]{SGD-complex}
A.~Nemirovski, A.~Juditsky, G.~Lan, and A.~Shapiro.
\newblock Robust stochastic approximation approach to stochastic programming.
\newblock \emph{SIAM Journal on Optimization}, 19\penalty0 (4):\penalty0
  1574--1609, 2009.

\bibitem[Nesterov(2004)]{nesterov2003introductory}
Y.~Nesterov.
\newblock \emph{Introductory lectures on convex optimization: A basic course},
  volume~87.
\newblock Springer Science \& Business Media, 2004.

\bibitem[Nesterov(1983)]{nesterov1983}
Y.~E. Nesterov.
\newblock A method for solving the convex programming problem with convergence
  rate o (1/k\^{} 2).
\newblock In \emph{Dokl. akad. nauk Sssr}, volume 269, pages 543--547, 1983.

\bibitem[Paszke et~al.(2017)Paszke, Gross, Chintala, Chanan, Yang, DeVito, Lin,
  Desmaison, Antiga, and Lerer]{pytorch}
A.~Paszke, S.~Gross, S.~Chintala, G.~Chanan, E.~Yang, Z.~DeVito, Z.~Lin,
  A.~Desmaison, L.~Antiga, and A.~Lerer.
\newblock Automatic differentiation in {PyTorch}.
\newblock In \emph{NIPS 2017 Autodiff Workshop: The Future of Gradient-based
  Machine Learning Software and Techniques}, Long Beach, CA, US, 2017.

\bibitem[Rakhlin et~al.(2012)Rakhlin, Shamir, and
  Sridharan]{Rakhlin-Shamir-Sridharan2011}
A.~Rakhlin, O.~Shamir, and K.~Sridharan.
\newblock Making gradient descent optimal for strongly convex stochastic
  optimization.
\newblock In \emph{International Conference on Machine Learning}, pages
  1571--1578, 2012.

\bibitem[Robbins and Monro(1951)]{SGD-1951}
H.~Robbins and S.~Monro.
\newblock A stochastic approximation method.
\newblock \emph{The Annals of Mathematical Statistics}, pages 400--407, 1951.

\bibitem[Shamir and Zhang(2013)]{Shamir-Zhang2013}
O.~Shamir and T.~Zhang.
\newblock Stochastic gradient descent for non-smooth optimization: Convergence
  results and optimal averaging schemes.
\newblock In \emph{International Conference on Machine Learning}, pages 71--79,
  2013.

\bibitem[Shor(2012)]{shor2012}
N.~Z. Shor.
\newblock \emph{Minimization methods for non-differentiable functions},
  volume~3.
\newblock Springer Science \& Business Media, 2012.

\bibitem[Sutskever et~al.(2013)Sutskever, Martens, Dahl, and
  Hinton]{sutskever2013importance}
I.~Sutskever, J.~Martens, G.~Dahl, and G.~Hinton.
\newblock On the importance of initialization and momentum in deep learning.
\newblock In \emph{International Conference on Machine Learning}, pages
  1139--1147. PMLR, 2013.

\bibitem[Tieleman and Hinton(2012)]{RMSProp}
T.~Tieleman and G.~Hinton.
\newblock Lecture 6.5-rmsprop, coursera: Neural networks for machine learning.
\newblock \emph{Technical Report, University of Toronto}, 2012.

\bibitem[Vaswani et~al.(2019)Vaswani, Mishkin, Laradji, Schmidt, Gidel, and
  Lacoste-Julien]{auto-line-search}
S.~Vaswani, A.~Mishkin, I.~Laradji, M.~Schmidt, G.~Gidel, and
  S.~Lacoste-Julien.
\newblock Painless stochastic gradient: Interpolation, line-search, and
  convergence rates.
\newblock In \emph{Advances in Neural Information Processing Systems}, 2019.

\bibitem[Xu et~al.(2016)Xu, Lin, and Yang]{xu2016accelerated}
Y.~Xu, Q.~Lin, and T.~Yang.
\newblock Accelerated stochastic subgradient methods under local error bound
  condition.
\newblock \emph{arXiv:1607.01027}, 2016.

\bibitem[Yuan et~al.(2019)Yuan, Yan, Jin, and Yang]{yuan2019stagewise}
Z.~Yuan, Y.~Yan, R.~Jin, and T.~Yang.
\newblock Stagewise training accelerates convergence of testing error over
  {SGD}.
\newblock \emph{Advances in Neural Information Processing Systems},
  32:\penalty0 2608--2618, 2019.

\bibitem[Zhang et~al.(2020)Zhang, Lang, Liu, and Xiao]{zhang2020statistical}
P.~Zhang, H.~Lang, Q.~Liu, and L.~Xiao.
\newblock Statistical adaptive stochastic gradient methods.
\newblock \emph{arXiv:2002.10597}, 2020.

\end{thebibliography}

\newpage
\vspace{0.5em}

{\centering
\section*{Supplementary Material for "On the Convergence of Step Decay Step-Size for Stochastic Optimization"}

}
\vspace{0.5em}

	 
	

\section*{A. Proofs of Section \ref{sec:nonconvex}}

Before presenting the proofs of Section~\ref{sec:nonconvex}, we state and prove the following useful lemma.
\begin{lemma}\label{lem:nonconvex}
Suppose that $f$ is $L$-smooth on $\R^d$ and the stochastic gradient oracle is variance-bounded by $V^2$. If $\eta_t \leq 1/L$, consider the SGD algorithm, we have 
\begin{align*}
\frac{\eta_t}{2}\E[\left\|\nabla f(x_t)\right\|^2] \leq \E[f(x_t)] - \E[f(x_{t+1})] + \frac{LV^2\eta_t^2}{2}.
\end{align*}
\end{lemma}

\begin{proof}
Recall the SGD iterations $x_{t+1} = x_t - \eta_t \hat{g}_t$, where $\E[\hat{g}_t] = \nabla f(x_t)$. By the smoothness of $f$ on $\R^d$,  we have
\begin{align}
f(x_{t+1}) & \leq f(x_t) + \left\langle \nabla f(x_t), x_{t+1} - x_t \right\rangle + \frac{L}{2}\left\|x_{t+1} -x_t \right\|^2 \notag \\
& \leq f(x_t) + \left\langle \nabla f(x_t), -\eta_t \hat{g}_t \right\rangle + \frac{L}{2}\left\|x_{t+1} -x_t \right\|^2 \notag \\
\label{equ:lem:nonconvex:1} & \leq f(x_t) -\eta_t \left\langle \nabla f(x_t), \hat{g}_t \right\rangle  + \frac{L\eta_t^2}{2}\left\|\hat{g}_t\right\|^2.
\end{align}
The stochastic gradient oracle is variance-bounded by $V^2$, i.e.,  $\E[\left\| \hat{g}_t - \nabla f(x_t)\right\|^2] \leq V^2$. Taking expectation on both sides of \eqref{equ:lem:nonconvex:1} and applying $\E[\hat{g}_t] = \nabla f(x_t)$  and the variance-bounded assumption gives
\begin{align}
\E[f(x_{t+1})] & \leq \E[f(x_t)] - \eta_t \E[\left\|\nabla f(x_t)\right\|^2] + \frac{L\eta_t^2}{2}\E[\left\|\hat{g}_t \right\|^2] \notag \\
& \leq \E[f(x_t)] - \eta_t \E[\left\|\nabla f(x_t)\right\|^2] + \frac{L\eta_t^2}{2}\E[\left\|\hat{g}_t -\nabla f(x_t) + \nabla f(x_t)\right\|^2] \notag \\\
& \leq \E[f(x_t)] - \eta_t \E[\left\|\nabla f(x_t)\right\|^2] + \frac{L\eta_t^2}{2}\left(\E[\left\|\hat{g}_t -\nabla f(x_t)\right\|^2 + \E[\left\|\nabla f(x_t)\right\|^2]\right) \notag \\\
& \leq \E[f(x_t)] + \left(-\eta_t + \frac{L\eta_t^2}{2}\right)\E[\left\|\nabla f(x_t)\right\|^2] + \frac{L\eta_t^2}{2}\E[\left\|\hat{g}_t -\nabla f(x_t)\right\|^2 \notag \\\
\label{inequ:nonconvex:1} & \leq \E[f(x_t)] + \left(-\eta_t + \frac{L\eta_t^2}{2}\right) \E[\left\|\nabla f(x_t)\right\|^2] + \frac{LV^2\eta_t^2}{2}.
\end{align}
where the second inequality follows the fact that gradient $\hat{g}_t$ is unbiased ($\E[\hat{g}_t] = \nabla f(x_t)$). We thus have that
\begin{align*}
 \E[\left\|\hat{g}_t-\nabla f(x_t)+ \nabla f(x_t) \right\|^2] & = \E[\left\|\hat{g}_t-\nabla f(x_t)\right\|^2 + \left\|\nabla f(x_t)\right\|^2 +2\left\langle \hat{g}_t-\nabla f(x_t), \nabla f(x_t)\right\rangle ] \\
 & = \E[\left\|\hat{g}_t-\nabla f(x_t)\right\|^2] + \E[\left\|\nabla f(x_t)\right\|^2].  
\end{align*}
Since $\eta_0 \leq 1/L$, we have $ - \eta_t + \eta_t^2L/2 \leq -\eta_t/2$. Using this inequality, shifting $\E[\left\|\nabla f(x_t)\right\|^2]$ to the left side, and re-arranging \eqref{inequ:nonconvex:1}  gives
\begin{align*}
\frac{\eta_t}{2}\E[\left\|\nabla f(x_t)\right\|^2] \leq \E[f(x_t)] - \E[f(x_t)] + \frac{LV^2\eta_t^2}{2}.
\end{align*}
\end{proof}

\begin{proof}(of {\bf Proposition \ref{pro:1sqrt}})
By Lemma \ref{lem:nonconvex}, the following inequality holds for SGD
\begin{align}\label{equ:sqrt:1}
\frac{\eta_t}{2}\E[\left\|\nabla f(x_t)\right\|^2] \leq \E[f(x_t)] - \E[f(x_{t+1})] + \frac{LV^2\eta_t^2}{2}.
\end{align}
Under the diminishing step-size $\eta_t = \eta_0/\sqrt{t}$, we estimate the summation of $\eta_t$ and $\eta_t^2$ from $t=1$ to $T$, respectively:
\begin{align*}
 \sum_{t=1}^{T}\eta_t =  \eta_0\sum_{t=1}^{T}\frac{1}{\sqrt{t}} & \geq \eta_0\int_{t=1}^{T}\frac{1}{\sqrt{t}}  dt = 2\eta_0(\sqrt{T}-1) \\ 
 \sum_{t=1}^{T}\eta_t^2 =\eta_0^2\sum_{t=1}^{T}\frac{1}{t} & \leq \eta_0^2\left(1 + \int_{t=1}^{T}\frac{1}{t}dt \right) = \eta_0^2(\ln T+1).
\end{align*}
Recall that the output $\hat{x}_T$ is chosen randomly from the sequence $\left\lbrace x_t \right\rbrace_{t=1}^{T}$ with probability $P_t = \frac{\eta_t}{\sum_{t=1}^{T}\eta_t}$. We thus have
\begin{align*}
\E[\left\|\nabla f(\hat{x}_T)\right\|] & = \frac{\eta_t\E[\left\|\nabla f(x_t)\right\|^2]}{\sum_{t=1}^{T}\eta_t} \notag\\
& \leq \frac{2\sum_{t=1}^{T}\E[f(x_t) -\E[f(x_{t+1})]]}{\sum_{t=1}^{T}\eta_t} + \frac{LV^2\sum_{t=1}^{T}\eta_t^2}{\sum_{t=1}^{T}\eta_t} \notag \\
& \leq \frac{2(f(x_1) - f^{\ast})}{\sum_{t=1}^{T}\eta_t}  + \frac{LV^2\sum_{t=1}^{T}\eta_t^2}{\sum_{t=1}^{T}\eta_t} \notag \\
& \leq \frac{f(x_1)-f^{\ast}}{\eta_0(\sqrt{T}-1)} + \frac{LV^2\eta_0(\ln T +1 )}{2(\sqrt{T}-1)}.
\end{align*}
\end{proof}

\begin{proof}(of {\bf Theorem \ref{thm:nonconvex:1}}) In this case, we analyze the convergence of Algorithm \ref{alg:nonconvex} with $S= 2T/\log_{\alpha} T$ and $N = \log_{\alpha} T/2$. Invoking the result of Lemma \ref{lem:nonconvex}, at the current iterate $x_{i}^t$, the following inequality holds:
\begin{align*}
\frac{\eta_t}{2}\E[\left\|\nabla f(x_i^t)\right\|^2] \leq \E[f(x_i^t)] - \E[f(x_{i+1}^t)] + \frac{LV^2\eta_t^2}{2},
\end{align*}
Dividing both sides by $\eta_t^2/2$ yields
\begin{align}\label{inequ:nonconvex:2}
\frac{1}{\eta_t}\E[\left\|\nabla f(x_i^t)\right\|^2] \leq \frac{2}{\eta_t^2}\left(\E[f(x_i^t)] - \E[f(x_{i+1}^t)]\right) + LV^2
\end{align}
At each inner phase for $i \in [S]$, the step-size $\eta_t$ is a constant. Applying \eqref{inequ:nonconvex:2} repeatedly for $i=1,2,\cdots S$ gives
\begin{align}\label{inequ:nonconvex:3}
\frac{1}{\eta_t}\sum_{i=1}^{S}\E[\left\|\nabla f(x_i^t)\right\|^2] & \leq \frac{2}{\eta_t^2}\left(\E[f(x_1^t)] - \E[f(x_{S+1}^t)]\right) +  LV^2\cdot S
\end{align}
Since the output $\hat{x}_T$ of Algorithm \ref{alg:nonconvex} is randomly chosen from all  previous iterates $\left\lbrace x_i^t\right\rbrace$ with probability $P_i^t = \frac{1/\eta_t}{S\sum_{l=1}^{N}1/\eta_t} $,
\begin{align}
\E[\left\| \nabla f(\hat{x}_T)\right\|^2] & = \frac{1}{S\sum_{t=1}^{N}(1/\eta_t)}\sum_{t=1}^{N}\frac{1}{\eta_t} \sum_{i=1}^{S}\left\|\nabla f(x_i^t)\right\|^2 \notag \\
& \leq \frac{1}{S\sum_{t=1}^{N}(1/\eta_t)}\sum_{t=1}^{N}\frac{2}{\eta_t^2}\left(\E[f(x_1^t)] - \E[f(x_{S+1}^t)]\right) + \frac{LV^2\cdot S}{S\sum_{t=1}^{N}1/\eta_t} \notag \\
\label{eqn:nonconvex:4} & \leq \frac{2\eta_0}{S\eta_0^2\sum_{t=1}^{N}\alpha^{t-1}}\sum_{t=1}^{N}\alpha^{2(t-1)}\left(\E[f(x_1^t)] - \E[f(x_{S+1}^t)]\right) + \frac{LV^2\eta_0}{\sum_{t=1}^{N}\alpha^{t-1}}. 
\end{align}
By the update rule in Algorithm~\ref{alg:nonconvex}, the last point $x_{S+1}^t$ of the $t^{\text{th}}$ loop is the starting point of the next loop, i.e., $x_{S+1}^t = x_1^{t+1}$ for each $t \in [N]$. Applying the assumption that the objective function is upper bounded by $f_{\max} > 0$ for all $x\in\R^d$, 
\begin{align*}
\sum_{t=1}^{N}\alpha^{2(t-1)}\left(\E[f(x_1^t)] - \E[f(x_{S+1}^t)]\right)&  \leq  f(x_1^1) + (\alpha^2-1)\sum_{t=2}^{N}\alpha^{2(t-2)}\cdot \E[f(x_1^{t})] \leq \frac{f_{\max} }{\alpha^2}\cdot T.
\end{align*}
Plugging this inequality into \eqref{eqn:nonconvex:4} and substituting  $N = \log_{\alpha} T/2$ and $S  =2T/\log_{\alpha} T$, we  obtain that 
\begin{align*}
\E[\left\| \nabla f(\hat{x}_T)\right\|^2] & \leq \frac{2\eta_0 f_{\max} T }{\eta_0^2\alpha^2( 2T/\log_{\alpha} T)\cdot\frac{(\sqrt{T}-1)}{(\alpha-1)}}+ \frac{LV^2\eta_0}{ \frac{(\sqrt{T}-1)}{\alpha-1}} \\
& \leq \frac{ (\alpha-1)f_{\max}}{\eta_0\alpha^2}\cdot \frac{\log_{\alpha} T}{\sqrt{T}-1} + \frac{LV^2\eta_0(\alpha-1)}{\sqrt{T}-1}.
\end{align*}
 Therefore, by changing the base $\alpha $ of $\log_{\alpha} T$ to be natural logarithm, the theorem is proved.
\end{proof}


\begin{proof}(of {\bf Theorem \ref{thm:exp:nonconvex}})
 Based on Lemma \ref{lem:nonconvex}, at the current iterate $x_t$, we have
\begin{align*}
\frac{\eta_t}{2}\E[\left\|\nabla f(x_t)\right\|^2] \leq \E[f(x_t)] - \E[f(x_{t+1})] + \frac{LV^2\eta_t^2}{2}.
\end{align*}
Dividing the above inequality by $\eta_t^2/2$ and summing over $t=1$ to $T$ gives
\begin{align}\label{exp:inequ:1}
\sum_{t=1}^{T}\frac{1}{\eta_t} \E[\left\|\nabla f(x_t)\right\|^2] \leq \sum_{t=1}^{T}\frac{2(\E[f(x_t)] - \E[f(x_{t+1})])}{\eta_t^2} + LV^2T
\end{align}
Applying the assumption that the objective function $f$ is upper bounded by $f_{\max}$ and recalling the definition of the exponential decay step-size~\citep{li2020exponential}, i.e., $\eta_t = \eta_0/\alpha^t$ where $ \alpha = (\beta/T)^{-1/T}$ and $ \beta \geq 1$, we find
\begin{align}
\sum_{t=1}^{T}\frac{2(\E[f(x_t)] - \E[f(x_{t+1})])}{\eta_t^2} & \leq  \frac{2f(x_1)}{\eta_1^2} + 2\sum_{t=2}^{T}\left(\frac{1}{\eta_t^2}- \frac{1}{\eta_{t-1}^2}\right)\E[f(x_t)] \notag \\
& \leq \frac{2f(x_1)}{\eta_1^2} + \frac{2}{\eta_0^2}\left(\alpha^2-1\right)\sum_{t=2}^{T}\alpha^{2(t-1)} f_{\max} \notag\\
\label{exp:inequ:2} & \leq  \frac{2f_{\max}}{\eta_0^2}\cdot \left(\frac{T}{\beta}\right)^2.
\end{align}
Next, we estimate the sum of $1/\eta_t$ from $t=1$ to $T$:
\begin{align}
\sum_{t=1}^{T}\frac{1}{\eta_t} & = \frac{1}{\eta_0}\sum_{t=1}^{T}\alpha^t = \frac{\alpha(1-\alpha^T)}{\eta_0(1-\alpha)}  = \frac{(\frac{T}{\beta}-1)}{\eta_0(1-1/\alpha)}  \notag \\
\label{exp:inequ:3} &  \geq \frac{(\frac{T}{\beta}-1)}{\eta_0\ln(\alpha)} = \frac{(\frac{T}{\beta}-1)T}{\eta_0\ln(\frac{T}{\beta})}.
\end{align}
where the last inequality follows from the fact that $1-x \leq \ln(\frac{1}{x})$ for all $x > 0$.

Combining the selection rule for $\hat{x}_T$ with 
inequalities \eqref{exp:inequ:1}, \eqref{exp:inequ:2} and \eqref{exp:inequ:3}, we have
\begin{align*}
\E[\left\| \nabla f(\hat{x}_T)\right\|^2] & =  \frac{1}{\sum_{t=1}^{T}1/\eta_t}\left[\sum_{t=1}^{T}\frac{1}{\eta_t}\E[\left\|\nabla f(x_t)\right\|^2] \right] \\
& \leq \frac{\eta_0\ln(\frac{T}{\beta})}{(\frac{T}{\beta}-1)T}\left[\frac{2f_{\max}}{\eta_0^2}\left(\frac{T}{\beta}\right)^2 + LV^2T\right].
\end{align*}
Letting $\beta = \sqrt{T}$, we see that 
\begin{align*}
\E[\left\| \nabla f(\hat{x}_T)\right\|^2] \leq \frac{\eta_0\ln T}{2(\sqrt{T}-1)}\left[\frac{2f_{\max}}{\eta_0^2} + LV^2\right],
\end{align*}
which concludes the proof.
\end{proof}

\begin{proof}(of {\bf Theorem \ref{thm:sqrt}})
In this case, we consider the diminishing step-size $\eta_t = \eta_0/\sqrt{t}$.
By Lemma \ref{lem:nonconvex}:
\begin{align}\label{thm:equ:sqrt:1}
\frac{\eta_t}{2}\E[\left\|\nabla f(x_t)\right\|^2] \leq \E[f(x_t)] - \E[f(x_{t+1})] + \frac{LV^2\eta_t^2}{2}.
\end{align}
In the same way as Theorems \ref{thm:nonconvex:1} and \ref{thm:exp:nonconvex}, we 
divide \eqref{thm:equ:sqrt:1} by $\eta_t^2/2$ and sum over $t=1$ to $T$ to obtain
\begin{align}\label{thm:equ:sqrt:2}
\sum_{t=1}^{T}\frac{1}{\eta_t} \E[\left\|\nabla f(x_t)\right\|^2] & \leq \sum_{t=1}^{T}\frac{2(\E[f(x_t)] - \E[f(x_{t+1})])}{\eta_t^2} + LV^2T \notag \\
& \leq \frac{2}{\eta_0^2}\left[f(x_1) + \sum_{t=2}^{T}(t-(t-1))\E[f(x_t)]\right] + LV^2T \notag \\
& \leq \frac{2Tf_{\max}}{\eta_0^2} + LV^2T.
\end{align}
Recalling that the output $\hat{x}_{T}$ is randomly chosen from the sequence $\left\lbrace x_t \right\rbrace_{t=1}^{T}$ with probability $P_t = \frac{1/\eta_t}{\sum_{t=1}^{T}1/\eta_t}$ and applying \eqref{thm:equ:sqrt:2}, yields
\begin{align*}
\E[\left\|\nabla f(\hat{x}_T) \right\|^2]  = \frac{\sum_{t=1}^{T}1/\eta_t\E[\left\|\nabla f(x_t) \right\|^2]}{\sum_{t=1}^{T} 1/\eta_t }  & \leq  \frac{1}{\sum_{t=1}^{T} 1/\eta_t }\left[\frac{2Tf_{\max}}{\eta_0^2} + LV^2T\right] \notag \\
& \leq \left(\frac{3f_{\max}}{\eta_0} + \frac{3LV^2\eta_0}{2}\right)\cdot \frac{1}{\sqrt{T}}.
\end{align*}
where the last inequality holds because $\sum_{t=1}^{T}1/\eta_t \geq 1/\eta_0\cdot \int_{t=0}^{T}\sqrt{t}dt = \frac{2}{3\eta_0}T^{3/2}$. 
The proof is complete.
\end{proof}

\section*{B. Proofs of Section \ref{sec:convex}}
\begin{proof}(of {\bf Theorem 4.1})
The convexity of $f$ yields 
$\left\langle {g}_i^t, x-x_i^t\right\rangle \leq f(x)-f(x_i^t)$ for any $x \in \mathcal{X}$, where $g_i^t \in \partial f(x_i^t)$. Also, by convexity of $\mathcal{X}$, we have $\left\|\Pi_{\mathcal{X}}(u)-v\right\| \leq \left\|u-v\right\|$ for any points $u \in \R^d$ and $v \in \mathcal{X}$. Using these inequalities and applying the assumption that the stochastic gradient oracle is bounded by $G^2$, i.e., $\E[\left\|\hat{g}_i^t \right\|^2] \leq G^2$, for any $x \in \mathcal{X}$, we have
\begin{align}\label{equ:convex:0}
\E[\left\| x_{i+1}^t - x\right\|^2] & = \E[\left\| \Pi_{\mathcal{X}}(x_{i}^t - \eta_t \hat{g}_i^t) - x\right\|^2] \leq \E[\left\| x_{i}^t - \eta_t \hat{g}_i^t - x\right\|^2] \notag \\
& \leq \E[\left\| x_{i}^t- x\right\|^2] -2\eta_t \E[\left\langle {g}_i^t, x_i^t-x\right\rangle] + \eta_t^2G^2 \notag \\
& \leq \E[\left\| x_{i}^t- x\right\|^2] -2\eta_t[f(x_i^t)- f(x)] + \eta_t^2G^2
\end{align}
Shifting $[f(x_i^t)- f(x)]$ to the left side gives
\begin{align}\label{eqn:convex:1}
2\eta_t \E[f(x_i^t)-f(x)] \leq \E[\left\| x_{i}^t- x\right\|^2] - \E[\left\| x_{i+1}^t - x\right\|^2] + \eta_t^2 G^2.
\end{align}
Now, consider the final phase $N$ (let $t=N$) and $x=x^{\ast}$ and apply \eqref{eqn:convex:1}  recursively  from $i=1$ to $S$  to obtain
\begin{align}
2\eta_{N} \sum_{i=1}^{S}\E[f(x_i^{N})-f^{\ast}] & \leq \E[\left\|x_1^{N}-x^{\ast}\right\|^2]   + S\eta_{N}^2G^2.
\end{align}
Combining the assumption that  $\sup_{x,y\in \mathcal{X}}\left\|x-y\right\|^2 \leq D^2$ for some finite $D$ with the expression for the step-size in the final phase, $\eta_{N} = \eta_0\alpha/\sqrt{T}$, we have
\begin{align}
\frac{\sum_{i=1}^{S} \E[f(x_i^{N})]}{S} - f^{\ast} &  \leq \frac{\E\left[\left\|x_1^{N}-x^{\ast}\right\|^2\right]}{2\eta_{N}S} + \frac{G^2\eta_{N}}{2} \notag \notag \\
\label{thm:convex:part1}& \leq \frac{D^2}{4\eta_0\alpha}\cdot \frac{\log_{\alpha} T}{\sqrt{T}} + \frac{G^2\eta_0\alpha}{2}\cdot \frac{1}{\sqrt{T}}.
\end{align}
We have thus proven the first part of Theorem \ref{thm:convex:lastloop}. Next, based on the above results, we prove the error bound for the last iterate.
 We focus on the last phase $N$ and 
 apply \eqref{eqn:convex:1} recursively from $i=S-k$ to $S$ to find
\begin{align}\label{equ:convex:thm1:1}
2\eta_{N} \sum_{i={S}-k}^{S}\E[f(x_i^{N})-f(x_{{S}-k}^{N})] \leq \E\left[\left\|x_{{S}-k}^{N}-x_{{S}-k}^{N}\right\|^2\right] - \E\left[\left\|x_{{S}+1}^{N}-x_{{S}-k}^{N}\right\|^2\right] + (k+1)\eta_{N}^2G^2.
\end{align}
Introduce $W_{k+1}^{N} := \frac{1}{k+1}\sum_{i={S}-k}^{S}\E[f(x_i^{N})] $. Inequality \eqref{equ:convex:thm1:1} implies that 
\begin{align}\label{eqn:convex:last}
-f(x_{{S}-k}^{N}) \leq -W_{k+1}^{N} + \frac{\eta_{N} G^2}{2}.
\end{align}
By the definition of $W_{k+1}^{N} $, we have  $(k+1)W_{k+1}^{N}-kW_{k}^{N} = f(x_{{S}-k}^{N})$. Using this formula and applying \eqref{eqn:convex:last} gives
\begin{align*}
kW_{k}^{N} & = (k+1)W_{k+1}^{N} - f(x_{{S}-k}^{N}) \leq (k+1)W_{k+1}^{N} - W_{k+1}^{N} + \frac{\eta_{N}G^2}{2}.
\end{align*}
Dividing by $k$, we get
\begin{align*}
W_{k}^{N} \leq W_{k+1}^{N} + \frac{\eta_{N}G^2}{2k}.
\end{align*}
Using the above inequality repeatedly for $k=1,2,\cdots, S-1$, we have
\begin{align}
W_1^{N} \leq W_{S}^{N}  + \frac{\eta_{N}G^2}{2}\sum_{k=1}^{S-1}\frac{1}{k} \leq W_{S}^{N} + \frac{G^2\eta_0\alpha }{2\sqrt{T}}\left(\ln\left(\frac{2T}{\log_{\alpha} T}\right) + 1\right).
\end{align}
Recalling the definition of $W_{S}^{N}$ and applying \eqref{thm:convex:part1} 
into the above inequality, we obtain \begin{align*}
 \E[f(x_{S}^{N})] - f^{\ast} \leq    \frac{D^2}{4\eta_0\alpha}\cdot \frac{\log_{\alpha} T}{\sqrt{T}} + \frac{G^2\eta_0\alpha}{2}\cdot \frac{\ln T }{\sqrt{T}} + \frac{G^2\eta_0\alpha(1+\ln(2)/2)}{\sqrt{T}}.
\end{align*}
The proof is complete.
\end{proof}

\section*{C. Proofs of Section \ref{sec:sc}}

\begin{proof}(of {\bf Theorem \ref{thm:sc} })
In this case, we consider the step decay step-size (see Algorithm \ref{alg:convex}) with $N = T/\log_{\alpha} T$ and $S = \log_{\alpha} T$.
By the $\mu$-strongly convexity of $f$ on $\mathcal{X}$, we have 
\begin{align}
\label{eqn:sc2} f(x^{\ast}) & \geq  f(x_i^t) + \left\langle g_i^t, x^{\ast}-x_i^t\right\rangle + \frac{\mu}{2}\left\|x_i^t-x^{\ast}\right\|^2, \,\text{and \,} \\
\label{eqn:sc1} f(x_i^t) & \geq f(x^{\ast})  + \left\langle g^{\ast}, x_i^t-x^{\ast}\right\rangle + \frac{\mu}{2}\left\|x_i^t-x^{\ast}\right\|^2, \forall g^{\ast} \in \partial f(x^{\ast}).
\end{align}
Due to the fact that $x^{\ast}$ minimizes $f$ on $\mathcal{X}$, we have $\left\langle g^{\ast}, x-x^{\ast}\right\rangle \geq 0$ for all $g^{\ast} \in \partial f(x^{\ast})$ and $x \in \mathcal{X}$.
In particular, for $x=x_i^t$ we have $\left\langle g^{\ast}, x_i^t-x^{\ast}\right\rangle \geq 0$.  Plugging this into \eqref{eqn:sc1} and re-arranging \eqref{eqn:sc2} and \eqref{eqn:sc1} gives
\begin{equation} 
\left\langle g_i^t, x_i^t - x^{\ast}\right\rangle \geq \frac{\mu}{2}\left\|x_{i}^t - x^{\ast}\right\|^2 + f(x_i^t) - f(x^{\ast}) \geq \mu\left\|x_{i}^t - x^{\ast}\right\|^2.
\end{equation}
By the convexity of $\mathcal{X}$, we have $\left\|\Pi_{\mathcal{X}}(u) -v \right\|^2 \leq \left\|u-v\right\|^2$ for any $u\in \R^d$ and $v \in \mathcal{X}$. Then applying the update rule of Algorithm \ref{alg:convex} and using these inequalities gives 
\begin{align}
\E\left[\left\|x_{i+1}^t - x^{\ast}\right\|^2 \mid  \mathcal{F}_i^t \right]& = \E\left[\left\|\Pi_{\mathcal{X}}(x_i^t - \eta_t \hat{g}_i^t )- x^{\ast} \right\|^2 \mid  \mathcal{F}_i^t \right] \leq \E\left[\left\|x_i^t - \eta_t \hat{g}_i^t- x^{\ast} \right\|^2 \mid  \mathcal{F}_i^t \right] \notag \\
& \leq \left\|x_i^t - x^{\ast} \right\|^2 - 2\eta_t \E\left[\left\langle \hat{g}_i^t, x_i^t - x^{\ast}\right\rangle \mid \mathcal{F}_i^t\right] + \eta_t^2\E\left[\left\|\hat{g}_i^t\right\|^2 \mid \mathcal{F}_i^t\right] \notag\\
& \leq \left\|x_i^t - x^{\ast} \right\|^2 - 2\eta_t \left\langle g_i^t, x_i^t - x^{\ast}\right\rangle + \eta_t^2G^2 \notag \\
\label{eqn:key:1} & \leq (1-2\mu\eta_t)\left\|x_i^t - x^{\ast} \right\|^2  + \eta_t^2G^2,
\end{align}
where the third inequality follows from the stochastic gradient oracle is bounded by $G^2$, i.e.,  $\hat{g}_i^t$ satisfies that $\E[\hat{g}_i^t] = g_i^t \in \partial f(x_i^t)$ and $\E[\left\|\hat{g}_i^t\right\|^2] \leq G^2$.
In this case, the time horizon $T$ is divided into ${N}=\log_{\alpha} T $ phases and each is of length ${S}=T/\log_{\alpha} T$. Recursively applying \eqref{eqn:key:1} from $i=1$ to ${S}$ in the $t^{\text{th}}$ phase and using the assumption $\eta_0 < 1/(2\mu)$ gives
\begin{align}
\E\left[\left\|x_{S+1}^t - x^{\ast}\right\|^2\right] & \leq  (1-2\mu\eta_t)^{S}\E\left[\left\|x_1^t - x^{\ast} \right\|^2 \right] + G^2\eta_t^2\sum_{l=0}^{S-1} (1-2\mu\eta_t)^l \notag \\
\label{eqn:perloop} & \leq  (1-2\mu\eta_t)^{S}\E\left[\left\|x_1^t - x^{\ast} \right\|^2 \right]  + \frac{G^2}{2\mu}\cdot \eta_t.
\end{align}
Repeating the recursion \eqref{eqn:perloop} from $t=1$ to $N$, we get
\begin{align}
\E\left[\left\|x_{{S}+1}^{N} - x^{\ast}\right\|^2\right]  
& \leq \mathop{\Pi}\limits_{t=1}^{N}(1-2\mu\eta_t)^{S}\left\|x_1^1 - x^{\ast} \right\|^2 + \frac{G^2}{2\mu}\sum_{t=1}^{N} \eta_t \cdot \mathop{\Pi}\limits_{l > t}^{N} (1-2\mu\eta_{l})^{S} \notag\\
& \leq \exp\left(-2\mu {S}\sum_{t=1}^{N}\eta_t\right)\left\|x_1^1 - x^{\ast} \right\|^2 +  \frac{G^2}{2\mu}\sum_{t=1}^{N} \eta_t \exp\left(-2\mu {S}\sum_{l>t}^{N}\eta_{l}\right), \label{eqn:core:1}
\end{align}
where the second inequality follows from the fact that $(1+x)^{s} \leq \exp(sx)$ for any $x\in \R$ and $s >0$. Recalling the formula for the step decay step-size, $\eta_t = \eta_0/\alpha^{t-1}$ in the $t^{\text{th}}$ phase, and that $N=\log_{\alpha} T$ and ${S}= T/\log_{\alpha} T$, we can estimate the two sums that appear in \eqref{eqn:core:1} as follows:
\begin{align*}
{S}\sum_{t=1}^{N}\eta_t  = \frac{T}{\log_{\alpha} T}\cdot\frac{\eta_0(1-\alpha^{-N})}{1-1/\alpha}& = \frac{\eta_0\alpha}{(\alpha-1)}\cdot\frac{T-1}{\log_{\alpha} T} \\
    {S}\sum_{l>t}^{N}\eta_l = \frac{T}{\log_{\alpha} T}\cdot \frac{\eta_0\alpha^{-t}(1-\alpha^{-(N-t)})}{(1-1/\alpha)} & = \frac{\eta_0\alpha}{(\alpha-1)}\cdot\frac{T\alpha^{-t}-1}{\log_{\alpha} T}.
\end{align*}
Using these inequalities in \eqref{eqn:core:1} gives
\begin{align}\label{eqn:core:2}
 \E\left[\left\|x_{{S}+1}^{N} - x^{\ast}\right\|^2\right]  & \leq \exp\left(- \frac{2\mu\eta_0\alpha}{\alpha-1} \cdot\frac{T-1}{\log_{\alpha} T}\right)\left\|x_1^1 - x^{\ast} \right\|^2  + \frac{G^2\eta_0}{2\mu}\sum_{t=1}^{N} \frac{1}{\alpha^{t-1}} \exp\left(- \frac{2\mu\eta_0\alpha}{\alpha-1}\cdot\frac{ T\alpha^{-t}-1}{\log_{\alpha} T} \right).
\end{align}
Next, we turn to bound the right-hand side of \eqref{eqn:core:2}. Let $t^{\ast} := \max\left\lbrace 0,  \bigfloor{ \log_{\alpha}\left(\frac{2\mu\eta_0\alpha}{\alpha-1}\cdot\frac{T}{\log_{\alpha} T}\right)}\right\rbrace$. If $t^{\ast} \geq 1$, we prefer to divide the second term into two parts. First, we estimate this term for $t\leq t^{\ast}$:
\begin{align}
\frac{G^2\eta_0}{2\mu} \sum_{t=1}^{t^{\ast}} \frac{1}{\alpha^{t-1}} \exp\left(- \frac{2\mu\eta_0\alpha}{\alpha-1}\cdot\frac{ T\alpha^{-t}-1}{\log_{\alpha} T} \right) 
& \leq  \frac{G^2\eta_0}{2\mu}\sum_{t=1}^{t^{\ast}} \frac{1}{\alpha^{t-1}} \exp\left(- \frac{\alpha^{t^{\ast}}}{\alpha^{t}} + \frac{2\mu\eta_0\alpha}{(\alpha-1)\log_{\alpha} T}\right) \notag\\
&  \leq  \frac{G^2\eta_0\alpha\exp\left(\frac{2\mu\eta_0\alpha}{(\alpha-1)\log_{\alpha} T}\right)}{2\mu \alpha^{t^{\ast}}} \cdot \sum_{t=1}^{t^{\ast}} \frac{\alpha^{t^{\ast}}}{\alpha^t} \exp\left(- \frac{\alpha^{t^{\ast}}}{\alpha^{t}}\right)  \notag\\
  & \leq \frac{G^2\eta_0\alpha \exp\left(\frac{2\mu\eta_0\alpha}{(\alpha-1)\log_{\alpha} T}-1\right)}{2\mu \alpha^{t^{\ast}}} \notag \\ 
\label{eqn:tast:1} & \leq \frac{G^2(\alpha-1)\exp\left(\frac{\mu\eta_0\alpha}{(\alpha-1)\log_{\alpha} T}-1\right)}{2\mu^2} \cdot \frac{\log_{\alpha} T}{T},
\end{align}
where the third inequality uses that $\int_{x=1}^{+\infty} x \exp(-x)dx \leq 2/\exp(1)$.
Next, we estimate the term for $ t^{\ast}<t \leq N$:
\begin{align}
\frac{G^2\eta_0}{2\mu} \sum_{t=t^{\ast}+1}^{N} \frac{1}{\alpha^{t-1}} \exp\left(- \frac{2\mu\eta_0\alpha}{\alpha-1}\cdot\frac{ T\alpha^{-t}-1}{\log_{\alpha} T} \right) 
& \leq \frac{G^2\eta_0\alpha}{2\mu} \sum_{t=t^{\ast}+1}^{N} \frac{1}{\alpha^t} \exp\left(- \frac{\alpha^{t^{\ast}}}{\alpha^t} +\frac{2\mu\eta_0\alpha}{(\alpha-1)\log_{\alpha} T}\right) \notag \\
 &  \leq  \frac{G^2\eta_0\exp\left(\frac{2\mu\eta_0\alpha}{(\alpha-1)\log_{\alpha} T}\right)}{2\mu}\cdot \frac{\alpha(1-2/\exp(1))}{\alpha^{t^{\ast}}} \notag \\
\label{eqn:tast:2} & \leq \frac{G^2(\alpha-1)\exp\left(\frac{2\mu\eta_0\alpha}{(\alpha-1)\log_{\alpha} T}\right)(1-2/\exp(1))}{4\mu^2} \cdot \frac{\log_{\alpha} T}{T},
\end{align}
where the second inequality uses that $\int_{x=0}^{1}x\exp(-x)dx= 1-2/\exp(1)$.
Incorporating \eqref{eqn:tast:1} and \eqref{eqn:tast:2} into \eqref{eqn:core:2} gives
\begin{align*}
 \E\left[\left\|x_{{S}+1}^{N} - x^{\ast}\right\|^2\right] & \leq \frac{\left\|x_1^1 - x^{\ast} \right\|^2}{\exp\left( \frac{2\mu\eta_0\alpha}{\alpha-1} \cdot\frac{T-1}{\log_{\alpha} T}\right)}  + \frac{G^2(\alpha-1)\exp\left(\frac{2\mu\eta_0\alpha}{(\alpha-1)\log_{\alpha} T}\right)}{4\mu^2} \cdot \frac{\log_{\alpha} T}{T}.
\end{align*}
Changing the base $\alpha $ of $\log_{\alpha}$ to the natural logarithm, that is $\log_{\alpha} T = \ln T/\ln\alpha$, we arrive at the desired result.
\end{proof}

Before proving the lower bound of Algorithm \ref{alg:convex}, we state an utility lemma.

\begin{lemma}\citep{klein2015number}\label{lem:proba:1}  Let $X_1, X_2, \cdots, X_{K}$ be independent random variables taking values uniformly from $\left\lbrace -1, +1\right\rbrace$  and $X = \frac{1}{K}\sum_{i=1}^{K}X_i$. Suppose $2 \leq c \leq \frac{\sqrt{K}}{2}$, then
\begin{align*}
\mathbb{P}\left[X \geq \frac{c}{\sqrt{K}} \right] \geq \exp(-9c^2/2).
\end{align*}
\end{lemma}

The lemma was proposed by \citet{klein2015number} (see lemma 4) to show the tightness of the Chernoff bound. Recently it has been used to derive a high probability lower bound of the $1/t$ step-size \citep{harvey2019simple}.
We will now use Lemma \ref{lem:proba:1} to prove the following high probability lower bound of the step decay scheme in Algorithm \ref{alg:convex} for $S = T/\log_{\alpha} T$.

\begin{proof}(of {\bf Theorem \ref{thm:lower-bound}})
We consider the one-dimensional function $\tilde{f}_{T}(x) = \frac{1}{2}x^2$, where $x \in \mathcal{X}=[-4, 4]$. This function is $1$-strongly convex and $1$-smooth on $\mathcal{X}$. For any point $x_i^t$, the gradient oracle will return a gradient $x_i^t -z_i^t$ where $\E[z_i^t]=0$.  We apply the step decay step-size with $S =T/\log_{\alpha} T $ to $\tilde{f}_T$ using $x_1^1 = 0$ and $\eta_0=1$. Then the last iterate satisfies
\begin{align}
x_{{S}+1}^{N} = \sum_{t=1}^{N}\sum_{i=1}^{{S}} \eta_t (1-\eta_t)^{{S}-i} \Pi_{l > t}^{N}(1-\eta_l)^{{S}}z_i^t.
\end{align}
Letting $t^{\ast} = \log_{\alpha} T - \log_{\alpha}\log_{\alpha} T +1$, we have that $\eta_{t^{\ast}} = \eta_0/\alpha^{t^{\ast}-1} = \log_{\alpha}(T)/T$ and $(1-\eta_{t^{\ast}})^{{S}} = \exp(-1)$. For $t \neq t^{\ast}$ and $i \in [{S}]$, we pick $z_i^t = 0$. Then the final iterate $x_{{S}+1}^{N}$ can be estimated as
\begin{align*}
x_{{S}+1}^{N}  & = \eta_{t^{\ast}}\Pi_{l > t^{\ast}}^{N}(1-\eta_l)^{{S}}\sum_{i=1}^{{S}} (1-\eta_{t^{\ast}})^{{S}-i}z_i^t  \geq \frac{\log_{\alpha} T}{\exp(2)T}\sum_{i=1}^{{S}} (1-\eta_{t^{\ast}})^{{S}-i}z_i^{t^{\ast}}.
\end{align*}
For $\nu_i^{\ast} =(1-\eta_{t^{\ast}})^{{S}-i} $, it holds that $ \exp(-1) < \nu_i^{\ast} < 1$ for all $i\in [{S}]$. Define $ z_i^{t^{\ast}} = X_i^{\ast} / \nu_i^{\ast} $ where $X_i^{\ast}$ is uniformly chosen from $\left\lbrace -1, +1\right\rbrace$. Then $|z_i^{t^{\ast}}|\leq \exp(1) \in \mathcal{X}$ for any $i\in [{S}]$. Hence, this gradient oracle satisfies the assumptions. Now,
\begin{align}
x_{{S}+1}^{N}  \geq \frac{\log_{\alpha} T}{\exp(2)T}\sum_{i=1}^{{S}}X_i^{\ast} = \frac{1}{\exp(2)} \left(\frac{\log_{\alpha} T}{T} \sum_{i=1}^{{S}}X_i^{\ast} \right).
\end{align}
Invoking Lemma \ref{lem:proba:1} with $c = \sqrt{2\ln(1/\delta)}/3$ and $K = T/\log_{\alpha} T$ gives
\begin{align}
f(x_{{S}+1}^{N}) = \frac{1}{2}(x_{{S}+1}^{N})^2 \geq \frac{1}{2}\left(\frac{1}{\exp(2)}\frac{\sqrt{2\ln(1/\delta)}}{3\sqrt{T/\log_{\alpha} T}} \right)^2 = \frac{ \ln(1/\delta)}{9\exp(2)\ln\alpha}\cdot \frac{\ln T}{T}
\end{align}
with probability at least $\delta > 0$. Since we know  the optimal function value $\tilde{f}_{T}^{\ast}  = \tilde{f}_{T}(x^{\ast}) = 0$, the the desired high probability lower bound follows.\end{proof}

\begin{proof}(of {\bf Theorem \ref{sec:thm:last-iterate}})
Recalling the iterate updates of Algorithm \ref{alg:convex}, for any $x \in \mathcal{X}$, we have
\begin{align}
\E\left[\left\|x_{i+1}^{t} - x\right\|^2\right] &  = \E\left[\left\|\Pi_{\mathcal{X}}(x_{i}^{t} - \eta_t \hat{g}_i^t) - x\right\|^2 \right] \leq \E\left[\left\|x_{i}^{t} - \eta_t \hat{g}_i^t - x\right\|^2\right]  \notag \\
& \leq \E[\left\|x_{i}^{t}- x\right\|^2] - 2\eta_t\E[\left\langle g_i^t, x_i^t-x\right\rangle] + \eta_t^2\E\left[\left\|g_i^t\right\|^2\right] \notag \\
\label{eqn:nonsmooth:last-iterate:1} & \leq \E[\left\|x_{i}^{t}- x\right\|^2] - 2\eta_t\E[\left\langle g_i^t, x_i^t-x\right\rangle] + \eta_t^2G^2.
\end{align}
where the second inequality follows since $\left\|\Pi_{\mathcal{X}}(u) - v \right\| \leq \left\|u-v\right\|$ for any $v \in \mathcal{X}$ and the third inequality follows from the fact that the gradient oracle is bounded and unbiased.
In the following analysis, we focus on the final phase, that is $t = {N}$. Let $k$ be the integer in $\left\lbrace 0, 1, 2, \cdots, {S}-1\right\rbrace$. Extracting the inner product and summing over all
$i$ from ${S}-k$ to ${S}$ gives
\begin{align}\label{eqn:last-iterat}
\sum_{i={S}-k}^{{S}}\E\left[ \left\langle g_i^{N}, x_i^{N}-x\right\rangle\right] & \leq \frac{1}{\eta_{N}}\E\left[\left\|x_{{S}-k}^{{N}}- x\right\|^2\right] + (k+1)\eta_{N} G^2.
\end{align}
By the convexity of $f$ on $\mathcal{X}$, we have $\E[f(x) -f(x_i^{N})] \geq \E\left[\left\langle g_i^{N}, x-x_i^{N}\right\rangle\right]$. Plugging this into \eqref{eqn:last-iterat},  we get
\begin{equation}\label{inequ:lem:avg:1}
\E\left[\frac{1}{k+1}\sum_{i={S}-k}^Sf(x_i^{N}) - f(x)\right] \leq  \frac{1}{(k+1)\eta_{N}}\E\left[\left\|x_{{S}-k}^{{N}}- x\right\|^2\right] + \eta_{N} G^2.
\end{equation}
We pick  $x = x_{{S}-k}^{N}$ in \eqref{inequ:lem:avg:1} to find
\begin{equation}\label{eqn:last-iterate:core}
\frac{1}{k+1}\sum_{i={S}-k}^{S}\E[f(x_i^{N})] - \E[f(x_{{S}-k}^{N})] \leq \eta_{{N}}G^2.
\end{equation}
Let $W_{k+1} = \frac{1}{k+1}\sum_{i={S}-k}^{{S}} \E[f(x_i^{N})]$ which is the average of the expected function values at the last $k+1$ iterations of the final phase ${N}$. The above inequality implies that 
\begin{equation*}
- f(x_{{S}-k}^{N}) \leq W_{k+1} + \eta_{N}G^2.
\end{equation*}
By the definition of $W_k$, we have $W_1=\E[f(x_S^{N})]$ and $kW_{k} = (k+1) W_{k+1} - \E[f(x_{{S}-k}^{N})]$. Using these formulas and applying \eqref{eqn:last-iterate:core} gives
\begin{align*}
kW_{k} = (k+1) W_{k+1} - f(x_{{S}-k}^{N})  \leq (k+1) W_{k+1} - W_{k+1}  + \eta_{N}G^2
\end{align*}
which, after dividing by $k$, yields
\begin{align*}
W_{k} \leq W_{k+1} + \frac{\eta_{N}G^2}{k}.
\end{align*}
Applying the above inequality recursively for $k=1,\cdots, {S}-1$, we get
\begin{equation}\label{inequ:func:core1}
W_1 = \E[f(x_{S}^{N})] \leq W_{{S}} + \eta_{N}G^2\sum_{k=1}^{{S}-1} \frac{1}{k} \leq  W_{S} + \eta_{N}G^2(\ln({S}-1)+1).
\end{equation}
It only remains to estimate $W_{{S}} $. At the ${N}^{\text{th}}$ phase, the iterate starts from $x_{1}^{{N}}$. We pick $x = x^{\ast}$ and $k={S}-1$ in \eqref{inequ:lem:avg:1} so that
\begin{equation}\label{inequ:lem:avg:2}
W_{S} := \E\left[\frac{1}{{S}}\sum_{i=1}^Sf(x_i^{N})\right] \leq f(x^{\ast}) +\frac{1}{\eta_NS}\E\left[\left\|x_{1}^{{N}}- x^{\ast}\right\|^2\right] + \eta_{N} G^2.
\end{equation}
Note that in order to estimate $W_{S}$, we have to bound $\E\left[\left\|x_{1}^{{N}}- x^{\ast}\right\|^2\right]$ first. From inequality \eqref{eqn:core:1} of Theorem \ref{thm:sc}, we know that the distance between the starting point $x_1^{N}$ of the ${N}^{\text{th}}$ phase and $x^{\ast}$ can be bounded as follows:
\begin{align}\label{eqn:last:n-1}
\E\left[\left\|x_{1}^{{N}}- x^{\ast}\right\|^2\right] & \leq \exp\left(-2\mu {S}\sum_{t=1}^{{N}-1}\eta_t\right)\left\|x_1^1 - x^{\ast} \right\|^2 +  \frac{G^2}{2\mu}\sum_{t=1}^{{N}-1} \eta_t \exp\left(-2\mu {S}\sum_{l>t}^{{N}-1}\eta_{l}\right). 
\end{align}
We now follow the proof of Theorem \ref{thm:sc} to estimate $\E\left[\left\|x_{1}^{{N}}- x^{\ast}\right\|^2\right]$.  Substituting the step-size $\eta_t = \eta_0/\alpha^{t-1}$ for $t \in [{N}]$, ${N}=\log_{\alpha} T$ and $ {S}= T/\log_{\alpha} T$, we have
\begin{align*}
{S}\sum_{t=1}^{{N}-1}\eta_t  = \frac{T}{\log_{\alpha}(T)}\frac{\eta_0(1-\alpha^{-{N}+1})}{1-1/\alpha}& = \frac{\eta_0\alpha}{(\alpha-1)}\cdot\frac{T-\alpha}{\log_{\alpha}(T)} \\
    {S}\sum_{l>t}^{{N}-1}\eta_l = \frac{T}{\log_{\alpha}(T)}\frac{\eta_0\alpha^{-t}(1-\alpha^{-({N}-t-1)})}{(1-1/\alpha)} & =  \frac{\eta_0\alpha}{\alpha-1}\cdot\frac{(T\alpha^{-t}-\alpha)}{\log_{\alpha}(T)}.
\end{align*}
Therefore, using these inequalities in  \eqref{eqn:last:n-1} gives
\begin{align*}
\E\left[\left\|x_{1}^{{N}}- x^{\ast}\right\|^2\right] & \leq \exp\left(- \frac{2\eta_0\mu\alpha}{\alpha-1} \cdot\frac{T-\alpha}{\log_{\alpha} T}\right)\left\|x_1^1 - x^{\ast} \right\|^2 + \frac{G^2(\alpha-1)\exp\left(\frac{2\mu\eta_0\alpha^2}{(\alpha-1)\log_{\alpha} T}\right)}{4\mu^2} \cdot \frac{\log_{\alpha} T}{T}.
\end{align*}
Incorporating the above results and substituting $\eta_{N} = \eta_0\alpha/T$ and ${S}=T/\log_{\alpha} T$ into \eqref{inequ:func:core1}, we have
\begin{align*}
\E[f(x_{S}^{N})] - f(x^{\ast}) & \leq \frac{\E\left[\left\|x_{1}^{{N}}- x^{\ast}\right\|^2\right]}{\eta_{N} {S}} + \eta_{N} G^2 + \eta_{N} G^2(\ln({S}-1)+1) \notag \\
& \leq \frac{\left\|x_1^1 - x^{\ast} \right\|^2 \log_{\alpha}T}{\eta_0\alpha\exp\left(\frac{2\eta_0\mu\alpha}{\alpha-1} \cdot\frac{T-\alpha}{\log_{\alpha} T}\right)} + \frac{G^2(\alpha-1)\exp\left(\frac{2\mu\eta_0\alpha^2}{(\alpha-1)\log_{\alpha} T}\right)}{4\mu^2\eta_0\alpha}\cdot \frac{\log_{\alpha}^2 T}{T} \\
& \vspace{0.4em} + \frac{G^2\eta_0\alpha (\ln(T)+2)}{T}.
\end{align*}
By changing the base of $\alpha$ to be natural logarithm, i.e., $\log_{\alpha} T = \ln T/\ln \alpha$, the proof is finished.
\end{proof}

\begin{proof}(of {\bf Theorem \ref{thm:nonsmooth:average}})
Recall inequality \eqref{eqn:nonsmooth:last-iterate:1} in the proof of Theorem \ref{sec:thm:last-iterate}: for any $x \in \mathcal{X}$ it holds that 
\begin{align*}
\E[\left\|x_{i+1}^{t} - x\right\|^2] \leq \E[\left\|x_{i}^{t}- x\right\|^2] - 2\eta_t\E[\left\langle g_i^t, x_i^t-x\right\rangle] + \eta_t^2G^2.
\end{align*}
By the convexity of $f$,  we have $\E[f(x)- f(x_i^t)] \geq \E[\left\langle {g}_i^t, x-x_i^t\right\rangle]$, so the above inequality implies that 
\begin{align}\label{eqn:thm:average:1}
2\eta_t\E[f(x_i^t) - f(x)] \leq \E[\left\|x_{i}^{t}- x\right\|^2]- \E[\left\|x_{i+1}^{t} - x\right\|^2] + \eta_t^2G^2.
\end{align}
Let $ t^{\ast}= \max\left\lbrace 0,  \bigfloor{ \log_{\alpha}\left(\frac{2\mu\eta_0\alpha^2}{\alpha-1}\cdot\frac{T}{\log_{\alpha} T}\right)}\right\rbrace$ and $x =x^{\ast}$. By applying \eqref{eqn:thm:average:1} repeatedly and summing over all $t^{\ast} \leq t \leq {N}$ and $i\in [{S}]$, we have
\begin{align}\label{eqn:thm:average:2}
\sum_{t=t^{\ast}}^{{N}}\eta_t \sum_{i=1}^{S} f(x_i^t)-f(x^{\ast}) \leq \E \left[\left\|x_1^{t^{\ast}} - x^{\ast}\right\|^2\right] + {S} G^2\sum_{t=t^{\ast}}^{{N}} \eta_t^2.
\end{align}
Let $$\hat{x}_T := \frac{ \sum_{t=t^{\ast}}^{{N}}\eta_t \sum_{i=1}^{S} x_i^t}{{S}\sum_{t=t^{\ast}}^{{N}}\eta_t}.$$ Since $\mathcal{X}$ is convex and each iterate $x_i^t $ belongs to $\mathcal{X}$, we have $\hat{x}_T \in \mathcal{X}$.
By the convexity of $f$ and \eqref{eqn:thm:average:2} it then follows that
\begin{align}
\E[f(\hat{x}_{T}) -f(x^{\ast})] = \E \left[\frac{ \sum_{t=t^{\ast}}^{{N}}\eta_t \sum_{i=1}^{S} f(x_i^t)}{{S}\sum_{t=t^{\ast}}^{{N}}\eta_t}\right]-f(x^{\ast}) &  \leq  \frac{ \sum_{t=t^{\ast}}^{{N}}\eta_t \sum_{i=1}^{S} \E[f(x_i^t)]}{{S}\sum_{t=t^{\ast}}^{{N}}\eta_t} -f(x^{\ast})\notag \\
\label{eqn:averg:key}& \leq \frac{\E \left[\left\|x_1^{t^{\ast}} - x^{\ast}\right\|^2\right]}{{S}\sum_{t=t^{\ast}}^{{N}}\eta_t} + \frac{{S} G^2\sum_{t=t^{\ast}}^{{N} }\eta_t^2}{{S}\sum_{t=t^{\ast}}^{{N}}\eta_t}. 
\end{align}
Next, we turn to estimate $\E \left[\left\|x_1^{t^{\ast}} - x^{\ast}\right\|^2\right]$. By \eqref{eqn:core:1} with ${N}=t^{\ast}-1$, we have 
\begin{align}
\E\left[\left\|x_1^{t^{\ast}}-x^{\ast}\right\|^2\right]  = \E\left[\left\|x_{{S}+1}^{t^{\ast}-1} - x^{\ast}\right\|^2\right] 
& \leq \mathop{\Pi}\limits_{t=1}^{t^{\ast}-1}(1-2\mu\eta_t)^{S}\left\|x_1^1 - x^{\ast} \right\|^2 + \frac{G^2}{2\mu}\sum_{t=1}^{t^{\ast}-1} \eta_t \mathop{\Pi}\limits_{l > t} (1-2\mu\eta_{l})^{S} \notag\\
\label{eqn:averag:2} & \leq \exp\left(-2\mu {S}\sum_{t=1}^{t^{\ast}-1}\eta_t\right)\left\|x_1^1 - x^{\ast} \right\|^2 +  \frac{G^2}{2\mu}\sum_{t=1}^{t^{\ast}-1} \eta_t \exp\left(-2\mu {S}\sum_{l>t}^{t^{\ast}-1}\eta_{l}\right) . 
\end{align}
Next, we estimate the summation of $\eta_l$ from $l=t+1$ to $l = t^{\ast}-1$ \begin{align}\label{eqn:averag:1}
{S}\sum_{l > t}^{t^{\ast}-1} \eta_l = {S}\cdot\frac{\frac{\eta_0}{\alpha^{t}}(1-(1/\alpha)^{t^{\ast}-t-1})}{ (1-1/\alpha)}= \frac{T\eta_0}{\log_{\alpha} T}\left(\frac{\frac{1}{\alpha^{t-1}} - \frac{1}{\alpha^{t^{\ast}-2}}}{\alpha-1}\right)
= \frac{\eta_0 }{\alpha-1}\cdot \frac{T}{\log_{\alpha} T \alpha^{t-1}} - \frac{1}{2\mu}.
\end{align}
Incorporating \eqref{eqn:averag:1} into the second term of \eqref{eqn:averag:2} gives
\begin{align}
\sum_{t=1}^{t^{\ast}-1} \eta_t \exp\left(-2\mu {S}\sum_{l>t}^{t^{\ast}-1}\eta_{l}\right) & = \eta_0\sum_{t=1}^{t^{\ast}-1} \frac{1}{\alpha^{t-1}} \exp\left(-\frac{2\mu\eta_0}{(\alpha-1)\alpha^{t-1}} \cdot\frac{T}{\log_{\alpha} T}  + 1\right) \notag \\
& = \frac{\eta_0\alpha^2\exp(1)}{\alpha^{t^{\ast}}}\sum_{t=1}^{t^{\ast}-1} \frac{\alpha^{t^{\ast}}}{\alpha^{t+1}}\exp\left(-\frac{\alpha^{t^{\ast}}}{\alpha^{t+1}}\right) \notag \\
 \label{eqn:averag:3} & \leq \frac{2\eta_0\alpha^2}{\alpha^{t^{\ast}}}
= \frac{(\alpha-1)}{\mu}\cdot \frac{\log_{\alpha} T}{T},
\end{align}
where the inequality follows from the fact that $\int_{x=1}^{+\infty} x \exp(-x)dx \leq 2/\exp(1)$.
Letting $t=0$ in \eqref{eqn:averag:1}, we have
\begin{align}\label{eqn:averag:4}
{S}\sum_{l=1}^{t^{\ast}-1} \eta_l = 
= \frac{\eta_0\alpha}{\alpha-1}\cdot\frac{T}{\log_{\alpha} T} - \frac{1}{2\mu}.
\end{align}
Plugging \eqref{eqn:averag:3} and \eqref{eqn:averag:4} into \eqref{eqn:averag:2}, we get
\begin{align}\label{eqn:x:ast}
\E\left[\left\|x_1^{t^{\ast}}-x^{\ast}\right\|^2\right] & \leq \frac{\left\|x_1^1 - x^{\ast} \right\|^2 }{\exp\left(\frac{2\mu\eta_0\alpha}{\alpha-1}\cdot\frac{T}{\log_{\alpha} T}-1\right)}+  \frac{G^2(\alpha-1)}{2\mu^2}\cdot \frac{\log_{\alpha} T}{T}.
\end{align}
Incorporating \eqref{eqn:x:ast} into \eqref{eqn:averg:key} and using  $\frac{\alpha-1}{2\mu\alpha} \leq {S}\sum_{t=t^{\ast}}^{{N}}\eta_t  \leq \frac{1}{2\mu}$ and ${S}\sum_{t=t^{\ast}}^{{N}}\eta_t^2 \leq \frac{\log_{\alpha} T}{4\mu^2 (\alpha+1) T} $ gives
\begin{align*}
\E[f(\hat{x}_{T}) -f(x^{\ast})] & \leq \frac{\E[\left\|x_1^{t^{\ast}} - x^{\ast}\right\|^2]}{{S}\sum_{t=t^{\ast}}^{{N}}\eta_t} + \frac{{S} G^2\sum_{t=t^{\ast}}^{{N} }\eta_t^2}{{S}\sum_{t=t^{\ast}}^{{N}}\eta_t} \notag \\
& \leq \frac{2\mu\alpha \E[\left\|x_1^{t^{\ast}} - x^{\ast}\right\|^2]}{\alpha-1} + \frac{\alpha G^2}{2\mu(\alpha^2-1)}\cdot\frac{\log_{\alpha} T}{T}\notag \\
& \leq \frac{2\mu\alpha}{\alpha-1}\cdot\frac{\left\|x_1^1 - x^{\ast} \right\|^2}{\exp\left(\frac{2\mu\eta_0\alpha}{\alpha-1}\cdot\frac{T}{\log_{\alpha} T}-1\right)}+  \frac{\alpha\left(1+\frac{1}{2(\alpha^2-1)}\right)G^2}{\mu}\cdot \frac{\log_{\alpha} T}{T}
\end{align*}
which concludes the proof.
\end{proof}

\section*{{\bf D. The Details of the Setup in Numerical Experiments}}

In this section, we provide some details for the numerical experiments in Section \ref{sec:numerical} and give some complementary experimental results. 

To better understand the relationship between all the considered step-sizes, we draw Figure \ref{fig:lr} to show the step-size $\eta_t$ ($y$-axis is $\log(\eta_t)$) versus the number of iterations (starting from the same initial step-size). In the left picture, we show the many step-sizes, studied in Sections \ref{experiment:mnist} and \ref{experiment:cifar}, which finally reach the order of $1/\sqrt{T}$ for the nonconvex and convex cases.  In the strongly convex case (the right picture), we show the step-sizes which are based on the order of $1/T$. We also add yet another kind of exponentially decaying step-size (called Exp(H-K-2014)) proposed by \citet{hazan2014beyond}: $\eta_i = \eta_0/2^i$,  $t \in [T_i, T_{i+1})$ and $T_{i+1} = 2T_i$, where $\sum_i T_i = T$.
\begin{figure}[ht]
 \vskip 0.1in
\begin{center}
\centerline{
\subfigure[nonconvex \& convex cases]{\label{nonconvex-stepsize} \includegraphics[width=0.3\textwidth,height=1.6in]{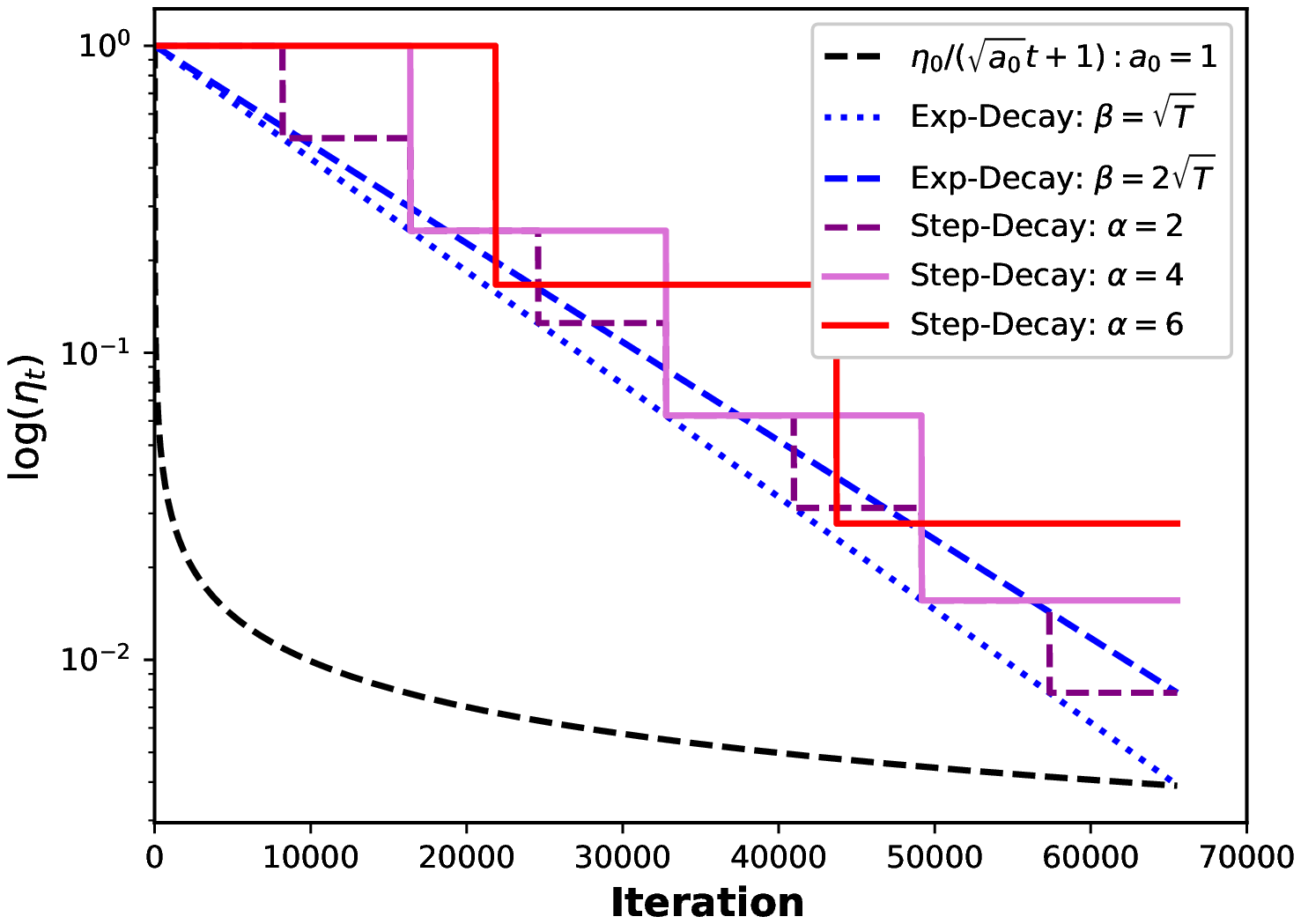} }
\subfigure[strongly convex case]{\label{sc-stepsize} \includegraphics[width=0.3\textwidth,height=1.6in]{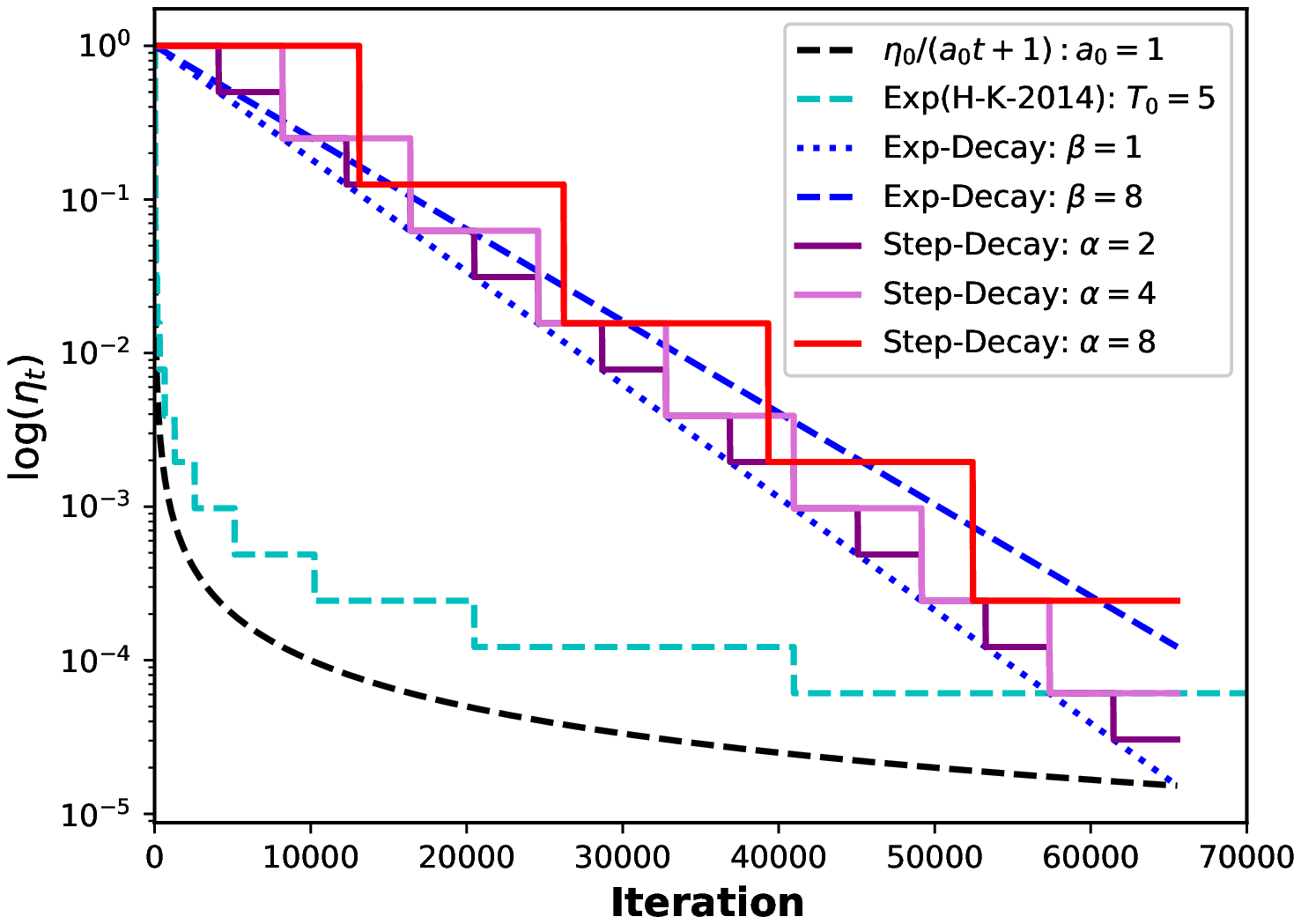}}
}
\caption{Step-sizes involved in the experiments}
\label{fig:lr} 
\end{center}
\vskip -0.1in
\end{figure}

From Figure \ref{nonconvex-stepsize}, we observe that Exp-Decay for $\beta = \sqrt{T}$ can be regarded as a lower bound of Step-Decay. From another viewpoint, we can see that when the decay factor $\alpha$ is very close  to 1, the proposed Step-Decay will reduce to Exp-Decay for $\beta = \sqrt{T}$. A similar relationship can also be observed from Figure \ref{sc-stepsize}.

\subsection*{D.1 The Details of the Experiments on  MNIST}

The MNIST dataset consists of a training set of 60,000 examples and a testing set of 10,000 examples. We train MNIST on a fully connected two-layer network (784-100-10). The $l_2$ regularization parameter is $10^{-4}$ and the mini-batch size is 128. We run 128 epochs, which implies that the number of iterations $T$ is equal to the training size $(60,000)$. 

In order to fairly compare the considered step-sizes, the initial step-size $\eta_0$ is chosen from the search grid  $ \left\lbrace 0.001, 0.005, 0.01, 0.05, 0.1, 0.5, 1, 5\right\rbrace$. For $1/t$ and $1/\sqrt{t}$, the initial step-size is $\eta_0 = 1$, and $a_0$ is tuned by searching for the final step-size $\eta_{T}$  over the grid $\left\lbrace 0.001, 0.005, 0.01, 0.05, 0.1, 0.5, 1, 5\right\rbrace$; in our experiments, it turned out that the best value of $\eta_T$ was 0.01. The best initial step-size $\eta_0$ was found to be $0.5$ for both Exp-Decay and Step-Decay. Similarly, the parameter $\beta$ for Exp-Decay is selected to make sure that its final step-size $\eta_{T}$ is tuned over the grid $\left\lbrace 0.001, 0.005, 0.01, 0.05, 0.1, 0.5, 1, 5\right\rbrace$; the best tuning of $\eta_T$ was found to be 0.05. For Step-Decay, the decay factor $\alpha $ is empirically chosen from an interval $(1, 12]$ and the search 
grid is in units of 1 after $\alpha \geq 2$. The outer-loop size $N$ is $\lfloor{\log_{\alpha}T/2}\rfloor$ which is numerically better than its ceil. The best choice was found to be $\alpha=7$ ($N=2$).  

\subsection*{\textbf{D.2 The Details of the Experiments on CIFAR10 and CIFAR100} }

The benchmark datasets CIFAR10 and CIFAR100 both consist of 60000 colour images (50000 training images and the rest 10000 images for testing). The maximum epochs called for the two datasets is 164 and batch size is $128$. 

First, we employ a 20-layer Resident Network model \citep{he2016deep}  called ResNet20 to train CIFAR10. We use vanilla SGD without dampening and a weight-decay of 0.0005. The hyper-parameters are selected to work best according to their performance on the test dataset. 

 For all evaluated step-sizes, the initial step-size $\eta_0\in\left\lbrace 0.0001, 0.0005, 0.001, 0.005, 0.01, 0.05, 0.1, 0.5, 1 \right\rbrace$. For the constant step-size, the best choice was achieved by $\eta_t=0.05$. For the $1/t$ step-size, the initial step-size $\eta_0 = 1$ and the parameter $a_0$ is tuned such that $\eta_{T}$ reaches the search grid $ \left\lbrace 0.0001, 0.0005, 0.001, 0.05, 0.01, 0.5, 0.1 \right\rbrace$ (the grid search yielded $\eta_{T} = 0.05$). We tuned the $1/\sqrt{t}$ step-size in the same way as the $1/t$ step-size: the initial step-size  $\eta_0 = 1$ and $\eta_T=0.05$. For Exp-Decay, $\eta_0 = 1$ and $\beta$ is chosen such that the final step-size $\eta_T$ reaches the search grid for step-size (resulting in $\eta_{T} = 0.01$). For Step-Decay, the best initial step-size is achieved at $\eta_0 = 0.5$ and $\alpha = 6$.

The numerical results on CIFAR10 is shown in Figure \ref{fig:stepdecay:cifar10}. 
We can see that the sudden jumps in step-size helps the algorithm to get a lower testing loss and higher accuracy (red curve) compared to other step-sizes. 
\begin{figure}[ht]
\vskip 0.1in
\begin{center}
\centerline{
\subfigure[Training loss]{\includegraphics[width=0.3\textwidth,height=1.8in]{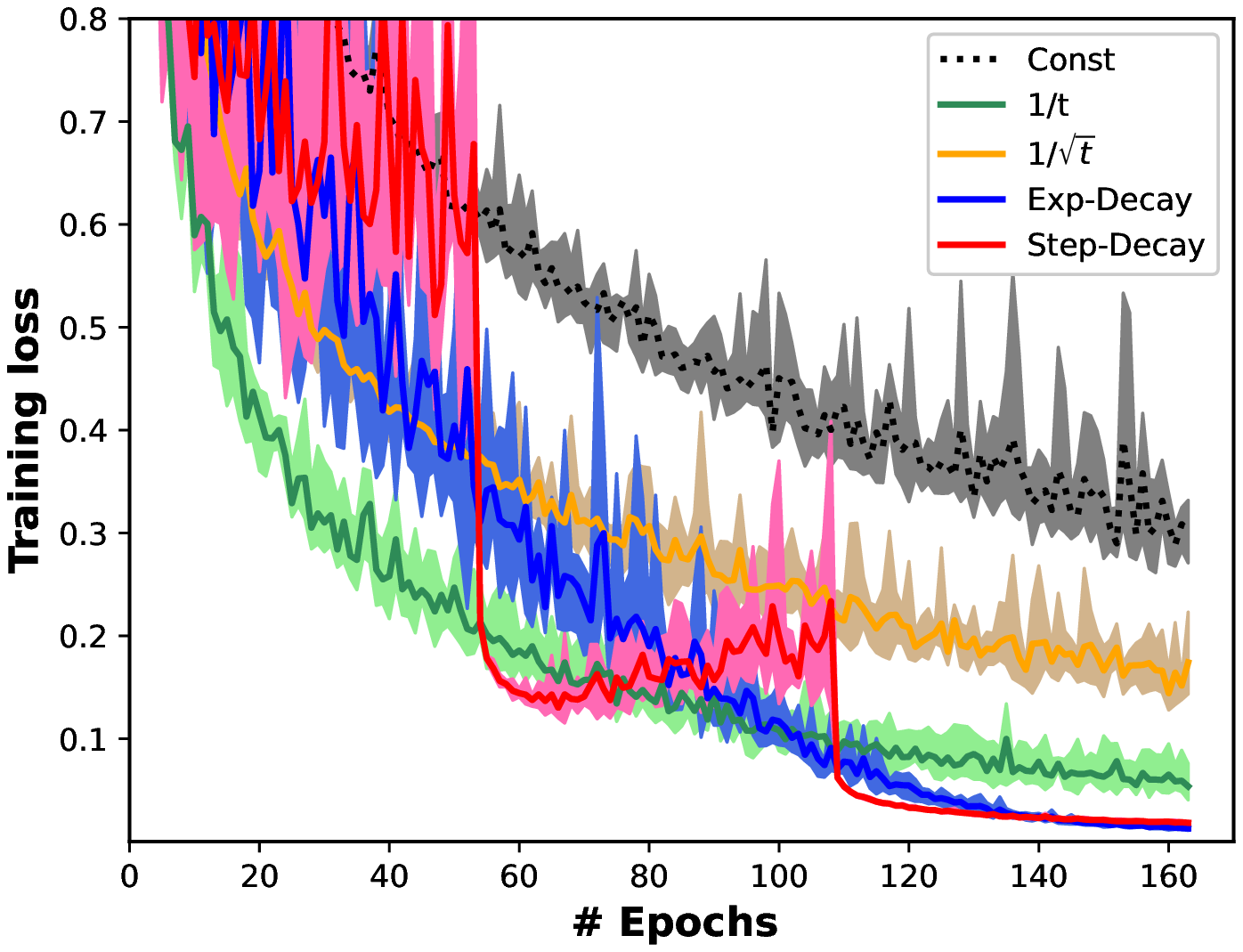} }
\subfigure[Testing loss]{\includegraphics[width=0.3\textwidth,height=1.8in]{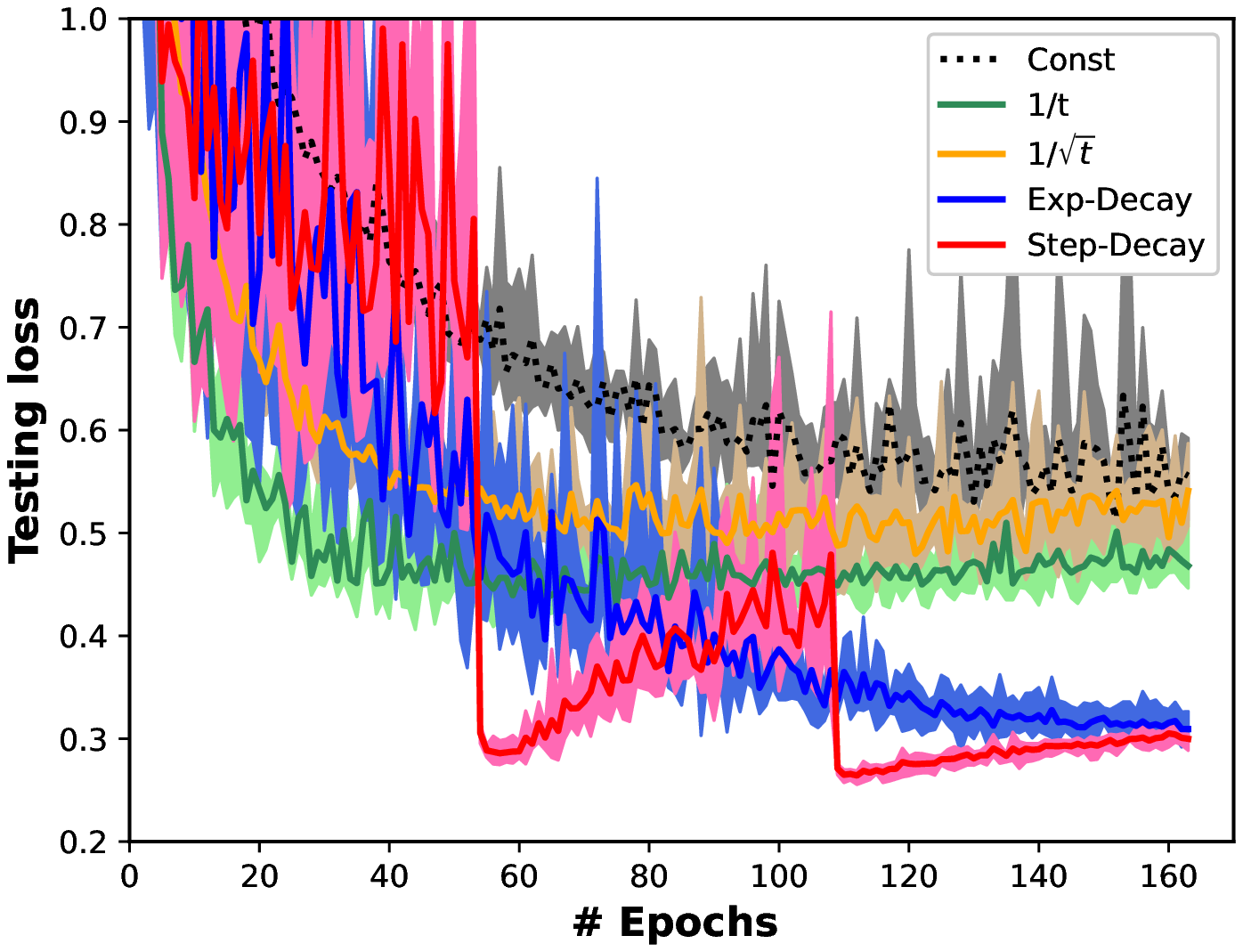} }
\subfigure[Testing accuracy]{\includegraphics[width=0.3\textwidth,height=1.6in]{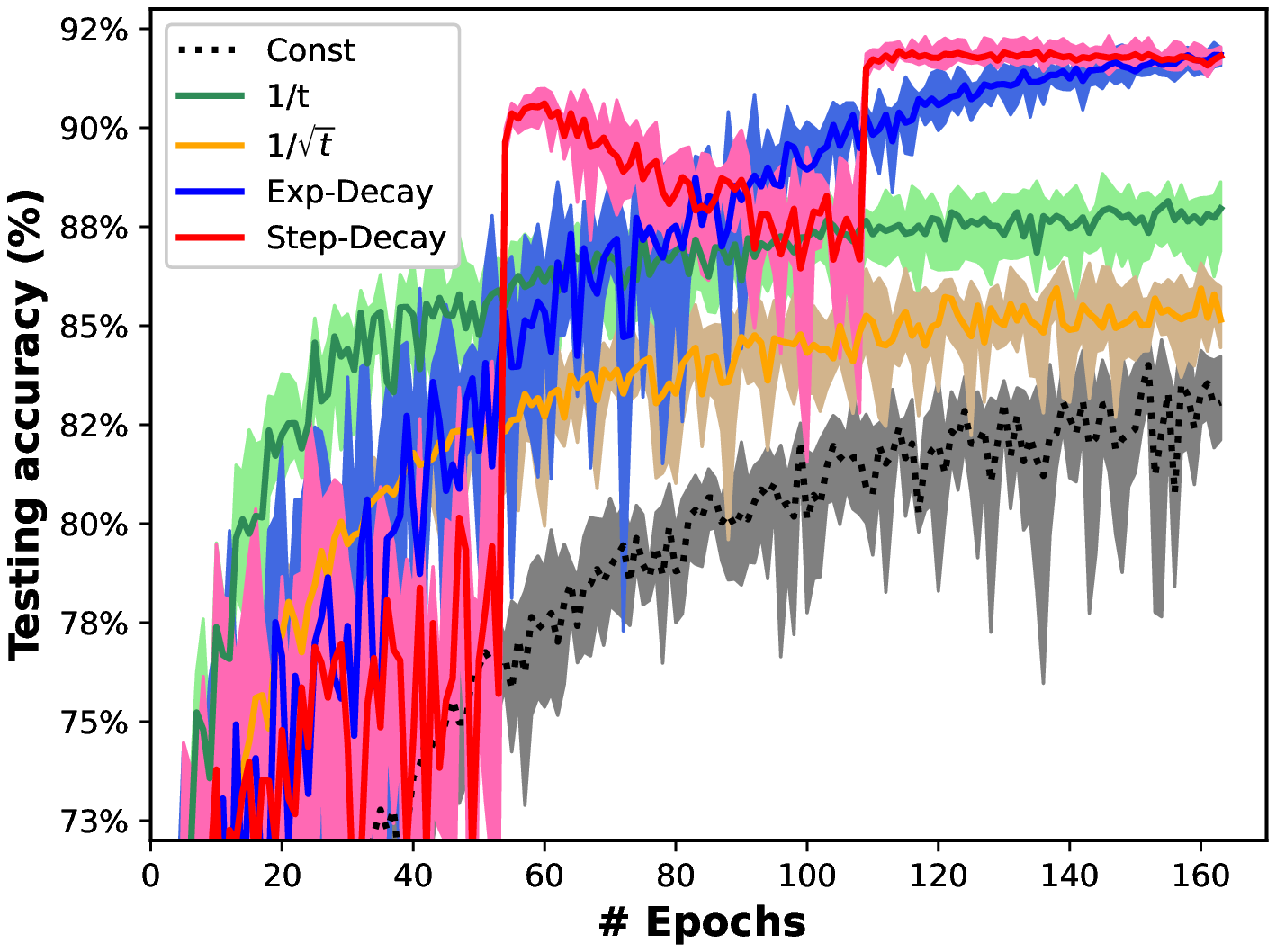}  }} 
\caption{Results on CIFAR10 - ResNet20}
\label{fig:stepdecay:cifar10}  
\end{center}
\vskip -0.2in
\end{figure}

Next, we give the details about how to select the optimal values of the parameters$\eta_0$, $a_0$, $\beta$ and $\alpha$ for CIFAR100. The initial step-size  is chosen from $ \left\lbrace 0.0001, 0.0005, 0.001, 0.005, 0.01, 0.05, 0.1, 0.5, 1 \right\rbrace$ for all step-sizes. 
For the constant step-size, we found $\eta_0=0.1$ and a weight-decay of 0.0001. For the $1/t$ step-size, $\eta_0 = 1$ and $a_0$ was set such that $\eta_{T}$ is searched over the grid above, which yielded $\eta_{T} = 0.01$. For $1/\sqrt{t}$, $\eta_0 = 1$ and $a_0$ is set to make sure that $\eta_{T}$ is tuned from the set for step-size, resulting in $\eta_T=0.01$. 
For Exp-Decay, $\eta_0 = 1$ and $\beta$ is chosen such that $\eta_T$ reaches the grid $\left\lbrace 0.0001, 0.0005, 0.001, 0.05, 0.01, 0.5, 0.1 \right\rbrace$ (resulting in $\eta_{T} = 0.01$). For the Step-Decay, the initial step-size $\eta_0 = 1$ and the decay factor $\alpha = 6$.
 
Next, we detail the parameters tuning for the algorithms in Table \ref{tab:cifar100:adap-grad}. The maximum epochs called was 164 and the batch size was 128. This implies that $T = 164\cdot T/128$. The weight-decay was set to 0.0005 for Adam and NAG. The best initial step-size for AdaGrad was found to be $\eta_0=0.05$ (weight-decay is 0); For Adam, we found the parameters ($\beta_1, \beta_2$)=(0.9, 0.99). The weight-decay was set to 0.025 for AdamW while the other parameters were the same as for Adam.  For NAG, the momentum parameter was set to 0.9. For Exp-Decay: the best-tuned $\beta$ was $0.005\cdot T$ and $\eta_0=0.1$ for NAG; $\beta=0.01 \cdot T$ and $\eta_0=0.005$ for Adam and AdamW; For Step-Decay: the optimal $\alpha$ was found to be 6 for all methods including Adam, AdamW and NAG while the best $\eta_0=0.05$ for NAG and $\eta_0=0.005$ for Adam and AdamW.

In Figure \ref{fig:cifar100:outer-loop}, we show how the number of outer-loop iterations $N$ changes with the decay factor $\alpha \in (1,12]$. The decay factor is an important hyper-parameter for Step-Decay. To figure out the decay factor affects the performance, we plot the testing loss and generalization error (the absolute value of the difference between training loss and testing loss), as well as the testing accuracy in Figures \ref{fig:cifar100:loss} and \ref{fig:cifar100:accuracy}, respectively. All results are repeated 5 times. The best choice of the decay factor is found to be $\alpha=6$, according to the best performance on testing loss and accuracy. It is observed that $\alpha \in [4, 6]$ performs better and is more stable than $\alpha \in [7, 12)$. The main reason is that the length of each phase for $\alpha \in [7, 12)$ is larger than that of $\alpha \in [4,6]$ so that it loses its advantages in the end (the generalization is weakened). Moreover, suppose that the number of outer-loop iterations $N$ is fixed, for example at $N=3$ (where $\alpha \in [4,6]$) or $N=2$ (where $\alpha \in [7,12]$), we can see that the testing loss is getting better if we increase the decay factor.
\begin{figure}[ht]
  \centering
   \vskip 0.1in
  \subfigure[\# Outer-loop]{ \label{fig:cifar100:outer-loop} \includegraphics[width=0.3\textwidth, height=1.8in]{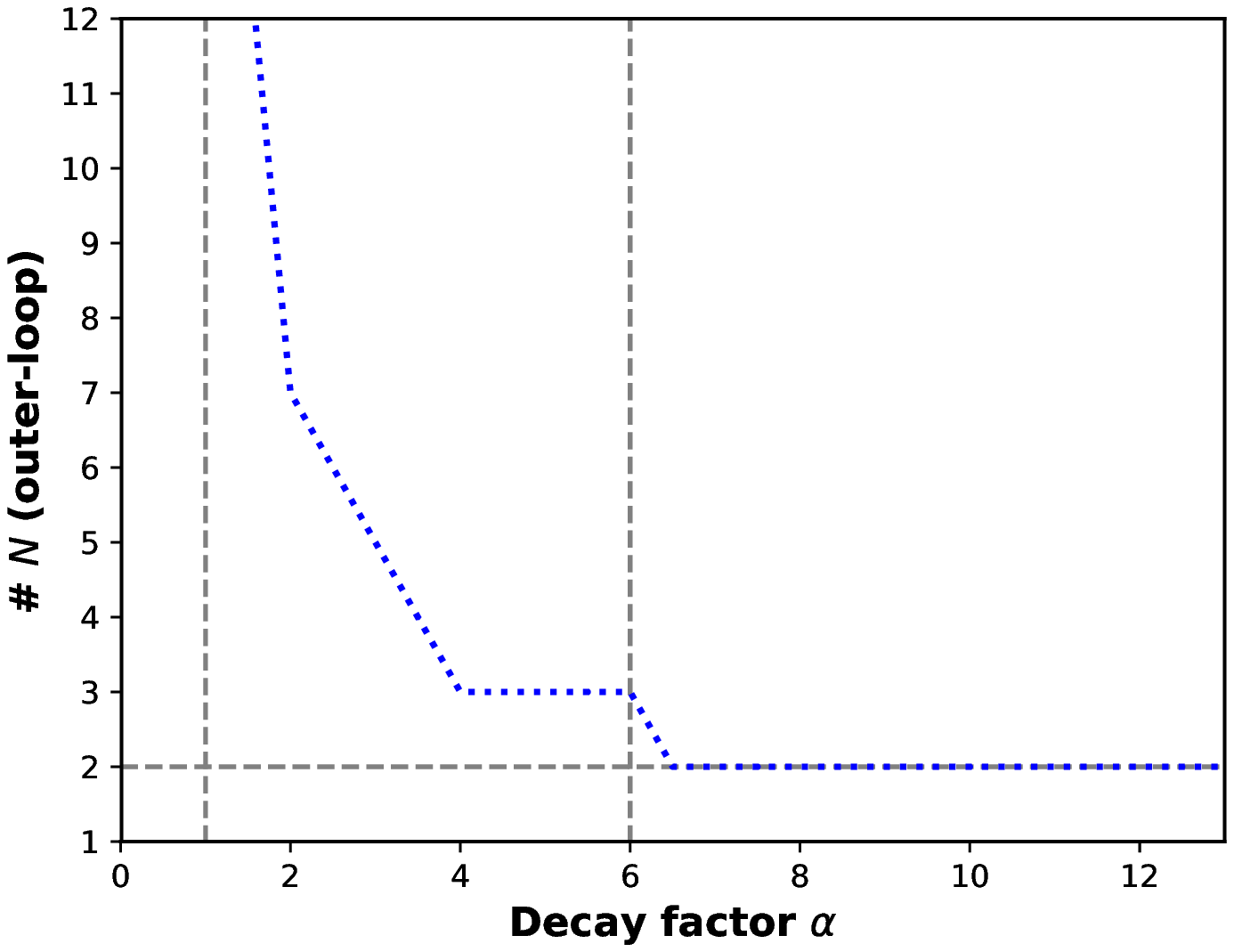}}
\subfigure[Loss]{ \label{fig:cifar100:loss} \includegraphics[width=0.3\textwidth,height=1.8in]{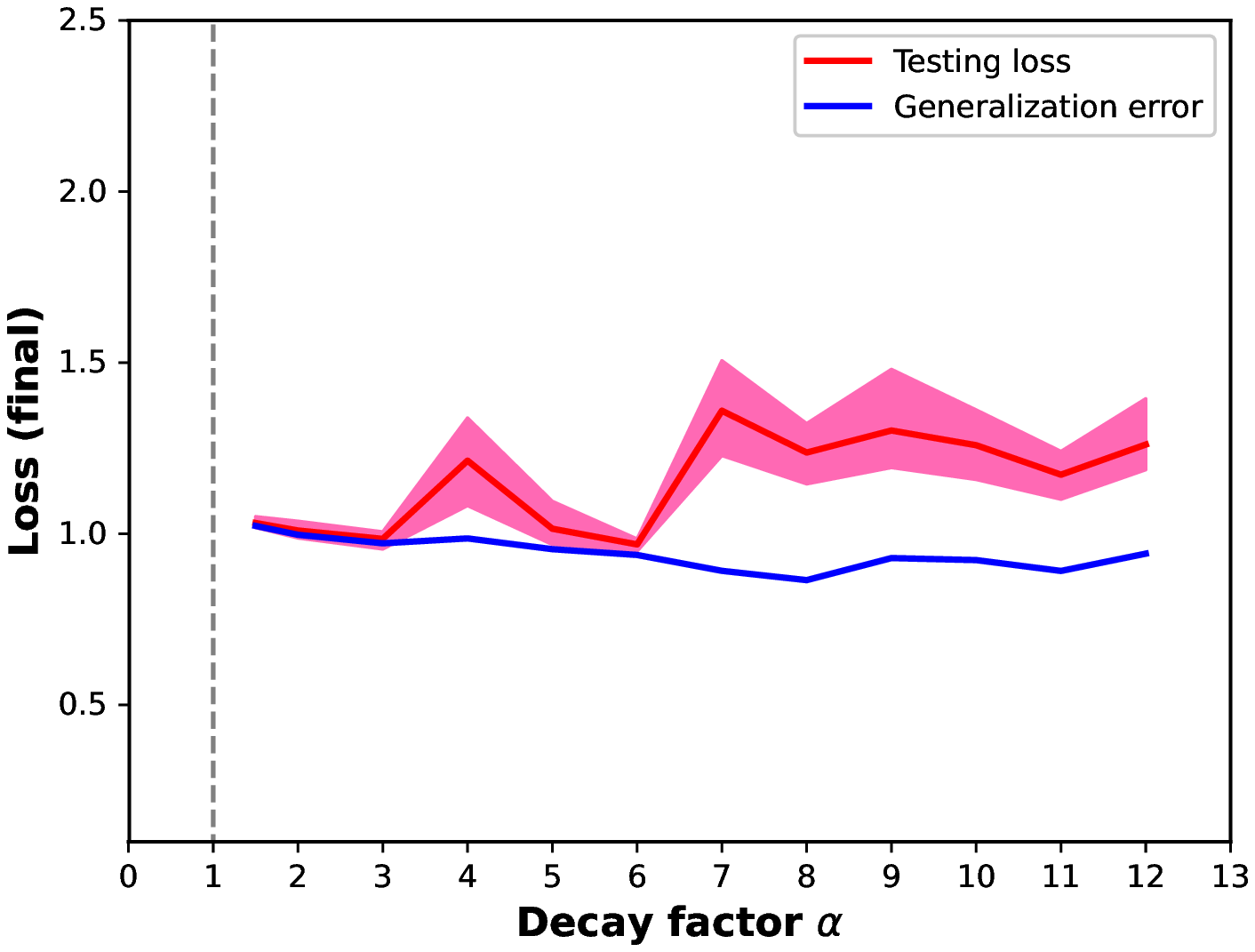}}
  \subfigure[Accuracy]{\label{fig:cifar100:accuracy} \includegraphics[width=0.3\textwidth,height=1.8in]{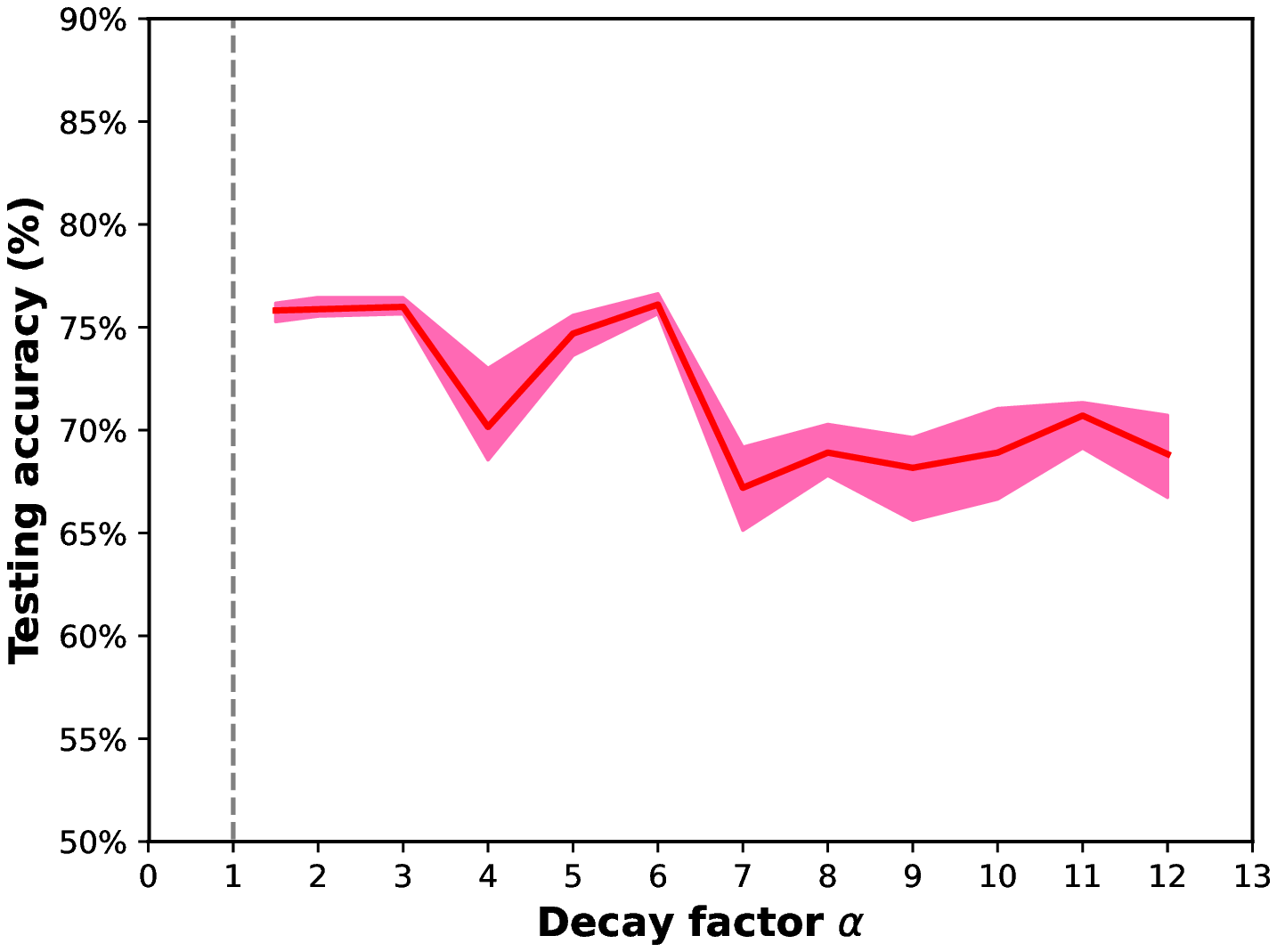}}
   \caption{The performance of decay factor $\alpha$ on CIFAR100}
\label{fig:stepdecay:cifar100:rate} 
 \vskip -0.2in
\end{figure}

\subsection*{\textbf{D.3 Numerical Details for Regularized Logistic Regression}}

In this part, we present the numerical results for all considered step-sizes on regularized logistic regression. The initial step-size $\eta_0 $ is best-tuned from the search grid $\left\lbrace 0.01, 0.05, 0.1, 0.5, 1, 10, 50, 100 \right\rbrace$ for all step-sizes.
For the constant step-size, the initial step-size is $\eta = 1$. For $1/t$, $1/\sqrt{t}$, Exp(H-K-2014), Exp-Decay \citep{li2020exponential} and Step-Decay, the initial step-size $\eta_0$ is 10. We tune $a_0$ for $1/t$ and $1/\sqrt{t}$ step-sizes such that the final step-size $\eta_T $ is searched over the grid $\left\lbrace 0.01, 0.05, 0.1, 0.5, 1, 10, 25, 50, 75, 100 \right\rbrace$ ($\eta_T=0.1$). Similarly, the parameter $\beta$ of Exp-Decay \citep{li2020exponential} is chosen such that $\eta_T$ is searched over the grid $\left\lbrace 0.01, 0.05, 0.1, 0.5, 1, 10, 50, 100 \right\rbrace$ ($\eta_{T}=0.01$).  The initial period $T_0$ for Exp(H-K-2014) is $T_0=5$. For Step-Decay, the decay factor is chosen to be $\alpha=4$.

Compared to the polynomially diminishing step-sizes (e.g. $1/t$, $1/\sqrt{t}$), we can observe that Exp-Decay and Step-Decay not only yields rapid improvements initially, but they also converge to a good solution in the end. 
\begin{figure}[ht!]
\begin{center}
\centerline{\subfigure[Training loss]{ \label{fig:rcv:loss} \includegraphics[width=0.3\textwidth,height=1.6in]{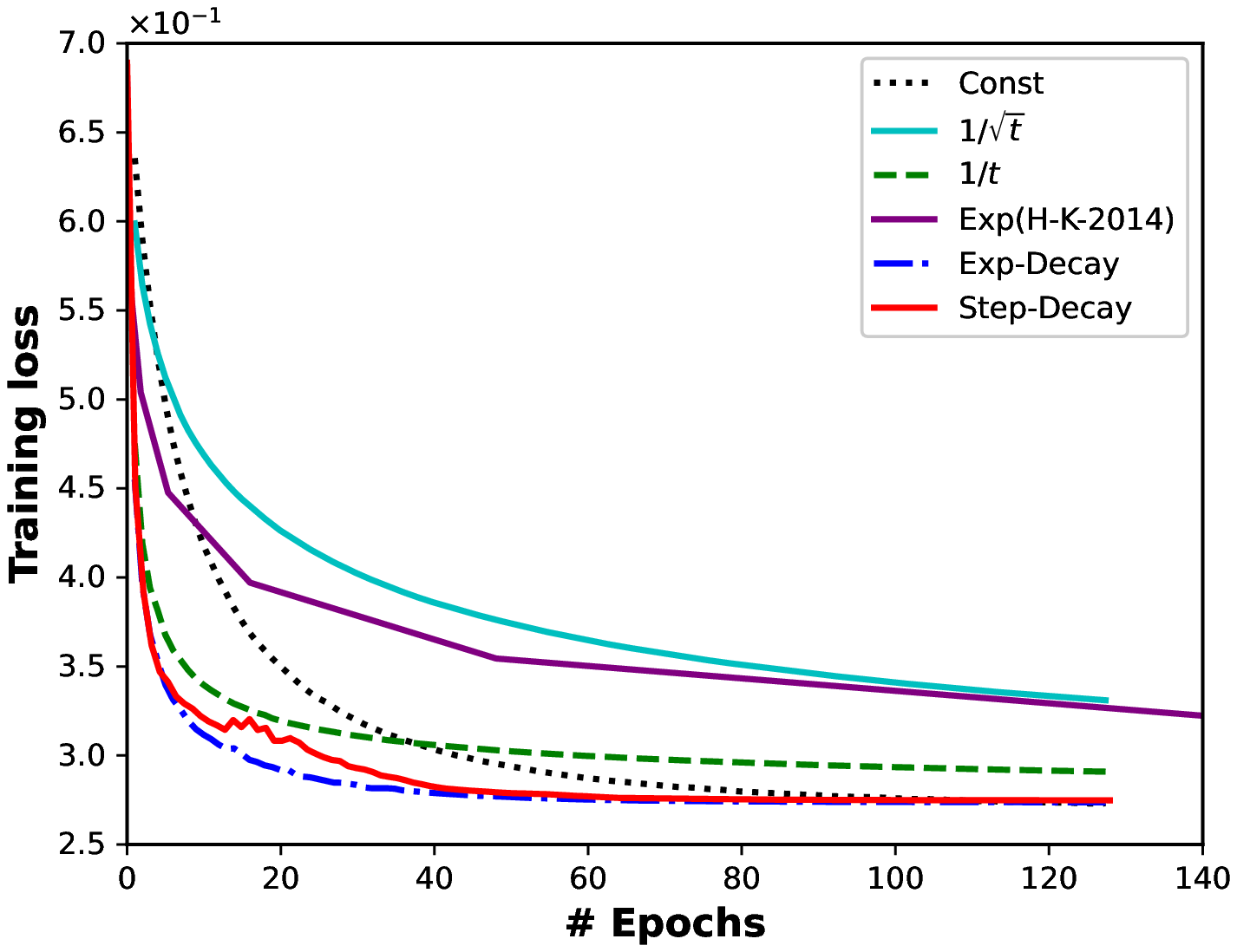} }
   \subfigure[Testing accuracy]{ \label{fig:rcv:accuracy} \includegraphics[width=0.3\textwidth,height=1.6in]{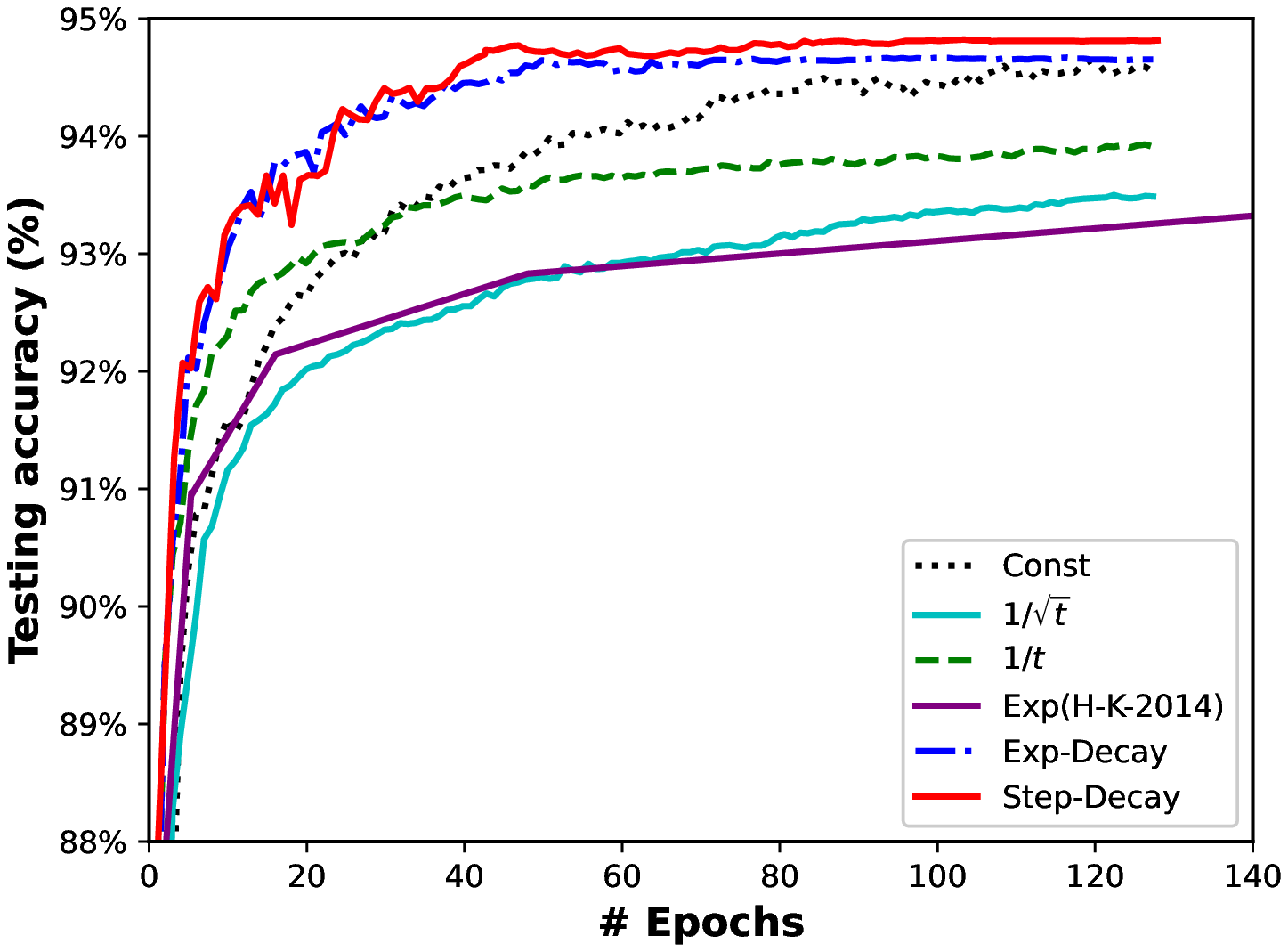} }
}
\caption{The results on rcv1.binary - logistic($L_2$)}
\label{fig:stepdecay:rcv} 
\end{center}
\vskip -0.2in
 \end{figure}

\end{document}